\documentclass[a4paper,11pt]{amsproc}
\usepackage{amsmath,amsthm,amssymb,amscd}
\usepackage{mathtext}
\usepackage{euscript}

\usepackage[matrix,arrow,curve]{xy}
\usepackage{graphicx}
\input{xypic}
\DeclareGraphicsExtensions{.jpg}

\newtheorem{theorem}{Theorem}[section]
\newtheorem{proposition}[theorem]{Proposition}
\newtheorem{corollary}[theorem]{Corollary}
\newtheorem{lemma}[theorem]{Lemma}
\newtheorem{question}[theorem]{Question}
\newtheorem*{theorem*}{Theorem}
\theoremstyle{definition}

\newtheorem{example}[theorem]{Example}
\newtheorem{construction}[theorem]{Construction}
\theoremstyle{remark}
\newtheorem*{remark}{Remark}

\numberwithin{equation}{section}

\newcommand{\Ext}{\mathop{\mathrm{Ext}}\nolimits}
\newcommand{\Hom}{\mathop{\mathrm{Hom}}\nolimits}
\def\id{\mathop{\mathrm{id}}}
\def\Im{\mathop{\rm Im}}
\newcommand{\Ker}{\mathop{\rm Ker}}
\newcommand{\rank}{\mathop{\mathrm{rank}}}

\def\C{\mathbb C}
\def\Q{\mathbb Q}
\def\R{\mathbb R}
\def\Z{\mathbb Z}

\newcommand{\OSU}{\varOmega^{SU}}
\newcommand{\OU}{\varOmega^U}
\def\MU{\mathit{MU}}
\def\MSU{\mathit{MSU}}
\def\BU{\mathit{BU}}
\def\BSU{\mathit{BSU}}
\def\SU{\mathit{SU}}
\def\cf{c^{\scriptscriptstyle U}}

\def\du{d^{\scriptscriptstyle U}}
\def\bdu{\overline{d}\phantom{d}\!\!\!^{\scriptscriptstyle U}}

\def\ge{\geqslant}

\def\le{\leqslant}

\newcommand{\llra}{\relbar\joinrel\hspace{-1pt}\longrightarrow}
\newcommand{\lllra}{\relbar\joinrel\hspace{-1pt}\llra}

\newcommand{\bin}[2]{{\textstyle\binom{#1}{#2}}} 
\newcommand{\mb}[1]{{\textbf {\textit#1}}}

\def\pd{ p_{{}_{\!\varDelta}} }

\newcommand{\cs}{\mathbin{\#}}
\def\pt{\mathit{pt}}

\begin{document}

\title[$SU$-bordism]{$SU$-bordism: structure results and geometric representatives}

\author{Georgy Chernykh}
\address{Department of Mathematics and Mechanics, Moscow
State University, Leninskie Gory, 119991 Moscow, Russia}
\email{aaa057721@gmail.com}

\author{Ivan Limonchenko}
\address{National Research University ``Higher School of Economics'', Moscow, Russia}
\email{iylim@mail.ru}

\author{Taras Panov}
\address{Department of Mathematics and Mechanics, Moscow
State University, Leninskie Gory, 119991 Moscow, Russia}
\address{Institute for Information Transmission Problems of the Russian Academy of Sciences, Moscow}
\address{Institute of Theoretical and Experimental Physics, Moscow}
\email{tpanov@mech.math.msu.su}

\thanks{The first and third authors were partially supported by the Russian Foundation for Basic Research (grants no.~17-01-00671, 18-51-50005).
The research of the second author was carried out within the HSE University Basic Research Program 
and funded by the Russian Academic Excellence Project '5-100'.
The third author was also supported by the Simons Foundation.}



\begin{abstract}
In the first part of this survey we give a modernised exposition of the structure of the special unitary bordism ring, by combining the classical geometric methods of Conner--Floyd, Wall and Stong with the Adams--Novikov spectral sequence and formal group law techniques that emerged after the fundamental 1967 work of Novikov. In the second part we use toric topology to describe geometric representatives in $SU$-bordism classes, including toric, quasitoric and Calabi--Yau manifolds.
\end{abstract}

\dedicatory{Dedicated to Sergei Petrovich Novikov on the occasion of his 80th birthday}

\maketitle

\tableofcontents

\section*{Introduction}
$SU$-bordism is the bordism theory of smooth manifolds with a special unitary structure in the stable tangent bundle. Geometrically, an $SU$-structure on a manifold $M$ is defined by a reduction of the structure group of the stable tangent bundle of~$M$ to the group $SU(N)$. Homotopically, an $SU$-structure is the homotopy class of a lift of the map $M\to BO(2N)$ classifying the stable tangent bundle to a map $M\to BSU(N)$. A manifold $M$ admits an $SU$-structure whenever it admits a stably complex structure with $c_1(\mathcal TM)=0$.

The theory of bordism and cobordism experienced a spectacular development in the beginning of the 1960s. Most leading topologists of the time contributed to this development. The idea of bordism was first explicitly formulated by Pontryagin~\cite{pont55} who related the theory of framed bordism to the stable homotopy groups of spheres. In the early works such as Rokhlin~\cite{rohl59} bordism was called ``intrinsic homology'', referring to  Poincar\'e's original idea of homological cycles. The most basic of bordism theories, unoriented bordism, was the subject of the fundamental work of Thom~\cite{thom54}, who calculated the unoriented bordism ring~$\varOmega^O$ completely. The description of the oriented bordism ring $\varOmega^{SO}$ was completed by the end of the 1950s in the works of Novikov~\cite{novi60,novi62} (the ring structure modulo torsion) and Wall~\cite{wall60} (products of torsion elements), with important earlier contribution made by Thom~\cite{thom54} (description of the ring $\varOmega^{SO}\otimes\Q$), Averbuch~\cite{aver59} (absence of odd torsion), Milnor~\cite{miln60} (the additive structure modulo torsion) and Rokhlin~\cite{rohl59}.

The theory culminated in the calculation of the complex (or unitary) bordism ring $\OU$ in the works of Milnor~\cite{miln60} and Novikov~\cite{novi60,novi62}. The ring $\OU$ was shown to be isomorphic to a graded integral polynomial ring $\Z[a_i\colon i\ge1]$ on infinitely many generators, with one generator in every even degree, $\deg a_i=2i$. This result has since found numerous applications in algebraic topology and beyond. We review the unitary bordism theory in Section~\ref{combor}, since it is instrumental in the subsequent description of the structure of the $SU$-bordism ring.

The study of $SU$-bordism in the 1960s outlined the limits of applicability of methods of algebraic topology. The coefficient ring $\OSU$ is considered to be known. It is not a polynomial ring, although it becomes so after inverting~$2$. The main contributors here are Novikov~\cite{novi62} (description of the ring $\OSU\otimes\Z[\frac12]$), Conner and Floyd~\cite{co-fl66m} (products of torsion elements), Wall~\cite{wall66} and Stong~\cite{ston68} (the multiplicative structure of $\OSU/\mathrm{Tors}$).
Nevertheless, as noted by Stong~\cite[p.~266]{ston68}, ``an intrinsic description of $\OSU/\mathrm{Tors}$ is extremely complicated''. The best known description of the ring $\OSU/\mathrm{Tors}$ is a subtly embedded subring in the polynomial ring $\mathcal W$, the coefficient ring of Conner--Floyd's theory of $c_1$-spherical manifolds (see the details in Section~\ref{Wsect}).

The Adams--Novikov spectral sequence and formal group law techniques brought in topology by the fundamental work of Novikov~\cite{novi67} led to a new systematic approach to earlier geometric calculations with the $SU$-bordism ring. In particular, the exact sequence of Conner and Floyd~\eqref{CF5seq} relating the graded components of the rings $\OSU$ and~$\mathcal W$ admits an intrinsic description in terms of nontrivial differentials in the Adams--Novikov spectral sequence for the $MSU$ spectrum (see Section~\ref{calcsect}).
This approach was further developed in the context of bordism of manifolds with singularities in the works of Mironov~\cite{miro75}, Botvinnik~\cite{botv92} and Vershinin~\cite{vers93}. The main purpose was to describe the coefficient ring $\varOmega^{Sp}$ of the next classical bordism theory, symplectic bordism (nowadays also known as quaternionic bordism), which still remains unknown and mysterious. See~\cite[\S3]{buch75} for an account of results on $\varOmega^{Sp}$ known by 1975. The Adams--Novikov spectral sequence has also become the main computational tool for the stable homotopy groups of spheres~\cite{rave86}.

\medskip

There is also the classical problem of finding geometric representatives of bordism classes in different bordism theories, in particular, for the unitary and special unitary bordism rings. The importance of this problem was emphasised in the original works such as Conner and Floyd~\cite{co-fl66m}.

Over the rationals, the bordism rings are generated by projective spaces, but the integral generators are more subtle as they involve divisibility conditions on characteristic numbers. One of
the few general results on geometric representatives for bordism classes known from the early 1960s is that the complex bordism ring $\OU$, which is an integral polynomial ring, can be generated
by the so-called Milnor hypersurfaces $H(n_1,n_2)$. These
are hyperplane sections of the Segre embeddings of products $\C
P^{n_1}\times\C P^{n_2}$ of complex projective spaces. Similar
generators exist for unoriented and oriented bordism rings.

The early progress was impeded by the lack of examples of
higher-dimensional (stably) complex manifolds for which the
characteristic numbers can be calculated explicitly. With the
appearance of toric varieties in the late 1970s and
subsequent development of toric topology in the beginning of this century~\cite{bu-pa15}, a host of
explicitly constructed concrete examples of stably complex and $SU$-manifolds  with a large torus symmetry has been produced. The characteristic numbers of these manifolds can be
calculated effectively using combinatorial-geometric techniques. These developments enriched bordism and cobordism theory with new geometric methods. 

In~\cite{bu-ra98}, Buchstaber and Ray constructed a set of
generators for $\OU$ consisting entirely of complex projective
toric manifolds $B(n_1,n_2)$, which are projectivisations of sums
of line bundles over the bounded flag manifolds. Another toric family $\{L(n_1,n_2)\}$ with the same property is presented in Section~\ref{toricsect}. Characteristic numbers of toric manifolds satisfy quite
restrictive conditions (e.\,g. their Todd genus is always~$1$) which
prevent the existence of a toric representative in every bordism
class; quasitoric manifolds enjoy more flexibility. Wilfong~\cite{wilf14} identified low-dimensional complex bordism classes which contain projective toric manifolds (there is a full description in dimensions up to~$6$, and partial results in dimension~$8$). Furthermore, by a result of Solomadin and Ustinovskiy~\cite{so-us16}, polynomial generators of the ring $\OU$ can be chosen among projective toric manifolds (a partial result of this sort was obtained earlier in~\cite{wilf16}).
Quasitoric manifolds enjoy more flexibility:
it was shownby Buchstaber, Panov and Ray~\cite{b-p-r07} that one can get a geometric representative in \emph{every} complex bordism class if toric manifolds are relaxed
to quasitoric ones; the latter still have a large torus
action, but are only stably complex instead of being complex.
In part~II of this survey we review similar results in the context of $SU$-bordism.

A renewed interest in $SU$-manifolds has been stimulated by the study of mirror symmetry and other geometric constructions motivated by theoretical physics; the notion of a Calabi--Yau manifold plays a central role here. By a Calabi--Yau manifold one usually understands a K\"ahler
$SU$-manifold; it has a Ricci flat metric by a theorem of Yau. The relationship between Calabi--Yau manifolds and $SU$-bordism is discussed in Sections~\ref{SUtoricsect}--\ref{lowdimsect} of this survey.

\medskip

Part~I contains the structure results on the $SU$-bordism ring~$\OSU$. We combine geometric methods of Conner--Floyd, Wall and Stong with the Adams--Novikov spectral sequence and formal group law techniques in this description.

\smallskip

Section~\ref{combor} is a summary of complex bordism theory. By a theorem of Milnor and Novikov, 
\[
  \OU\cong\mathbb{Z}[a_{i}\colon i\ge 1],\quad \deg{a_i}=2i,
\]
and two stably complex manifolds are bordant if and only if they have identical
Chern characteristic numbers. Polynomial generators are detected by a special characteristic number $s_i$ (sometimes called the Milnor number). For any integer $i\ge 1$, set 
\[
  m_{i}=\begin{cases}
  1&\text{if $i+1\neq p^k$ for any prime $p$;}\\
  p&\text{if $i+1=p^k$ for some prime $p$ and integer $k>0$.}
  \end{cases}
\]
Then the bordism class of a stably complex manifold $M^{2i}$ may be taken to be the $2i$-dimensional generator $a_i$ if and only if $s_{i}[M^{2i}]=\pm m_{i}$. 

\smallskip

$SU$-manifolds and $SU$-bordism are introduced in Section~\ref{SUsect}. By a theorem of Novikov, $\OSU\otimes\Z[{\textstyle\frac12}]$ is a polynomial algebra with one generator
in every even degree~$\ge4$:
\[
  \OSU\otimes\Z[{\textstyle\frac12}]\cong
  \Z[{\textstyle\frac12}][y_i\colon i\ge2],\quad \deg y_i=2i.
\]
The bordism class of an $SU$-manifold $M^{2i}$ may be taken to be the $2i$-dimensional generator $y_i$ if and only if $s_i[M^{2i}]=\pm m_im_{i-1}$ up to a power of~$2$. The extra divisibility in dimensions $2p^k$ comes from the simple observation that the $s_i$-number of an $SU$-manifold $M^{2i}$ of dimension $2i=2p^k$ is divisible by~$p$ (Proposition~\ref{pkSU}).

\smallskip

The algebra of operations $A^U$ in complex cobordism and the Adams--Novikov spectral sequence are considered in Section~\ref{opersect}.

\smallskip

The $A^U$-module structure of $U^*(MSU)$ needed for calculations with the Adams--Novikov spectral sequence is determined in Section~\ref{UMSUsect}.
Two geometric operations are introduced. The boundary homomorphism $\partial\colon\OU_{2n}\to\OU_{2n-2}$ sends a bordism class $[M^{2n}]$ to the bordism class $[N^{2n-2}]$ dual to $c_1(M)=c_1(\det\mathcal TM)$. The restriction of $\det\mathcal TM$ to $N$ is the normal bundle $\nu(N\subset M)$.
The stably complex structure on~$N$ is
defined via the isomorphism $\mathcal T M|_N\cong\mathcal T
N\oplus\nu(N\subset M)$. Then $c_1(N)=0$, so $N$ is an $SU$-manifold. This implies that $\partial^2=0$.

Similarly, the homomorphism $\varDelta\colon\OU_{2n}\to\OU_{2n-4}$ takes a bordism class $[M^{2n}]$ to the bordism class of the submanifold $L^{2n-4}$ dual to $-c_1^2(M)=c_1(\det\mathcal 
TM)c_1(\overline {\det\mathcal TM})$ with the restriction of~$\det\mathcal TM\oplus\overline{\det\mathcal T M}$ giving the complex structure in the normal bundle.

The $A^U$-module $U^*(MSU)$ is then identified with the quotient $A^U/(A^U \varDelta + A^U \partial)$ (Theorem~\ref{umsu2}).

\smallskip

The Adams--Novikov spectral sequence for the $MSU$ spectrum is calculated in Section~\ref{calcsect}, and the consequences are drawn for the structure of the $SU$-bordism ring~$\OSU$. It is proved in Theorem~\ref{osutors} that the kernel of the forgetful homomorphism $\OSU\to\OU$ consists of torsion elements, and every torsion element in $\OSU$ has order~$2$.

To describe the torsion part of $\OSU$, Conner and Floyd~\cite{co-fl66m} introduced the group 
\[
  \mathcal W_{2n}=\Ker(\varDelta\colon\OU_{2n}\to\OU_{2n-4})
\]
and identified it with the the subgroup of $\OU_{2n}$ consisting of
bordism classes $[M^{2n}]$ such that every Chern number of
$M^{2n}$ of which $c_1^2$ is a factor vanishes (see Theorem~\ref{W3def}). The forgetful
homomorphism decomposes as $\OSU_{2n}\to\mathcal
W_{2n}\to\OU_{2n}$, and the restriction of the boundary
homomorphism $\partial\colon\mathcal W_{2n}\to\mathcal W_{2n-2}$
is defined. (A similar approach was previously used by Wall~\cite{wall60} to identify the torsion of the oriented bordism ring~$\varOmega^{SO}$.)

The relationship between the groups $\OSU_*$ and $\mathcal W_*$ is described by the following exact sequence of Conner and Floyd:
\begin{equation}\label{CF5seq}
  0\longrightarrow\OSU_{2n-1}\stackrel\theta\longrightarrow
  \OSU_{2n}\stackrel\alpha\longrightarrow\mathcal W_{2n}
  \stackrel\beta\longrightarrow\OSU_{2n-2}
  \stackrel\theta\longrightarrow\OSU_{2n-1}\longrightarrow0,
\end{equation}
where $\theta$ is the multiplication by the generator
$\theta\in\OSU_1\cong\Z_2$, $\alpha$ is the forgetful
homomorphism, and $\alpha\beta=-\partial\colon\mathcal W_{2n}\to\mathcal W_{2n-2}$. This exact sequence has the form of an exact couple, whose derived couple can be identified with the $E_2$ term of the Adams--Novikov spectral sequence for $MSU$ (see Lemma~\ref{cf3}).

Homology of $(\mathcal W_*,\partial)$ was described by Conner and Floyd~\cite[Theorem~11.8]{co-fl66m} as a polynomial algebra over $\Z_2$ on the following generators:
\[
  H(\mathcal W_*,\partial)\cong\Z_2[\omega_2,\omega_{4k}\colon k\ge 2], \quad 
  \deg \omega_2=4,\;\deg \omega_{4k}=8k.
\]
This leads to the following description of the free and torsion parts of~$\OSU$ (Theorem~\ref{freetors}):

\begin{itemize}
\item[(a)] $\mathop{\mathrm{Tors}}\OSU_n=0$ unless $n=8k+1$
or $8k+2$, in which case $\mathop{\mathrm{Tors}}\OSU_n$ is a
$\Z_2$-vector space of rank equal to the number of partitions
of~$k$.

\smallskip

\item[(b)] $\OSU_{2i}/\mathop{\mathrm{Tors}}$ is isomorphic to
$\Ker(\partial\colon\mathcal W\to\mathcal W)$ if $2i\not\equiv
4\mod 8$ and is isomorphic to
$\mathop{\mathrm{Im}}(\partial\colon\mathcal W\to\mathcal W)$ if
$2i\equiv 4\mod 8$.

\smallskip

\item[(c)] There exist $SU$-bordism classes $w_{4k}\in\OSU_{8k}$, $k\ge1$, such that every torsion element of $\OSU$ is uniquely expressible in the
form $P\cdot\theta$ or $P\cdot\theta^2$ where $P$ is a polynomial
in $w_{4k}$ with coefficients $0$ or~$1$.
An element $w_{4k}\in\OSU_{8k}$ is determined by the condition that it represents a polynomial generator $\omega_{4k}$ in $H_{8k}(\mathcal W_*,\partial)$ for $k\ge 2$, and $w_4\in\OSU_8$ represents~$\omega_2^2$.
\end{itemize}

The direct sum $\mathcal W=\bigoplus_{i\ge0}\mathcal W_{2i}$ is
\emph{not} a subring of~$\OU$: one has $[\C P^1]\in\mathcal W_2$,
but $c_1^2[\C P^1\times\C P^1]=8\ne0$, so $[\C P^1]\times[\C
P^1]\notin\mathcal W_4$. However, $\mathcal W$ becomes a
commutative ring with unit with respect to the \emph{twisted
product}
\[
  a\mathbin{*}b=a\cdot b+2[V^4]\cdot\partial a\cdot\partial b,
\]
where $\cdot$ denotes the product in $\OU$ and $V^4=\C P^1\times\C P^1-\C P^2$. 
This leads to a complex-oriented multiplicative cohomology theory introduced and studied by Buchstaber in~\cite{buch72}.

The ring structure of $\mathcal W$ is described in Theorem~\ref{Wring}:
$\mathcal W$ is an integral polynomial ring on generators in every even
degree except~$4$:
\[
  \mathcal W\cong
  \Z[x_1,x_i\colon i\ge3],\quad x_1=[\C P^1],\quad\deg x_i=2i,
\]
with $s_i(x_i)=m_im_{i-1}$ for $i\ge3$. 
The boundary operator
$\partial\colon\mathcal W\to\mathcal W$, $\partial^2=0$,
satisfies the identity
\[
  \partial(a\mathbin{*} b)=a\mathbin{*}\partial b+
  \partial a\mathbin{*} b-
  x_1\mathbin{*}\partial a\mathbin{*}\partial b.
\]
and the polynomial generators of $\mathcal W$ can be chosen so as to satisfy the relations
\[
  \partial x_1=2, \quad \partial x_{2i}=x_{2i-1}.
\]

\smallskip

The ring structure of $\OSU$ is described in Section~\ref{ringosusec}.
The forgetful map $\alpha\colon\OSU\to\mathcal W$ is a ring
homomorphism.
Therefore, the ring $\OSU/\mathrm{Tors}$ can be described as a subring of~$\mathcal W$.

We have
\[
  \mathcal W\otimes\Z[{\textstyle\frac12}]\cong
  \Z[{\textstyle\frac12}][x_1,x_{2k-1},2x_{2k}-x_1x_{2k-1}\colon k\ge2],
\]
where $x_1^2=x_1\mathbin{*}x_1$ is a $\partial$-cycle, and each of the
elements $x_{2k-1}$ and $2x_{2k}-x_1x_{2k-1}$ with $k\ge2$ is a $\partial$-cycle.

It follows from the description of the ring $\mathcal W$ that there exist indecomposable elements $y_i\in\OSU_{2i}$, $i\ge2$, such that
$s_i(y_i)=m_im_{i-1}$ if $i$ is odd, $s_2(y_2)=-48$, and
$s_i(y_i)=2m_{i}m_{i-1}$ if $i$ is even and $i>2$. These elements
are mapped as follows under the forgetful homomorphism
$\alpha\colon\OSU\to\mathcal W$:
\[
  y_2\mapsto 2x_1^2,\quad y_{2k-1}\mapsto x_{2k-1},\quad
  y_{2k}\mapsto 2x_{2k}-x_1x_{2k-1},\quad k\ge2.
\]
In
particular, $\OSU\otimes\Z[\frac12]\cong\Z[\frac12][y_i\colon i\ge2]$ embeds in $\mathcal W\otimes\Z[{\textstyle\frac12}]$ as
the polynomial subring generated by $x_1^2$, $x_{2k-1}$ and
$2x_{2k}-x_1x_{2k-1}$.

\medskip

In Part II we describe geometric representatives for $SU$-bordism classes arising from toric topology.

\smallskip

In Section~\ref{toricsect} we collect the necessary facts about toric varieties and quasitoric manifolds, their cohomology rings and characteristic classes.

\smallskip

In Section~\ref{qtSUsect} we provide explicitly constructed families of quasitoric manifolds that admit an $SU$-structure, following L\"u and Panov~\cite{lu-pa16}. Quasitoric $SU$-manifolds can be
constructed by taking iterated complex projectivisations (which
are projective toric manifolds) and then modifying the stably
complex structure so that the first Chern class becomes zero. The
underlying smooth manifold of the result is still toric, but the
stably complex structure is not the standard one. Nevertheless, the resulting $SU$-structures on quasitoric manifolds are invariant under the torus actions. The first examples of this sort were obtained by
L\"u and Wang in~\cite{lu-wa16}.

\smallskip

In Section~\ref{qtSUgensect} we describe quasitoric generators for the $SU$-bordism ring. According to a result of~\cite{lu-pa16} (which we include as Theorem~\ref{mainth}), there exist quasitoric $SU$-manifolds $M^{2i}$ of dimension $2i\ge10$ with
$s_i(M^{2i})=m_im_{i-1}$ if $i$ is odd and
$s_i(M^{2i})=2m_{i}m_{i-1}$ if $i$ is even. These quasitoric
manifolds represent the indecomposable elements $y_i\in\OSU$ which are polynomial generators of $\OSU\otimes\Z[\frac12]$. In low dimensions $2i<10$, it is known that 
quasitoric $SU$-manifolds $M^{2i}$ are null-bordant. It is therefore interesting to ask which $SU$-bordism classes of dimension $>8$ can be represented by quasitoric manifolds.

As we have seen from the description of the ring $\OSU$ above,
characteristic numbers of $SU$-manifolds satisfy intricate
divisibility conditions. Ochanine's theorem~\cite{ocha81}
asserting that the signature of an $(8k+4)$-dimensional $SU$-manifold
is divisible by~16 is one of the most famous examples. We
therefore find it quite miraculous that polynomial generators for
the $SU$-bordism ring $\OSU\otimes\Z[\frac12]$ occur within the most basic families
of examples that one can produce using toric methods: 2-stage
complex projectivisations, and 3-stage projectivisations with the
first stage being just~$\C P^1$. The proof of Theorem~\ref{mainth}
involves calculating the characteristic numbers and checking
divisibility conditions. Some interesting results on binomial coefficients modulo a prime are obtained as a byproduct.

\smallskip

In Section~\ref{SUtoricsect} we review Batyrev's construction~\cite{baty94} of Calabi--Yau manifolds arising from toric geometry. In its most basic form, this construction gives an algebraic hypersurface representing the $SU$-bordism class $\partial[V]$ for a smooth toric Fano variety~$V$. A more general construction produces (smooth) Calabi--Yau manifolds from hypersurfaces in toric Fano varieties with Gorenstein singularities, using a special desingularisation. Gorenstein toric Fano varieties correspond to so-called reflexive polytopes, and there are finitely many of them in each dimension. Four-dimensional reflexive polytopes and Calabi--Yau threefolds arising from them are completely classified~\cite{kr-sk}, \cite{a-g-h-j-n15}; there are also classification results for five-dimensional reflexive polytopes and Calabi--Yau fourfolds.

\smallskip

The $SU$-bordism classes of the Calabi--Yau hypersurfaces in smooth toric Fano varieties generate the $SU$-bordism ring $\OSU\otimes\Z[\frac12]$. More precisely, the indecomposable elements $y_i\in\OSU$ defined above  can be represented by integral linear combinations
of the bordism classes of Calabi--Yau hypersurfaces. This result, proved in~\cite{l-l-p18}, is reviewed in Section~\ref{CYSUsect}  (unlike the situation with quasitoric manifolds, there is no restriction on the dimension of $y_i$ here).

It is interesting to ask which bordism classes in $\varOmega^{SU}$ can be represented by Calabi--Yau manifolds. This question is an $SU$-analogue of the following well-known open problem of Hirzebruch: which bordism classes in $\varOmega^U$ contain connected (irreducible) non-singular algebraic varieties? If one drops the connectedness assumption, then any $U$-bordism class of positive dimension can be represented by an algebraic variety in view of a theorem of Milnor (see~\cite[p.~130]{ston68}). Since a product and a positive integral linear combination of algebraic classes are also algebraic classes (possibly, disconnected), one only needs to find in each dimension $i$ algebraic varieties $M$ and $N$ with $s_i(M)=m_i$ and $s_i(N)=-m_i$. For $SU$-bordism, the situation is different: if a class $a\in\varOmega^{SU}$ can be represented by a Calabi--Yau manifold, then $-a$ does not necessarily have this property.

This issue already occurs in complex dimension~$2$: the class $y_2\in\OU_4$ can be represented by a Calabi--Yau surface (a $K3$-surface), while $-y_2$ cannot be represented by a smooth complex surface. 
The situation is different in dimension~$3$, where both generators $y_3$ and $-y_3$ can be represented by Calabi--Yau threefolds. The same holds in complex dimension~$4$, as shown by Theorem~\ref{ldCYSUthm}.


\medskip

The authors are grateful to Victor Buchstaber and Peter Landweber for their attention to our work and for many useful comments and suggestions.

\part{Structure results}

\section{Complex bordism}\label{combor}
We briefly summarise the basic definitions and constructions of complex bordism (also known as \emph{unitary bordism} or~\emph{$U$-bordism}). More details can be found 
in~\cite{co-fl66m}, \cite{ston68}, \cite{buch12} and~\cite{bu-pa15}. 

Let $\eta_n$ denote the universal (tautological) complex $n$-plane bundle over the infinite-dimensional Grassmannian $BU(n)$. Let $\zeta$ be a real $2n$-plane bundle over a cellular space  (a $CW$-complex)~$X$. A \emph{complex structure} on $\zeta$ can be defined in one of the following equivalent ways:
\begin{itemize}
\item[(1)] an equivalence class of real vector bundle isomorphism $\zeta\to\xi$, where $\xi$ is a complex $n$-plane bundle over $X$, and two such isomorphisms are equivalent if they differ by composing with an isomorphism of complex vector bundles;

\item[(2)] a homotopy class of real $2n$-plane bundle maps $\zeta\to\eta_n$ which are isomorphisms on each fibre;

\item[(3)] a homotopy class of a lift of the map $X\to BO(2n)$ classifying the bundle $\zeta$ to a map $X\to BU(n)$.
\end{itemize}

All manifolds are smooth, compact and without boundary (unless otherwise specified).
A \emph{stably complex structure} (a \emph{unitary structure}, or a \emph{$U$-structure}) on a manifold $M$ (possibly, with boundary) is an equivalence class of complex structures on the stable tangent bundle of~$M$, that is, an equivalence class of bundle isomorphisms
\begin{equation}\label{stabcom}
  c_{\mathcal T}\colon\mathcal 
  T{M}\oplus\underline{\mathbb{R}}^k\stackrel\cong\longrightarrow\xi,
\end{equation}
where $\xi$ is complex vector bundle, and $\underline{\R}^k$ denotes the trivial real $k$-plane bundle over~$M$. Two such complex structures are said to be \emph{equivalent} if they differ by adding trivial complex summands and composing with isomorphisms of complex vector bundles. An isomorphism~\eqref{stabcom} defines a lift of the map $M\to BO(2l)$ classifying the bundle $\mathcal TM\oplus\underline\R^k$ to a map $M\to BU(l)$; here $2l=\dim_\R\xi=\dim M+k$. Composing $c_{\mathcal T}$ with an isomorphism of complex bundles results in a homotopy of the lift, and adding a trivial complex summand $\underline\C^m$ to~\eqref{stabcom} results in composing the lift with the canonical map $\BU(l)\to\BU(l+m)$. Therefore, stably complex structures on $M$ correspond naturally and bijectively to the homotopy classes of lifts of the classifying map $M\to BO$ to a map $M\to BU$. 

\begin{remark}
Instead of defining a stably complex structure as an equivalence class of isomorphisms~\eqref{stabcom}, one can define it by fixing a single isomorphism for sufficiently large~$k$. The reason is that adding trivial complex summands induces a canonical one-to-one correspondence between complex structures on the bundles $\mathcal T{M}\oplus\underline{\mathbb{R}}^k$ with different $k$ if $k\ge2$, see~\cite[Theorem~2.3]{co-fl66m}.
\end{remark}

A \emph{stably complex manifold} (a \emph{unitary manifold} or a \emph{$U$-manifold}) is a pair $(M,c_{\mathcal T})$ consisting of a manifold and a stably complex structure on it.
 
Complex (co)bordism is a generalised (co)homology theory arising from $U$-manifolds. It can be defined either geometrically or homotopically. 

In the geometric approach, the bordism group $U_n(X)$ is defined as the set of bordism classes of maps $M\to X$, where $M$ is an $n$-dimensional $U$-manifold. The details of the geometric approach are described in~\cite[\S1]{co-fl66m} (see also~\cite[Appendix~D]{bu-pa15}). We briefly recall the key points here.

\begin{construction}[geometric $U$-bordism]
A stably complex manifold $M$ \emph{bords} (or is \emph{null-bordant}) if there is a stably complex manifold with boundary $W$ such that $\partial W=M$ and the stably complex structure induced on the boundary of $W$ coincides with that of~$M$. The induced stably complex structure on $\partial W$ is defined via the isomorphism $\mathcal T W|_{\partial W}\cong\mathcal T M\oplus\underline\R$. This isomorphism depends on whether we choose an inward or outward pointing normal vector to $M$ in $W$ as a basis for~$\underline\R$, and whether we place this normal vector at the beginning or at the end of the tangent frame of~$M$. Our choice is to use the outward pointing normal and place it at the end. Then using the stably complex structure on $W$ we obtain a stably complex structure on $M=\partial W$ by means of the isomorphism
\[
  \mathcal TM\oplus\underline\R^{k+1}\cong
  \mathcal TW|_{\partial W}\oplus\underline\R^k\cong\xi.
\]

If we choose the inward pointing normal instead of the outward pointing, then the resulting stably complex structure on $M=\partial W$ will be different. If $c_{\mathcal T}\colon\mathcal 
  T{M}\oplus\underline{\mathbb{R}}^{k+1}\stackrel\cong\longrightarrow\xi$ is the stably complex structure on $M$ described above, then it can be seen that the stably complex structure resulting from the inward pointing is equivalent to the following:
\begin{equation}\label{oppcs}
  {\mathcal T}\!M\oplus\underline{\R}^{k+1}\oplus\underline{\C}\stackrel{c_{\mathcal
  T}\oplus\tau}{\lllra}\xi\oplus\underline{\C}
\end{equation}
where $\tau\colon\C\to\C$ is the complex conjugation. 

Given a stably complex manifold $(M,c_{\mathcal T})$, we refer to the stably complex structure defined by~\eqref{oppcs} as the \emph{opposite} to $c_{\mathcal T}$ and denote it by $-c_{\mathcal T}$. When $c_{\mathcal T}$ is clear from the context, we use $M$ instead of $(M,c_{\mathcal T})$ and $-M$ instead of $(M,-c_{\mathcal T})$.

For a fixed topological pair $(X,A)$ and a nonnegative integer $n$, consider pairs $(M,f)$, where $M$ is a compact $n$-dimensional $U$-manifold with boundary and $f\colon(M,\partial M)\to (X,A)$. Such a pair $(M,f)$ \emph{bords} (or is \emph{null-bordant}) if there exists a compact $(n+1)$-dimensional $U$-manifold $W$ with boundary and a map $F\colon W\to X$ such that
\begin{itemize}
\item[(a)] $M$ is a regularly embedded submanifold of $\partial W$, and the $U$-structure on $M$ is obtained by restricting the $U$-structure on~$\partial W$;
\item[(b)] $F|_{M}=f$ and $F(\partial W\setminus M)\subset A$. 
\end{itemize} 

The pairs $(M_{1},f_{1})$ and $(M_2,f_{2})$ are \emph{bordant} if the disjoint union $(M_{1},f_{1})\sqcup (-M_{2},f_{2})$ bords. Bordism is an equivalence relation: reflexivity follows by considering the stably complex structure on $M\times I$ such that $\partial(M\times I)=M\sqcup(-M)$, and transitivity uses the angle straightening procedure. The resulting equivalence class is referred to as the \emph{bordism class} of~$(M,f)$.

Denote by $[M,f]$ the bordism class of $(M,f)$. Bordism classes $[M,f]$ form an abelian group with respect to the disjoint union, which we denote $U'_n(X,A)$ for a moment, and refer to as the (geometric) \emph{unitary bordism group} of~$(X,A)$. Geometric $U$-bordism is a generalised homology theory, satisfying the Eilenberg--Steenrod axioms except for the dimension axiom. 
\end{construction}

The homotopic approach is based on the notion of \emph{$\MU$-spectrum}, which we also recall briefly. 

\begin{construction}[homotopic $U$-bordism]
The Thom space of the universal complex $n$-plane bundle $\eta_n$ over $\BU(n)$ is denoted by $\MU(n)$. The Thom spectrum $\MU=\{Y_i,\; \varSigma Y_{i}\to Y_{i+1}\colon i\ge 0\}$ has $Y_{2k}=\MU(k)$, $Y_{2k+1}=\varSigma Y_{2k}$, the map $\varSigma Y_{2k}\to Y_{2k+1}$ is the identity, and $\varSigma Y_{2k+1}\to Y_{2k+2}$ is defined as the map $\varSigma^2\MU(k)=S^2\wedge\MU(k)\to\MU(k+1)$ of Thom spaces corresponding to the bundle map $\eta_k\oplus\underline{\C}\to\eta_{k+1}$ classifying~$\eta_k\oplus\underline{\C}$. The $MU$-spectrum defines a generalised (co)homology theory, known as (homotopic) \emph{unitary (co)bordism}, with bordism and cobordism groups of a cellular pair $(X,A)$ given by
\begin{equation}\label{hbgro}
\begin{aligned}
  U_n(X,A)&=\lim_{k\to\infty}\pi_{2k+n}\bigl((X/A)\wedge\MU(k)\bigr),\\
  U^n(X,A)&=\lim_{k\to\infty}\bigl[\varSigma^{2k-n}(X/A),\MU(k)\bigr].
\end{aligned}
\end{equation}
The bordism groups of a single space $X$ are defined as $U_n(X):=U_n(X,\varnothing)$. We shall use the notation $X_+$ for $X/\varnothing$, which is $X$ with a disjoint basepoint added.
When $(X,A)$ is a finite cellular pair, the bordism group $U_n(X,A)$ is isomorphic to $\pi_{2k+n}((X/A)\wedge\MU(k))$ for sufficiently large~$k$, and similarly for $U^n(X,A)$.

By definition, the homotopic bordism and cobordism groups of a point satisfy 
\[
  U_n(\pt)=U^{-n}(\pt)=\pi_{2k+n}(\MU(k))
\]
for sufficiently large~$k$, and $U_n(\pt)=0$ for $n<0$. 
\end{construction}

The equivalence of the geometric and homotopic approaches  to complex bordism is established by the following result of Conner and Floyd.

\begin{theorem}[{\cite[(3.1)]{co-fl66m}}]
The generalised homology theory $U'_*(\cdot)$ is isomorphic, over the category of cellular pairs and continuous maps, to the generalised homology theory $U_*(\cdot)$. 
\end{theorem}
\begin{proof}[Sketch of proof]
The proof follows the original ideas of Thom~\cite{thom54} in the oriented case (see 
also~\cite[Chapter~1]{co-fl64}). We define a functor $\varphi\colon U'_n(X,A)\to U_n(X,A)$ between homology theories and prove that it induces an isomorphism on homology of a point. 

For a cellular pair $(X,A)$, there is an isomorphism $U'_n(X,A)\cong U'_n(X/A,\pt)$, so we can restrict attention to the case $A=\varnothing$ and define the maps $\varphi\colon U'_n(X)\to U_n(X)$ only.

Take a geometric bordism class $[M,f]\in U'_n(X)$ represented by a map $f\colon M\to X$ from a $U$-manifold~$M$. We embed $M$ into some $\R^{n+2k}$ and denote by $\nu$ the normal bundle of this embedding. The real bundle isomorphism $\mathcal TM\oplus\nu\cong\underline{\R}^{n+2k}$ allows us to convert stably complex structures on $M$ to complex structures on the normal bundle~$\nu$.  (This can be done in the most naive way by working with tangent and normal frames, but one needs to check that this conversion procedure is compatible with the appropriate stabilisations, see also~{\cite[(2.3)]{co-fl66m}}.)

The \emph{Pontryagin--Thom map}
\[
  S^{2k+n}\to\mathop{\mathit{Th}}(\nu)
\]
identifies a closed tubular neighbourhood of $M$ in $\R^{2k+n}\subset S^{2k+n}$
with the total space $D(\nu)$ of the disc bundle of~$\nu$, and
collapses the closure of the complement of the tubular neighbourhood to the
basepoint of the Thom space~$\mathop{\mathit{Th}}(\nu)=D(\nu)/S(\nu)$.

Now we define a map $D(\nu)\to X \times D(\eta_k)$ in which the first component is the composite $D(\nu)\longrightarrow M\stackrel f\longrightarrow X$ and the second component is the disc bundle map corresponding to the classifying map $\nu\to\eta_k$ for the above defined complex structure on~$\nu$. Doing the same for the sphere bundles, we obtain a map of pairs
\[
  \bigl(D(\nu),S(\nu)\bigr)\to\bigl(X\times D(\eta_k),X\times S(\eta_k)\bigr)
\]
and therefore a map of Thom spaces 
\[
  \mathop\mathit{Th}(\nu)\to(X/\varnothing)\wedge\MU(k).
\]
Composing with the Pontryagin--Thom map, we obtain a map $S^{2k+n}\to(X/\varnothing)\wedge\MU(k)$ representing a class in the homotopy bordism group $U_n(X)$, see~\eqref{hbgro}. One needs to check that the maps resulting from bordant pairs $(M,f)$ are homotopic, therefore defining a functor $\varphi\colon U'_*(\cdot)\to U_*(\cdot)$.

\smallskip

To show that $\varphi\colon U'_*(\pt)\to U_*(\pt)$ is an isomorphism, we construct an inverse map 
$U_*(\pt)\to U'_*(\pt)$ as follows. Take a homotopy class of maps $g\colon S^{2k+n}\to\MU(k)$ representing an element in the homotopic bordism group $U_n(\pt)$. By changing $g$ within its homotopy class we may achieve that $g$ is smooth and transverse along the zero section $\BU(k)\subset\MU(k)$. Then $M:=g^{-1}(\BU(k))$ is an $n$-dimensional submanifold in~$S^{2k+n}$. Furthermore, there is a complex bundle map from the normal bundle $\nu$ of $M$ in $S^{2k+n}$ to the normal bundle of $\BU(k)$ in $\MU(k)$, which is~$\eta_k$. We therefore obtain a complex structure on $\nu$, which can be converted into a stably complex structure on~$M$. The result is a geometric bordism class in $U'_n(\pt)$, giving an inverse map to~$\varphi$.
\end{proof}

Hereafter we denote both geometric and homotopic unitary bordism groups by~$U_*(\cdot)$.

\begin{construction}[products]
For the product bundle $\eta_m\times\eta_n$, there is the corresponding classifying map $\BU(m)\times\BU(n)\to \BU(m+n)$ (unique up to a homotopy) and the bundle map $\eta_m\times\eta_n\to\eta_{m+n}$. It induces a map of Thom spaces
\[
  \MU(m)\wedge\MU(n)\to\MU(n+m),
\]
which is associative and commutative up to homotopy. The map above is used to define product operations in complex (co)bordism, turning it into a multiplicative (co)homology theory.
Namely, there is a canonical pairing (the \emph{Kronecker product})
\[
  \langle\;\,,\,\rangle\colon U^m(X)\otimes U_n(X)\to
  \varOmega^U_{n-m},
\]
the \emph{$\frown$-product}
\[
  \frown\colon U^m(X)\otimes U_n(X)\to
  U_{n-m}(X),
\]
and the \emph{$\smile$-product} (or simply \emph{product})
\[
  \smile\colon U^m(X)\otimes U^n(X)\to
  U^{m+n}(X),
\]
defined as follows. Assume given a cobordism class $x\in U^m(X)$
represented by a map $\varSigma^{2l-m}X_+\to MU(l)$ and a bordism
class $\alpha\in U_n(X)$ represented by a map $S^{2k+n}\to
X_+\wedge MU(k)$. Then $\langle
x,\alpha\rangle\in\varOmega^U_{n-m}$ is represented by the
composite
\[
\begin{CD}
  S^{2k+2l+n-m} @>\varSigma^{2l-m}\alpha>>
  \varSigma^{2l-m}X_+\wedge MU(k)
  @>x\wedge\,\id>> MU(l)\wedge MU(k)\to MU(l{+}k)
\end{CD}
\]
If $\varDelta\colon X_+\to (X\times X)_+=X_+\wedge X_+$ is the
diagonal map, then $x\frown\alpha\in U_{n-m}(X)$ is represented by
the composite map
\begin{gather*}
\begin{CD}
  S^{2k+2l+n-m} @>\varSigma^{2l-m}\alpha>>
  \varSigma^{2l-m}X_+\wedge MU(k) @>\varSigma^{2l-m}\!\varDelta\,\wedge\,\id>>
  X_+\wedge\varSigma^{2l-m}X_+\wedge MU(k)
\end{CD}\\
\begin{CD}
  @>\id\wedge x\wedge\id>> X_+\wedge MU(l)\wedge MU(k)\to X_+\wedge MU(l+k)
\end{CD}
\end{gather*}
The $\smile$-product is defined similarly; it turns
$U^*(X)=\bigoplus_{n\in\Z}U^n(X)$ into a graded commutative ring, called the
\emph{complex cobordism ring of~$X$}. The direct sum 
\[
  \varOmega_U:=U^*(\pt)=\bigoplus_n U^{n}(\pt)
\]
is often called simply the \emph{complex cobordism ring}. It is graded by nonpositive integers. We also use the notation $\OU$ for the nonnegatively graded ring $U_*(\pt)=\bigoplus_n U_{n}(\pt)$, the \emph{complex bordism ring}, where $U_n(\pt)=U^{-n}(\pt)$. Each ring $U^*(X)$ is a module over~$\varOmega_U$.

A stably complex $n$-manifold $M$ has the \emph{fundamental bordism class} $[M]\in U_n(M)$, which is defined geometrically as the bordism class of the identity map $M\to M$. There are the
\emph{Poincar\'e--Atiyah duality} isomorphisms~\cite{atiy61}, see also~\cite[Construction~D.3.4]{bu-pa15}:
\[
  D_U\colon U^k(M)\stackrel\cong\longrightarrow U_{n-k}(M),\quad x\mapsto x\frown[M].
\]
\end{construction}

We have 
\[  
  H^*(BU(n);\Z)\cong\Z[c_1,\ldots,c_n],\quad \deg c_i=2i,
\] 
where the $c_i$ are the universal Chern characteristic classes. Given a partition $\omega=(i_1,\ldots,i_k)$ of $n=|\omega|=i_1+\cdots+i_k$ by positive integers, define the monomial $c_\omega=c_{i_1}\cdots c_{i_k}$ of degree
$2|\omega|$ and the corresponding
characteristic class $c_\omega(\xi)$ of a complex $n$-plane
bundle~$\xi$. The corresponding tangential Chern
\emph{characteristic number} of a stably complex manifold $M$ is
defined by 
\[
  c_\omega[M]:=\langle c_\omega(\mathcal T M),[M]\rangle.
\]
Here $[M]$ is the fundamental homology class of~$M$,
and $\mathcal T M$ is regarded as a complex bundle via the
isomorphism~\eqref{stabcom}. We often write $c_\omega(M)$ instead of $c_\omega(\mathcal T M)$ for a stably complex manifold~$M$. The number $c_\omega[M]$ is assumed
to be zero when $2|\omega|\ne\dim M$.

One important characteristic class is $s_n$. It is defined as
the polynomial in $c_1,\ldots,c_n$ obtained by expressing the
symmetric polynomial $x_1^n+\cdots+x_n^n$ via the elementary
symmetric functions $i_i(x_1,\ldots,x_n)$ and then replacing
each $i_i$ by~$c_i$. Define the corresponding characteristic
number as 
\[
  s_n[M]:=\langle s_n(\mathcal T M),[M]\rangle.
\]
It is known as the \emph{$s$-number} or the \emph{Milnor number} of~$M$.

For any integer $i\ge 1$, set 
\begin{equation}\label{mi}
  m_{i}=\begin{cases}
  1&\text{if $i+1\neq p^k$ for any prime $p$;}\\
  p&\text{if $i+1=p^k$ for some prime $p$ and integer $k>0$.}
  \end{cases}
\end{equation}

The structure of the $U$-bordism ring $\OU$ is described by the following fundamental result of Milnor and Novikov:

\begin{theorem}[Milnor, Novikov]\label{Ustructure}\ 
\begin{itemize}
\item[(a)]
The complex bordism ring $\OU$ is a polynomial ring over $\Z$ with one generator in every positive even dimension:
\[
  \OU\cong\mathbb{Z}[a_{i}\colon i\ge 1],\quad \deg{a_i}=2i.
\]

\item[(b)]
The bordism class of a stably complex manifold $M^{2i}$ may be taken to be the $2i$-dimensional generator $a_i$ if and only if
\[
  s_{i}[M^{2i}]=\pm m_{i}.
\]

\item[(c)]
Two stably complex manifolds are bordant if and only if they have identical
sets of Chern characteristic numbers.
\end{itemize}
\end{theorem}

Part (c) of Theorem~\ref{Ustructure} can be restated by saying that the universal characteristic numbers homomorphism $e\colon\OU_{2n}\to H_{2n}(BU)$ is a monomorphism is each dimension. The latter homomorphism (for the normal characteristic numbers) can be identified with the composite
\[
  \OU_{2n}=\pi_{2n+2N}(MU(N))\longrightarrow H_{2n+2N}(MU(N))\longrightarrow 
  H_{2n}(BU(N))
\]
of the Hurewicz homomorphism and Thom isomorphism. By Serre's Theorem, the Hurewicz homorphism above is an isomorphism modulo the class of finite groups. The injectivity of $e\colon\OU_{2n}\to H_{2n}(BU)$ then follows from the absence of torsion in~$\OU$.

The ring isomorphism  $\OU\cong\mathbb{Z}[a_{i}\colon i\ge 1]$, $\deg{a_i}=2i$, was proved in 1960 by Novikov~\cite{novi60} using the Adams spectral sequence and the structure theory of Hopf algebras. A more detailed account of this argument was given in~\cite{novi62}. Milnor's work~\cite{miln60} contained only the proof of the additive isomorphism (including the absence of torsion in~$\OU$ and the ranks calculation); the ring structure of $\OU$ was intended to be included in the second part of~\cite{miln60}, which was not published. Another geometric proof for the ring isomorphism was given by Stong in 1965 and included in his monograph~\cite{ston68}. All these results preceded the introduction of formal group law techniques in cobordism by Novikov~\cite{novi67}. Quillen~\cite{quil71} used formal group laws and tom Dieck's power operations to prove that the classifying map from Lazard's universal formal group law to the formal group law in complex cobordism induces the ring isomorphism $\mathbb{Z}[a_{i}\colon i\ge 1]\cong\OU$.


\begin{construction}[formal group law of geometric cobordisms]\label{confgl}
Let $X$ be a cellular space. Since $\C P^\infty\simeq\MU(1)$, the cohomology group $H^2(X)=[X,\C P^\infty]$ is a subset (not a subgroup!) of the cobordism group $U^2(X)$. That is, every element $x\in H^2(X)$ determines a cobordism class $u_x\in U^2(X)$. The elements of $U^2(X)$ obtained in this way are called \emph{geometric cobordisms} of~$X$.

When $X=X^{k}$ is a manifold, a class $x\in H^2(X)$ is Poincar\'e dual to a submanifold $M\subset X$ of codimension~$2$ with a fixed complex structure on the normal bundle. Furthermore, if $X$ is a stably complex manifold representing a bordism class $[X]\in\OU_k$, then we have 
\[
  [M]=\varepsilon D_U(u_x)\in\OU_{k-2},
\]
where $D_U\colon U^2(X)\to U_{k-2}(X)$ is the Poincar\'e--Atiyah duality map and $\varepsilon\colon U_{k-2}(X)\to\OU_{k-2}$ is the augmentation. By definition, $\varepsilon D_U$ is the Kronecker product with~$[X]$.

Given two geometric cobordisms $u,v\in U^2(X)$ corresponding to elements $x,y\in
H^2(X)$ respectively, we denote by $u+_{\!{}_H}\!v$ the geometric
cobordism corresponding to the cohomology class $x+y$. Then following relation holds in $U^2(X)$:
\begin{equation}\label{fglgc}
  u+_{\!{}_H}\!v=F_U(u,v)=u+v+\sum_{k\ge1,\,l\ge1}\alpha_{kl}\,u^kv^l,
\end{equation}
where the coefficients $\alpha_{kl}\in\varOmega_U^{-2(k+l-1)}$ do
not depend on $u,v$ and~$X$. The series $F_U(u,v)$ given by~\eqref{fglgc} is
a (commutative one-dimensional) formal group law over the complex cobordism ring~$\varOmega_U$. It was introduced by Novikov in~\cite[\S5, Appendix~1]{novi67} and called the \emph{formal group law of geometric cobordisms}.
More details of this  construction can be found in~\cite{buch12} and~\cite[Appendix~E]{bu-pa15}.
\end{construction}

We have
\[
  U^*(BU)=\varOmega_U[[\cf_1,\cf_2,\ldots,\cf_i,\ldots]],
\]
where $\cf_i$ is the $i$th universal Conner--Floyd characteristic class, and the identity above is understood as an isomorphism between the graded components.
For a complex vector bundle $\xi$ over a cellular space $X$, the Conner--Floyd characteristic class $\cf_i(\xi)\in U^{2i}(X)$ is defined as the pullback $f^*(\cf_i)$ along the map $f\colon X\to\BU$ classifying~$\xi$.

Let $\eta$ be the tautological line bundle over~$\C P^\infty$ and let $\bar\eta$ be its conjugate (the line bundle of a hyperplane). The class $u=\cf_1(\bar\eta)\in U^2(\C P^\infty)$ is
the cobordism class corresponding to the inclusion $\C P^\infty=\BU(1)\to\MU(1)$, which is a homotopy equivalence. In other words, $\cf_1(\bar\eta)$ is the geometric cobordism corresponding to the first Chern class $c_1(\bar\eta)\in H^2(\C P^\infty)$.  Then $\cf_1(\eta)\in U^2(\C P^\infty)$ is the power series inverse to $u=\cf_1(\bar\eta)$ in the formal group law~$F_U$; we denote this series by~$\overline{u}$.

Similarly, for a complex line bundle $\xi$ over a cellular space~$X$, the first Conner--Floyd class $\cf_1(\xi)\in U^2(X)$ coincides with the geometric cobordism corresponding to $c_1(\xi)\in H^2(X)$. 
The formal group law of geometric cobordisms gives the expression of the first Conner--Floyd class of the tensor product $\xi\otimes\zeta$ of line bundles over~$X$ in terms of the classes
$u=\cf_1(\xi)$ and $v=\cf_1(\zeta)$:
\[
  \cf_1(\xi\otimes\zeta)=F_U(u,v).
\]

If $\xi$ is a complex vector bundle of an arbitrary dimension over~$X$, then the geometric cobordism corresponding to $c_1(\xi)\in H^2(X)$ is $\cf_1(\det\xi)\in U^2(X)$ (it is defined by the map $X\to\C P^\infty$ classifying the determinant line bundle $\det\xi$). In general,
$\cf_1(\det\xi)\ne\cf_1(\xi)$. Consider the determinant homomorphism $\det\colon U\to U(1)$ and the corresponding map $\det\colon BU\to BU(1)=\C P^\infty$. We define the universal characteristic class $\du=\det^*u\in U^2(BU)$. Then we have $\du(\xi)=\cf_1(\det\xi)$.

\section{$SU$-manifolds and the $SU$-spectrum}\label{SUsect}

A \emph{special unitary structure} (an \emph{SU-structure}) on a manifold $M$ is a stably complex structure $c_{\mathcal T}$, see~\eqref{stabcom}, with a choice of an $SU$-structure on the complex vector bundle~$\xi$. Equivalently, an $SU$-structure is the homotopy class of a lift of the map $M\to BU$ classifying $\xi$ to a map $M\to BSU$. A stably complex manifold $(M,c_{\mathcal T})$ admits an $SU$-structure if and only if 
the first (integral) Chern class of $\xi$ vanishes: $c_1(\xi)=0$. Furthermore, such an $SU$-structure is unique if $H^1(M;\mathbb Z)=0$ (the latter follows by considering the homotopy fibration sequence corresponding to the fibration $\BSU\to\BU$ with fibre~$S^1$). An \emph{$SU$-manifold} is a stably complex manifold with a fixed $SU$-structure. By some abuse of notation, we often refer to a stably complex manifold $M$ with $c_1(M)$ as an $SU$-manifold, meaning that such a manifold admits an $SU$-structure.

There is a generalised homology theory resulting from $SU$-structures, known as $SU$-bordism. As in the case of $U$-bordism, it can be defined either geometrically or homotopically. 

In the geometric approach, the bordism group $SU_n(X)$ is defined as the set of bordism classes of maps $M\to X$, where $M$ is an $n$-dimensional $SU$-manifold. The homotopic approach is based on the notion of the \emph{$\MSU$-spectrum}. Let $\widetilde\eta_n$ denote the universal (tautological) complex $n$-plane bundle over $\BSU(n)$. The Thom space of $\widetilde\eta_n$ is denoted by $\MSU(n)$. The Thom spectrum $\MSU=\{Z_i,\; \varSigma Z_{i}\to Z_{i+1}\colon i\ge 0\}$ has $Z_{2k}=\MSU(k)$ and $Z_{2k+1}=\varSigma Z_{2k}$. The \emph{$\SU$-bordism} and \emph{cobordism} groups of a cellular pair $(X,A)$ are given by
\[
\begin{aligned}
  SU_n(X,A)&=\lim_{k\to\infty}\pi_{2k+n}\bigl((X/A)\wedge\MSU(k)\bigr),\\
  SU^n(X,A)&=\lim_{k\to\infty}\bigl[\varSigma^{2k-n}(X/A),\MSU(k)\bigr].
\end{aligned}
\]
These define a multiplicative generalised (co)homology theory, as in the case of $U$-bordism.

The \emph{$SU$-bordism ring} is defined as $\OSU=SU_*(\pt)$. 

Unlike $\OU$, the ring $\OSU$ has torsion. The first torsion element appears already in dimension~$1$: the fact that $\MSU(k)$ has no cells in dimensions $2k+1$ through $2k+3$ implies that $\OSU_1=\pi_1^s=\Z_2$. The generator $\theta$ of $\OSU_1$ is represented by a circle with a nontrivial framing inducing a nontrivial $SU$-structure. 

The first structure result on the ring $\OSU$ was a theorem of Novikov from 1962, showing that $\OSU$ becomes a polynomial ring if we invert~$2$ (otherwise it is not a polynomial ring, even modulo torsion).
Recall from Theorem~\ref{Ustructure} that a bordism class $[M^{2i}]\in\OU_{2i}$ is a  polynomial generator of $\OU$ whenever $s_i[M^{2i}]=\pm m_i$, where the numbers $m_i$ are defined in~\eqref{mi}. More intricate divisibility conditions on the $s_i$-number are required to identify polynomial generators in the ring $\OSU\otimes\Z[{\textstyle\frac12}]$.

\begin{theorem}[{Novikov~\cite[Appendix~1]{novi62}}]\label{noviosu}
$\OSU\otimes\Z[{\textstyle\frac12}]$ is a polynomial algebra with one generator
in every even degree~$\ge4$:
\[
  \OSU\otimes\Z[{\textstyle\frac12}]\cong
  \Z[{\textstyle\frac12}][y_i\colon i\ge2],\quad \deg y_i=2i.
\]
The bordism class of an $SU$-manifold $M^{2i}$ may be taken to be the $2i$-dimensional generator $y_i$ if and only if
\[
  s_i[M^{2i}]=\pm m_im_{i-1}\quad\text{up to a power of~$2$}.
\]
\end{theorem}
Note that up to a power of $2$ we have
\[
  m_im_{i-1}=\begin{cases}
  1&\text{if $i\ne p^k$, $i\ne p^k-1$ for an odd prime $p$},\\
  p&\text{if $i= p^k$ or $i= p^k-1$ for an odd prime $p$}.
  \end{cases}
\]
The extra divisibility in dimensions $2i=2p^k$ comes from the following simple observation:

\begin{proposition}\label{pkSU}
If $M^{2n}$ is an $SU$-manifold of dimension $2n=2p^k$ for a prime $p$, then
\[
  s_n[M^{2n}]=0\mod p.  
\]
\end{proposition}
\begin{proof}
For $n=p^k$ we have
\[
  s_n(M^{2n})=x_1^n+\cdots+x_n^n\equiv(x_1+\cdots+x_n)^n=c_1^{n}(M^{2n})=0\mod p\qedhere
\]
\end{proof}

As in the case of unitary bordism, Theorem~\ref{noviosu} implies that the $SU$-bordism class of an $SU$-manifold is determined modulo $2$-primary torsion by its characteristic numbers. By the result of Anderson, Brown and Peterson~\cite{a-b-p66}, $KO$-theory chracteristic numbers together with the ordinary characteristic numbers determine the $SU$-bordism class completely.

\section{Operations in complex cobordism and the Adams--Novikov spectral sequence}\label{opersect}

A (stable) \emph{operation} $\theta$ of degree $n$ in complex cobordism is a family of additive maps
\[
  \theta\colon U^k(X,A)\to U^{k+n}(X,A),
\] 
defined for all cellular pairs $(X,A)$, which are functorial in $(X,A)$ and commute with the suspension isomorphisms. The set of all operations is a ring with respect to addition and composition; furthermore, there is an algebra structure over the ring~$\varOmega_U$.  This algebra is denoted by~$A^U$; it was described in the works of Landweber~\cite{land67} and Novikov~\cite[\S5]{novi67}.

\begin{construction}[operations and characteristic classes]
There is an isomorphism of $\varOmega_U$-modules
\[
  A^U\cong U^*(MU)=\lim\limits_{\longleftarrow}U^{*+2N}(\MU(N)).
\]
Given an element $a \in U^n(MU)$ of $A^U$ represented by a map of spectra $a \colon MU \to \varSigma^n MU$, we denote the corresponding operation by 
\[
  a^*\colon U^*(X)\to U^{*+n}(X),
\]
where $X$ is cellular space. The operation $a^*$ is described as follows.
Given an element $x\in U^m(X)$ represented by a map $x \colon X \to \varSigma^m MU$, the element $a^*x \in U^{m+n}(X)$ is represented by the composite 
\[
  X \stackrel{x} \longrightarrow \varSigma^m MU \stackrel{\varSigma^m a} 
  \lllra \varSigma ^{m+n} MU.
\]
This defines a left action of $A^U$ on the cobordism groups of~$X$, and turns $U^*$ into a functor to the category of graded left $A^U$-modules.

There is a similarly defined action 
\[  
  a_*\colon U_*(X)\to U_{*-n}(X)
\]
 of $A^U$ on the bordism groups. Given an element $x\in U_m(X)$ represented by a map $x \colon \varSigma^m S \to X \wedge MU$, the element $a_*x \in U_{m-n}(X)$ is represented by the composite 
\[
  \varSigma^{m-n} S \xrightarrow{\varSigma^{-n} x} \varSigma^{-n} (X \wedge MU)  
  \xrightarrow{\varSigma^{-n} (1 \wedge a)} X \wedge MU.
\]

There are natural Thom isomorphisms 
\[
  \varphi^N_* \colon U_{n+2N} (MU(N)) \to U_n (BU(N)),
  \quad \varphi_N^* \colon U^n (BU(N)) \to U^{n+2N} (MU(N)).
\]  
As $U_n(BU)$ is the direct limit of $U_n(BU(N))$, and $U^n(BU)$ is the inverse limit of $U^n(BU(N))$, and similarly for $MU$, we also have the stable Thom isomorphisms 
\[
  \varphi_* \colon U_{n} (MU) \to U_n (BU),\quad
  \varphi^* \colon U^n (BU) \to U^{n} (MU).
\]  
It follows that every universal characteristic class $\alpha\in U^n(BU)$ defines an operation $a = \varphi^* (\alpha)\in U^n(MU)$, and vice versa.

If $x\in U_m(X)$ is represented by a singular manifold $M^m \stackrel{f} \longrightarrow X$, then $a_* x$ can be interpreted geometrically as follows. Let $\alpha=(\varphi^*)^{-1}a$ be the characteristic class corresponding to~$a$.
Consider $\alpha (-\mathcal TM) \in U^n(M^m)$, where $\mathcal TM$ is the tangent bundle and  $-\mathcal TM$ is the stable normal bundle of~$M$. Applying the Poincar\'e--Atiyah duality operator $D_U \colon U^n(M^m) \to U_{m-n}(M^m)$ we obtain the element $D_U \alpha (-\mathcal TM) \in U_{m-n}(M)$ represented by $Y_{\alpha} \stackrel{f_{\alpha}} \longrightarrow M$. Then, $a_*x \in U_{m-n}(X)$ is represented by the composite $Y_{\alpha} \stackrel{f_{\alpha}} \longrightarrow M \stackrel{f} \longrightarrow X$.

There is an isomorphism of left $\varOmega_U$-modules 
\[
  A^U=U^*(MU)\cong\varOmega_U\mathbin{\widehat\otimes} S, 
\]
where $\widehat\otimes$ is the completed tensor product, and $S$ is the \emph{Landweber--Novikov algebra}, generated by the operations $S_\omega=\varphi^*(s^{\scriptscriptstyle U}_\omega)$ corresponding to universal characteristic classes $s^{\scriptscriptstyle U}_\omega\in U^*(BU)$ defined by symmetrising monomials $t_1^{i_1}\cdots t_k^{i_k}$ indexed by partitions $\omega=(i_1,\ldots,i_k)$. Therefore, any element $a\in A^U$ can be written uniquely as an infinite series $a = \sum_{\omega} \lambda_{\omega} S_{\omega}$ where $\lambda_{\omega}\in\varOmega_U$. The Hopf algebra structure of $S$ is described in~\cite{land67} and~\cite[\S5]{novi67}.
\end{construction}

Restricting to the case $X=\pt$, we obtain representations of $A^U$ on $\varOmega_U=U^*(\pt)$ and $\varOmega^U=U_*(\pt)$. Unlike the situation with the ordinary (co)homology, we have

\begin{lemma}[{see \cite[Lemma~3.1 and Lemma~5.2]{novi67}}]\label{faith}
The representations of $A^U$ on $\varOmega_U=U^*(\pt)$ and $\varOmega^U=U_*(\pt)$ are faithful.
\end{lemma}

\begin{remark}
More generally, given spectra $E$, $F$ of finite type, the natural homomorphism $F^*(E) \to \Hom^*(\pi_*(E), \pi_*(F))$ is injective when $\pi_*(F)$ and $H_*(E)$ do not have torsion; see~\cite{rudy98} for details.
\end{remark}

Alongside with the representation of $A^U$ in the bordism $U_*(X)$ of any $X$, there is another representation $A^U$ in $U_*(BU)$ defined as follows.

\begin{construction}[representation of $A^U$ in $U_*(BU)$, $a \mapsto \widetilde a$]\label{atilde}
Let $a \in U^n(MU)$ be an element of $A^U$. We define 
\[
  \widetilde a:=\varphi_* a_* \varphi_*^{-1}\colon U_m(BU)\to U_{m-n}(BU).
\]
The geometrical meaning of this operation is described as follows.
Let $[M,\xi] \in U_m(BU)$ be a bordism class, where $\xi$ is the pullback of the (stable) tautological bundle over $BU$ along a singular manifold $M \to\BU$.  The element $a \in U^n(MU)$ defines a universal characteristic class $\alpha=(\varphi^*)^{-1}a \in U^n(BU)$ and a class $\alpha (\xi) \in U^n(M)$. Consider the Poincar\'e--Atiyah dual class $D_U (\alpha (\xi))=[Y_a, f_a]  \in U_{m-n} (M)$, where $Y_a \stackrel{f_a} \longrightarrow M$ is a singular manifold of~$M$. Then 
\[
  \widetilde a[M,\xi] 
  =[Y_a, f_a^* (\xi + \mathcal TM) - \mathcal TY_a] \in U_{m-n} (BU).
\]
Applying the augmentation $\varepsilon\colon U_*(\BU)\to\OU$ we obtain
\begin{equation}\label{kron1}
  \varepsilon(\widetilde a[M,\xi])=[Y_a]=\bigl\langle(\varphi^*)^{-1}a,[M,\xi]\bigr\rangle
  \in U_{m-n}(\pt)=\varOmega_{m-n}^U,
\end{equation}
where $\langle\;,\,\rangle$ denotes the Kronecker product in (co)bordism of~$BU$.
\end{construction}

\begin{lemma}\label{atilfaith}
The representation $A^U$ on $U_*(BU)$ given by $a \mapsto \widetilde a$ is faithful.
\end{lemma}
\begin{proof}
Setting $\xi = -\mathcal TM$ in Construction~\ref{atilde}, we obtain 
\[
  \widetilde a [M, -\mathcal TM] = [Y_a, -\mathcal TY_a].
\]
This implies that we can consider the representation $a\mapsto a_*$ on $U_*(\pt)$ as a subrepresentation of the representation $a\mapsto\widetilde a$ on~$U_*(BU)$. Since $a\mapsto a_*$ is faithful by Lemma~\ref{faith}, the representation~$a\mapsto\widetilde a$ is also faithful.
\end{proof}

The main properties of the cohomological Adams--Novikov spectral sequence for complex cobordism are summarised next. Details can be found in~\cite{novi67}; see also~\cite{mish67}, \cite{baas70}, \cite{botv92}.

\begin{theorem}[Adams--Novikov spectral sequence for complex cobordism]\label{ANth}
Let $X$ be a connective spectrum whose ordinary homology with $\Z$-coefficients is torsion free and finitely generated in each dimension. Then there exists a spectral sequence 
\[
  \{E_r^{p, q},\quad d_r\colon E^{p, q}_r \to E^{p+r, q+r-1}_r,\quad r\ge2\} 
\]
with the following properties:

\begin{itemize}
\item[(a)] $E^{p, q}_2 = \Ext^{p,q}_{A^U}(U^*(X), U^*(\pt))$, where $U^*$ is the complex cobordism theory and $A^U=U^*(MU)$ is the algebra of operations.

\item[(b)] There exists a filtration 
\[
  \pi_n (X) = F^{0, n} \supset F^{1, n+1} \supset F^{2, n+2} \supset \cdots,
  \quad \bigcap_{s \geqslant 0} F^{s, n+s} = 0 
\]  
whose adjoint bigraded module coincides with the infinity term of the spectral sequence: $E^{p, q}_{\infty} \cong F^{p, q} / F^{p+1, q+1}$.

\item[(c)] The edge homomorphism 
\[
  \pi_n (X) = F^{0, n} \to E^{0, n}_{\infty} \to E^{0, n}_2 = \Hom_{A^U}^n(U^*(X), U^*(\pt))
\]  
coincides with the naturally defined map.
\end{itemize}
Furthermore, if $X$ is a ring spectrum, then the spectral sequence is multiplicative.
\end{theorem}

\begin{remark}
The natural map $h \colon \pi_n (X) \to \Hom_{A^U}^n(U^*(X), U^*(pt))$ in Theorem~\ref{ANth}~(c) is defined as follows. Given an element $\alpha \in \pi_n (X)$ represented by a map $f \colon \varSigma^n S \to X$ and an element $\beta \in U^p(X)$ represented by a map $g \colon X \to \varSigma^p \MU$, the element $h (\alpha)(\beta) \in U^{p-n}(\pt)$ is represented by the composite
\[
  \varSigma^n S \stackrel{f}\lllra
  X \stackrel {g} \lllra \varSigma^{p} MU.
\]
\end{remark}

\section{The $A^U$-module structure of $U^*(MSU)$}\label{UMSUsect}

In order to apply Theorem~\ref{ANth} to the special unitary bordism spectrum $\MSU$ we need to describe the $A^U$-module $U^*(\MSU)$. The main result here (Theorem~\ref{umsu2}) is due to Novikov. We provide a complete proof by filling in some details missing in~\cite{novi67}.

Consider the universal characteristic class $\du\in U^2(BU)$ introduced at the end of Section~\ref{combor}, $\du(\xi)=\cf_1(\det\xi)$. We also set $\bdu=\cf_1(\overline{\det\xi})$. The spectral sequence of the fibration $\BSU \to \BU\stackrel{\det}\longrightarrow BU(1)$ implies that the homomorphism $U^*(\BU)\to U^*(\BSU)$ is surjective and its kernel is the ideal $I(\du)$ generated by~$\du$. Using the Thom isomorphisms 
\[
  \varphi^* \colon U^*(\BSU) \to U^*(\MSU)\quad\text{and}\quad
  \varphi^* \colon U^*(\BU)  \to U^*(\MU),
\]
we obtain that the natural map $MSU\to MU$ induces an epimorphism $U^*(\MU)  \to U^*(\MSU)$ with kernel $\varphi^*(I(\du))$. As $U^*(\MU)  \to U^*(\MSU)$ is an $A^U$-module map, we obtain 
\begin{equation}\label{umsu1}
  U^*(\MSU) = A^U/\varphi^*(I (\du))\quad
  \text{ as an $A^U$-module}.
\end{equation}
This is the first description of the required $A^U$-module structure.
 
\medskip

Next we define some important operations in $A^U$. Recall that every characteristic class $\alpha\in U^*(\BU)$ defines an operation $\varphi^*(\alpha)\in A^U=U^*(\MU)$.

\begin{construction}[operations $\varDelta_{(k_1, k_2)}$]\label{Delta}
Given positive integers $k_1$, $k_2$, define
\[
  \varDelta_{(k_1, k_2)} = \varphi^* \bigl((\bdu)^{k_1}(\du)^{k_2}\bigr) \in   
  (A^U)^{2k_1+2k_2}.
\]
The corresponding operation $\widetilde \varDelta_{(k_1, k_2)}\colon U_*(BU)\to U_{*-2k_1-2k_2}(BU)$ (see Construction~\ref{atilde}) can be described geometrically as follows.
Assume given $[M, \xi]\in U_n(BU)$. Let $i_1 \colon{Y_1 \hookrightarrow M}$ and $i_2 \colon Y_2 \hookrightarrow M$ be codimension-$2$ submanifolds Poincar\'e dual to $-c_1(\xi)$ and $c_1(\xi)$ respectively. 
We have $\nu({Y_1 \subset M}) = \overline{(\det \xi)} |_{Y_1}$ and $\nu(Y_2 \subset M) = (\det \xi) |_{Y_2}$. The same submanifolds are Poincar\'e--Atiyah dual to the classes $\cf_1(\overline{\det\xi})=\bdu(\xi)$ and $\cf_1(\xi)=\du(\xi)$, respectively. The submanifold Poincar\'e--Atiyah dual to $(\bdu(\xi))^{k_1} (\du(\xi))^{k_2}\in U^{2k_1+2k_2}(M)$ is given by the transverse intersection 
\[
  Y_{k_1, k_2} = \underbrace{Y_1 \cdots Y_1}_{k_1}\cdot
  \underbrace{Y_2 \cdots Y_2}_{k_2}.
\]  
with the complex structure in the normal bundle $\nu=\nu (Y_{k_1, k_2}\subset M)=(\overline {\det \xi})^{\oplus k_1} \oplus (\det \xi)^{\oplus k_2}|_{Y_{k_1,k_2}}$. Then we have
\[
  \widetilde \varDelta_{(k_1, k_2)} [M, \xi] = [Y_{k_1, k_2},  
  \xi|_{Y_{k_1,k_2}}\!+\nu\,]\in U_{n-2k_1-2k_2}(BU).
\]  
In the case when $\xi = -\mathcal TM$ we obtain $(\varDelta_{(k_1, k_2)})_* [M] = [M_{k_1, k_2}]$, where $M_{k_1, k_2}$ is the submanifold dual to $(\det\mathcal TM)^{\oplus k_1} \oplus (\overline {\det\mathcal TM})^{\oplus k_2}$.

\end{construction}

\begin{construction}[operations $\varPsi_{(k_1, k_2)}$]\label{Psi}
Given nonnegative integers $k_1,k_2$, set $k=k_1+k_2$. Let $\xi$ be a complex line bundle over $\C P^n$. Consider the projectivisation bundle
$p\colon\C P(\xi \oplus \underline{\C}^k)\to\C P^n$ where $\underline{\C}^k$ denotes the trivial bundle of rank~$k$. The tangent bundle of $\C P(\xi \oplus \underline{\C}^k)$ splits stably as
\[
   \mathcal T \C P(\xi \oplus \underline{\C}^k)\oplus\underline\C \cong
   p^*\mathcal T\C P^n\oplus(\bar\eta\otimes p^*(\xi\oplus\underline{\C}^k))=
   p^*\mathcal T\C P^n\oplus(\bar\eta\otimes p^*\xi)\oplus
   \bar\eta^{\oplus k},
\]
where $\eta$ denotes the tautological line bundle over $\C P(\xi \oplus \underline{\C}^k)$, see~\cite[Theorem~D.4.1]{bu-pa15}. We change the stably complex structure on $\C P(\xi \oplus \underline{\C}^k)$ to a new one, determined by the isomorphism of real vector bundles
\[
   \mathcal T \C P(\xi \oplus \underline{\C}^k)\oplus\underline\R^2 \cong
   p^*\mathcal T\C P^n\oplus(\bar\eta\otimes p^*\xi)\oplus\bar\eta^{\oplus k_1}
  \oplus\eta^{\oplus k_2},
\]
and denote the resulting stably complex manifold by $P^{(k_1, k_2)}(\xi)$.

We obtain a bordism class $[P^{(k_1, k_2)}(\xi), p] \in U_{2n+2k}(\C P^n)$. Its dual cobordism class $\chi_{(k_1, k_2)} (\xi):=(D_U)^{-1}[P^{(k_1, k_2)}(\xi), p] \in U^{-2k}(\C P^n)$ defines a universal cobordism characteristic class of line bundles, which we denote $\chi_{(k_1, k_2)} \in U^{-2k} (\C P^{\infty})$.

Now we can extend the definition of $\chi_{(k_1, k_2)}$ to complex vector bundles of arbitrary rank by setting $\chi_{(k_1, k_2)} (\xi):= \chi_{(k_1, k_2)} (\det \xi)$. As a result, we obtain a universal characteristic class $\chi_{(k_1, k_2)}  \in U^{-2k}(BU)$ and the corresponding operation
\[
  \varPsi_{(k_1, k_2)} = \varphi^* \chi_{(k_1, k_2)}  \in U^{-2(k_1+k_2)}(MU) =   
  (A^U)^{-2(k_1+k_2)}.
\]
Geometrically, $(\varPsi_{(k_1, k_2)})_* [M^{2n}]$ is the $(2n + 2k_1 + 2k_2)$-manifold $[\C P (\overline{\det \mathcal T M} \oplus \C^{k_1 +k_2})]$ with the stably complex structure $p^*(\mathcal T M)\oplus(\bar\eta \otimes p^*(\overline{\det \mathcal T M}))\oplus \bar\eta^{\oplus k_1} \oplus\eta^{\oplus k_2}$.

We use the following notation for particular operations:
\begin{equation}\label{partialDelta}
  \partial=\varDelta_{(1,0)},\quad \varDelta=\varDelta_{(1,1)},\quad
  \chi=\varPsi_{(1,0)},\quad\varPsi=\varPsi_{(1,1)}.
\end{equation}
Geometrically, $\partial_*[M]$ is represented by a submanifold dual to $c_1(\det\mathcal TM)=c_1(M)$, and $\chi_*[M]$ is represented by the manifold $\C P (\overline {\det \mathcal T M} \oplus\underline{\C})$ with the standard stably complex structure.
The operations $\partial_*$ and $\varDelta_*$ were studied in detail by Conner and Floyd~\cite{co-fl66m}, they denoted them simply by $\partial$ and~$\varDelta$. 
\end{construction}

The operations introduced above satisfy algebraic relations described next.

\begin{lemma}\label{algrel}
We have
\[
   \partial^2 = \varDelta \partial = 0,\quad
   \varDelta \varPsi =\id,\quad \partial \varPsi = 0,\quad
   \chi\partial=[\C P^1]\partial,\quad
   \partial \chi \partial = 2 \partial.
\]
\end{lemma}
\begin{proof}
By Lemma~\ref{faith}, it suffices to check the relations on~$\varOmega^U$, the bordism of point.
Recall that $\partial_*[M]$ is represented by a submanifold dual to $c_1(M)$, which is an $SU$-manifold. Therefore, $(\varDelta_{(k_1, k_2)})_* \partial_* = 0$. In particular $\partial_*^2 = \varDelta_* \partial_* = 0$.

The identity $\varDelta_* \varPsi_* =\id$ is proved in~\cite[Theorem~8.1]{co-fl66m}. The identity $\partial_*\varPsi_*=0$ is stated in~\cite[Theorem~8.2]{co-fl66m}, but its proof contains an inaccuracy in the calculation of characteristic classes. We give a correct argument below.

Take $[M^{2n}]\in\varOmega_{2n}^U$. Then $\varPsi_*[M^{2n}]$ is represented by the manifold $\C P(\overline{\det\mathcal TM}\oplus\underline{\C}^2)$ with the stably complex structure given by the isomorphism 
\[
   \mathcal T \C P(\overline{\det\mathcal TM}\oplus \underline{\C}^2)\oplus\underline\R^2 
   \cong p^*\mathcal TM\oplus(\bar\eta\otimes p^*\overline{\det\mathcal TM})
   \oplus\bar\eta\oplus\eta.
\]
We denote this stably complex manifold by $P^{2n+4}$. Now, $\partial_*\varPsi_*[M^{2n}]=\partial_*[P^{2n+4}]$ is represented by a submanifold $N^{2n+2}\subset P^{2n+4}$ dual to $c_1(P^{2n+4})=c_1(\bar\eta)$. We can take as $N^{2n+2}$ the submanifold $\C P(\overline{\det\mathcal TM}\oplus\C)$ with the stably complex structure given by the isomorphism 
\[
   \mathcal T \C P(\overline{\det\mathcal TM}\oplus \underline{\C})\oplus\underline\R^2 
   \cong p^*\mathcal TM\oplus(\bar\eta\otimes p^*\overline{\det\mathcal TM})
   \oplus\eta.
\]
Note that $[N^{2n+2}]$ is precisely $(\varPsi_{(0,1)})_*[M^{2n}]$. To see that $N^{2n+2}$ is null-bordant, we calculate its total Chern class. We denote $c_i=c_i(M)$, $d=c_1(\bar\eta)$, then we have a relation $d^2=p^*c_1\cdot d$. Now we calculate
\begin{multline*}
  c(N^{2n+2})=(1+p^*c_1+\cdots+p^*c_n)(1+d-p^*c_1)(1-d)\\
  =(1+p^*c_1+\cdots+p^*c_n)(1-p^*c_1)\\
  =1+p^*(c_2-c_1^2)+p^*(c_3-c_1c_2)+\cdots+p^*(c_n-c_1c_{n-1})
\end{multline*}
(this calculation was performed incorrectly in~\cite[pp.~36--37]{co-fl66m}). Hence, $c_\omega(N^{2n+2})=p^*c'_\omega(M^{2n})$, where $c'_i=c_i-c_1c_{i-1}$, and all characteristic numbers $c_\omega[N^{2n+2}]$ vanish for dimensional reasons.

The identity $\partial\varPsi=\varPsi_{0,1}=0$ can also be obtained geometrically, by observing that the stably complex structure on $N^{2n+2}$ restricts to a trivial stably complex structure on each fibre $\C P^1=S^2$ of the projectivisation, so it extends over the associated $3$-disk bundle.

To verify the identity~$\chi_*\partial_*=[\C P^1]\partial_*$, observe that $\partial_*[M^{2n}]=[Y^{2n-2}]$ where $Y^{2n-2}$ is an $SU$-manifold, so that $\det \mathcal T Y$ is trivial. Then
$\chi_*\partial_*[M^{2n}]$ is represented by $\C P (\overline {\det \mathcal T Y} \oplus\underline \C)=\C P^1\times Y$, which implies the required identity.

The last identity is obtained by applying $\partial_*$ to the both sides of $\chi_*\partial_*=[\C P^1]\partial_*$. In the notation of the previous paragraph, we need to verify that $\partial_*(\C P^1\times Y)=2Y$, which follows by observing that $2Y\subset\C P^1\times Y$ represents the homology class dual to $c_1(\C P^1 \times Y) = c_1(\C P^1) \otimes 1$.
\end{proof}

\begin{remark}
In~\cite[\S5]{novi67}, the identity $[\partial,\chi]=2$ is asserted instead of $\partial\chi\partial=2\partial$. However, $[\partial,\chi]=2$ cannot hold. Indeed, applying $\partial$ from the right we get $\partial\chi\partial=2\partial$, and applying $\partial$ from the left we get $-\partial\chi\partial=2\partial$,  which implies $\partial=0$. On the other hand, $\partial[\C P^1]=2$.
\end{remark}

\begin{corollary}\label{alg1}
If a relation $a \partial + b \varDelta = 0$ holds for some $a,b\in A^U$, then $b = 0$.
\end{corollary}
\begin{proof}
Applying $\varPsi$ from the right to the relation, we get $b=0$.
\end{proof}

Now we can formulate the key result about $U^*(MSU)$, which will be used in the calculation of the corresponding Adams--Novikov spectral sequence.

\begin{theorem}[{\cite[Theorem~6.1]{novi67}}]\label{umsu2}
\ 
\begin{itemize}

\item[(a)] The left $A^U$-module $U^*(MSU)$ is isomorphic to $A^U/(A^U \varDelta + A^U \partial)$. The kernel of the natural homomorphism $A^U=U^*(MU)\to U^*(MSU)$ is identified with
$A^U \varDelta + A^U \partial$.

\item[(b)] The left annihilator of $\partial$ is equal to $A^U \varDelta + A^U \partial$.
\end{itemize}
\end{theorem}

\begin{proof}
The original proof in~\cite{novi67} is quite sketchy. Filling in the details required lots of technical work. The proof consists of three parts.

\medskip 

I. We show that $\widetilde \partial ( U_*(BU) ) = U_*(BSU)$. In other words, a bordism class $[X,\xi]\in U_m(BU)$ lies in the image of $\widetilde\partial$ if and only if represented by a pair $(X,\xi)$ where $\xi$ is an $SU$-bundle, i.\,e. $c_1(\xi)=0$. 

To prove the inclusion $\widetilde \partial ( U_*(BU) ) \supset U_*(BSU)$, take $[X,\xi]\in U_m(BU)$ with $c_1(\xi)=0$. Consider the bordism class $[X\times\C P^1,\xi\times\eta]\in U_{m+2}(BU)$, where $\eta$ is the tautological line bundle over~$\C P^1$. By the definition of $\widetilde\partial$ (Construction~\ref{atilde}), $\widetilde\partial[X\times\C P^1,\xi\times\eta]=[Y,\zeta]$, where $Y\subset X\times\C P^1$ is a codimension-$2$ submanifold dual to $c_1(\xi\times\eta)=1\otimes c_1(\eta)$, so we can take $Y=X$, and 
\[
  \zeta=\xi\times\eta|_X+\mathcal T(X\times\C P^1)|_X-\mathcal T X=\xi
\]
as stable bundles. Therefore, $[X,\xi]=\widetilde\partial[X\times\C P^1,\xi\times\eta]$. 

To prove the inclusion $\widetilde \partial ( U_*(BU) ) \subset U_*(BSU)$,
take $[Y, \zeta] = \widetilde \partial[X, \xi]$. We need to show that $\zeta$ is represented by an $SU$-bundle. By Construction~\ref{atilde},
\[
  \widetilde\partial [X, \xi]=[Y,\xi|_{Y}+\mathcal TX|_Y-\mathcal T Y]\in U_{m-2}(BU),
\]
where $Y\subset X$ is a codimension-$2$ submanifold with the normal bundle $\nu(Y\subset X)=
\overline{\det\xi}|_Y$. Then 
\[
  c_1(\zeta)=c_1(\xi|_{Y}+\mathcal TX|_Y-\mathcal T Y)
  =c_1(\xi|_{Y})+c_1(\nu)= c_1(\det\xi|_{Y})+
  c_1(\overline{\det\xi}|_Y)=0,
\]
so $\zeta$ is an $SU$-bundle.

\medskip

II. We show that $\mathrm{Ann}_L\partial=\varphi^*(I(\du))$, where $\mathrm{Ann}_L$ denotes the left annihilator of $\partial$ in~$A^U$. 
Let $a \partial =0$ for some $a\in A^U$. Then $\widetilde a \widetilde \partial = 0$, which is equivalent by part~I to $\widetilde a |_{U_*(BSU)} = 0$. In other words, $\widetilde a[X, \xi] = [Y_a, f_a^*(\xi + \mathcal TX) -\mathcal TY_a] = 0$ for any $SU$-bundle~$\xi$. In particular $[Y_a] = 0$ in $\varOmega_U$. On the other hand, $[Y_a]=\langle(\varphi^*)^{-1}a,
[X,\xi]\rangle$ by~\eqref{kron1}. It follows that $(\varphi^*)^{-1}a\in U^*(BU)=\Hom_{\varOmega^U}(U_*(BU),\varOmega^U)$ lies in the ideal $I(\du)$, because the latter consists precisely of homomorphisms $U_*(BU)\to\varOmega^U$ vanishing on bordism classes of $SU$-bundles. Thus, $a \in \varphi^*(I(\du))$ and $\mathrm{Ann}_L(\partial) \subset \varphi^*(I(\du))$. For the opposite inclusion, note that $a\in\varphi^* (I(\du))$ implies that $\widetilde a |_{U_*(BSU)} = 0$. By Part~I, $\widetilde a \widetilde \partial = 0$. Now, Lemma~\ref{atilfaith} gives $a\partial=0$, so $a\in \mathrm{Ann}_L(\partial)$.

\medskip

III. We show that $\varphi^*(I(\du))=A^U \varDelta + A^U \partial$.

Corollary~\ref{alg1} implies that $A^U \varDelta + A^U \partial$ is a direct sum, so we write it as $A^U \varDelta \oplus A^U \partial$.
 
Lemma~\ref{algrel} and Part~II give the inclusion $A^U \varDelta \oplus A^U\partial\subset\mathrm{Ann}_L\partial=\varphi^*(I(\du))$.
Consider the short exact sequence
\begin{equation}\label{exactN}
  0 \longrightarrow A^U \varDelta \oplus A^U \partial 
  \stackrel i\longrightarrow \varphi^* (I(\du)) 
  \longrightarrow \varphi^* (I(\du))/(A^U \varDelta \oplus A^U \partial) \longrightarrow 0
\end{equation}
of graded $\varOmega_U$-modules. Denote
\[
  N=\varphi^* (I(\du))/(A^U \varDelta \oplus A^U \partial)
\] 
We need to show that $N=0$.

First, we show that $N$ has no $\varOmega_U$-torsion. Suppose $\lambda n=0$ for a nonzero $\lambda\in~\varOmega_U$ and $n=x+(A^U \varDelta + A^U \partial)\in N$, $x \in \varphi^* (I(\du))$. That is, $\lambda x = a \varDelta + b \partial$ for some $a,b\in A^U$. Multiplying by $\varPsi$  from the right and using Proposition~\ref{algrel} we obtain $a = \lambda x \varPsi$ and $b \partial = \lambda x - \lambda x \varPsi \varDelta = \lambda y$. Therefore, $\widetilde b \widetilde \partial = \widetilde \lambda \widetilde y$. Now, for a bordism class $[Y, \zeta]\in U_*(BSU)$ we have
\[
  \bigl\langle(\varphi^*)^{-1}b,[Y, \zeta]\bigr\rangle = 
  \bigl\langle(\varphi^*)^{-1}b,\widetilde \partial[X, \xi]\bigr\rangle = 
  \varepsilon(\widetilde \lambda \widetilde y\,[X, \xi])=
  \lambda\varepsilon(\widetilde y\,[X, \xi]),
\]
where the first identity follows from part I, and the second from~\eqref{kron1}. Consider the natural projection $p \colon U^*(\BU) \to U^*(\BSU)$, which is Kronecker dual to the natural inclusion $U_*(\BSU) \hookrightarrow U_*(\BU)$. Then the above identity implies that $p((\varphi^*)^{-1}b)=\lambda w$ for some $w\in U^*(\BSU)$. We have $w = p(t)$ for some $t\in U^*(\BU)$, hence, $p((\varphi^*)^{-1}b - \lambda t) = 0$ and we obtain that $(\varphi^*)^{-1}b - \lambda t \in \Ker p = I(\du)$. Hence, $b-\lambda \varphi^*(t) \in \varphi^*(I(\du))$ and $b \partial = \lambda \varphi^*(t) \partial$ by part~II. It follows that $\lambda x = a \varDelta + b \partial = \lambda (x \varPsi \varDelta + \varphi^*(t) \partial)$. Since $A^U$ has no $\varOmega_U$-torsion, we conclude that $x = x \varPsi \varDelta + \varphi^*(t) \partial \in A^U \varDelta \oplus A^U \partial$ and therefore $n=0$.

Now consider the following $A^U$-linear maps:
\begin{align*}
\pd \colon& A^U \to A^U \varDelta,&\quad p_\partial \colon& A^U \to A^U \partial,\\
&\ a\mapsto 2 a \varPsi \varDelta, &\quad
&\ a \mapsto a (1 - \varPsi \varDelta) \chi \partial.
\end{align*}

These maps behave like mutually orthogonal projections. Namely, they satisfy the identities
\[
\pd |_{A^U\!\varDelta} = 2\, \mathrm {id}_{A^U\! \varDelta}, \qquad
\pd |_{A^U\partial} = 0,\qquad
p_\partial |_{A^U \partial} = 2\, \mathrm {id}_{A^U \partial},\qquad
p_\partial |_{A^U\!\varDelta} = 0.
\]
This is a direct calculation using Proposition~\ref{algrel}:
\begin{gather*}
  \pd (a \varDelta) = 2 a \varDelta\, \varPsi \varDelta = 2 a \varDelta, \qquad
  \pd (b \partial) = 2 b \partial\, \varPsi \varDelta = 0,\\ 
  p_\partial (a \varDelta) = a \varDelta (1 - \varPsi \varDelta) \chi \partial = (a \varDelta -   
  a \varDelta\, \varPsi \varDelta) \chi \partial = 0,\\ 
  p_\partial (b \partial) = b \partial (1 - \varPsi \varDelta) \chi \partial = ( b \partial - b   
  \partial\, \varPsi \varDelta) \chi \partial = b \partial \chi \partial = 2 b \partial.
\end{gather*}
We therefore have an $A^U$-linear map $p =  \pd\!+ p_\partial \colon A^U \to A^U\!\varDelta \oplus A^U \partial$ satisfying $p|_{A^U\!\varDelta \oplus A^U \partial} = 2 \mathop\mathrm{id}_{A^U\!\varDelta \oplus A^U \partial}$. We use the following algebraic fact. 

\begin{lemma}\label{semis}
Let\, $0 \to A \xrightarrow{i} B \xrightarrow{\pi} C \to 0$ be an exact sequence of abelian groups. Assume $A$ does not have $n$-torsion for a fixed $n\in\Z$ and there exists a homomorphism $p \colon B \to A$ satisfying $p \circ i = n \id_A$. Then there exists an injective homomorphism $s \colon nC \hookrightarrow B$. 

If we start with a short exact sequence of $R$-modules for a commutative ring~$R$, then $s$ is also an $R$-module map.
\end{lemma}
\begin{proof}
Let $nc \in nC$. If $nc = \pi (nb)$ then $nc = \pi (nb - i (p (b)))$ and $p(nb - i (p(b))) = n p(b) - n p(b) = 0$. Hence, there is an element $x:=nb-i(p(b)) \in B$ satisfying $\pi(x)=nc$ and $p(x)=0$. If $x'$ is another such element, then $\pi(x-x')=0$ so $x-x'=i(y)$ and $0=p(x-x')=p(i(y)) = n y$. Since $A$ has no $n$-torsion, $y = 0$ and $x=x'$. Hence, $x$ is defined uniquely and there is a well defined homomorphism $s \colon nC \to B$, $nc\mapsto x$, satisfying $p \circ s = 0$ and $\pi \circ s = \id_{nC}$. The latter identity implies that $s$ is injective. 
\end{proof}

Applying Lemma~\ref{semis} to the short exact sequence~\eqref{exactN} and $p =  \pd\!+ p_\partial$ restricted to $\varphi^*I(\du)$,  we conclude that $2N$ injects into $\varphi^*I(\du) \subset A^U$. 
Since $N$ has no $2$-torsion, $N$ itself also injects into $\varphi^*I(\du) \subset A^U$.
Furthermore, applying $\otimes_{\varOmega_U}\Z$ to~\eqref{exactN}, we obtain a short exact sequence of graded abelian groups
\begin{equation}\label{sexact}
  0\to\bigl((A^U \varDelta) \otimes_{\varOmega_U}\Z\bigr) \oplus \bigl((A^U      
  \partial)\otimes_{\varOmega_U}\Z\bigr) 
  \stackrel{i\otimes_{\varOmega_U}\Z}{\lllra} \varphi^* (I(\du))
  \otimes_{\varOmega_U}\Z \to N \otimes_{\varOmega_U}\Z \to 0.
\end{equation}
The injectivity of the second map follows from the identity $(p\otimes_{\varOmega_U}\Z)(i\otimes_{\varOmega_U}\Z)=2\id$ and the absence of torsion in $((A^U \varDelta) \otimes_{\varOmega_U}\Z\bigr) \oplus \bigl((A^U\partial)\otimes_{\varOmega_U}\Z)$ (the latter group is described below).
Note that $M \otimes_{\varOmega_U}\Z = M/(\varOmega^+_U M)$ for any $\varOmega_U$-module $M$, where $\varOmega^+_U$ denotes the ideal of nonzero (negatively) graded elements in~$\varOmega_U$.

\smallskip

Next, we show that $N \otimes_{\varOmega_U}\Z$ is finite in each degree using a dimension counting argument.

As $\varDelta$ has the right inverse $\varPsi$, the $A^U$-module $A^U\varDelta$ is free on a single $4$-dimensional generator. That is, $(A^U \varDelta)^{2k} = U^{2k-4}(MU)$. Hence, 
\[
  ((A^U \varDelta)\otimes_{\varOmega_U}\Z)^{2k} = (U^{*-4} 
  (MU)\otimes_{\varOmega_U}\Z)^{2k}=H^{2k-4}(MU; \Z)\cong
   \Z^{\vphantom{\bigl)}p(k-2)},
\]
where $p(k)$ denotes the number of integer partitions of~$k$. Furthermore, 
\begin{multline*}
  (A^U \partial)^{2k} = (A^U)^{2k-2} \partial   
  \cong(A^U)^{2k-2}/(\mathop{\mathrm{Ann}}\nolimits_L\partial)^{2k-2} 
  \\= (A^U)^{2k-2}/(\varphi^*I(\du))^{2k-2} = U^{2k-2}(MSU),
\end{multline*}
where the third identity follows from part~II of this proof, and the last one is~\eqref{umsu1}. It follows that 
\[
  ((A^U \partial)\otimes_{\varOmega_U}\Z)^{2k} \cong
  H^{2k-2}(\MSU; \Z)=\Z^{\vphantom{\bigl)}\widetilde p(k-1)},
\]
where $\widetilde p(k)$ is a number of integer partitions of $k$ without~$1$. 
Finally, $(\varphi^*I(\du))\otimes_{\varOmega_U}\Z = \varphi^*_H I(c_1)$, where $\varphi^*_H \colon H^*(BU;\Z) \to H^*(MU; \Z)$ is the Thom isomorphism in ordinary cohomology and $I(c_1)$ is the ideal in $H^*(BU; \Z)$ generated by the universal first Chern class~$c_1$. Therefore,
\[
  ((\varphi^*I(\du))\otimes_{\varOmega_U}\Z)^{2k} = (\varphi^*_H I(c_1))^{2k} = 
  \Z^{\vphantom{\bigl)}p(k-1)}.
\]
Plugging the identities above into the $(2k)$th homogeneous part of~\eqref{sexact} we obtain
\[
  0 \to \Z^{\vphantom{\bigl)}p(k-2)+\widetilde p(k-1)} \to 
  \Z^{\vphantom{\bigl)}p(k-1)} \to (N \otimes_{\varOmega_U}\Z)^{2k} \to 0.
\]
Now the identity $p(k-1) = p(k-2) + \widetilde p(k-1)$ implies that $(N \otimes_{\varOmega_U}\Z)^{2k}$ is a finite group.

We therefore have a graded $\varOmega_U$-submodule $N$ of $A^U$ such that $(N \otimes_{\varOmega_U}\Z)^{2k}$ is a finite group for any~$k$. We need to show that $N=0$. Consider the $\varOmega_U$-linear projection $p_{\omega} \colon A^U \to \varOmega_U$ which maps $a \in A^U$ to its coefficient $\lambda_\omega$ in the power series expansion $a = \sum_{\omega} \lambda_{\omega} S_{\omega}$, where $S_\omega\in A^U$ are the Landweber--Novikov operations.
As $N \otimes_{\varOmega_U}\Z=N/(\varOmega^+_U N)$ is finite in each dimension, we obtain that $p_\omega(N)/(\varOmega^+_U p_\omega(N))$ is also finite in each dimension. We claim that $p_\omega(N)=0$. The general algebraic setting is as follows. Let $R$ be a nonnegatively (or nonpositively) graded ring without torsion, and let $I \subset R$ be an ideal such that $I / (R^+ I)$ is finite in each dimension. Then $I=0$. Indeed, let $x \in I$ be an element of minimal degree. Then $nx \in R^+ I$ for some nonzero integer~$n$. As $\deg x$ is minimal in $I$, every nonzero element of $R^+ I$ has degree greater then~$\deg x$. Hence, $n x = 0$. As $R$ has no torsion, we conclude that $x=0$ and $I=0$. 
Returning to our situation, we obtain that $p_\omega(N)=0$ for any $\omega$. Thus, $N=0$ as claimed.

We have therefore proved that $\varphi^*(I(\du))=A^U \varDelta + A^U \partial$. Combining this identity with~\eqref{umsu1} we obtain statement~(a) of the theorem, and combining it with the  identity of part~II of the proof, we obtain that $\mathrm{Ann}_L\partial=A^U \varDelta + A^U \partial$, proving statement~(b).
\end{proof}

\section{Calculation with the spectral sequence}\label{calcsect}
Here we apply the Adams--Novikov spectral sequence (Theorem~\ref{ANth}) to the $\SU$-bordism spectrum $X=\MSU$. As a result, we obtain a multiplicative spectral sequence with the $E_2$-term
\[
  E^{p, q}_2 = \Ext^{p,q}_{A^U}(U^*(MSU),U^*(\pt)),
\]
converging to $\pi_*(\MSU)=\varOmega^{SU}_*$.

Theorem~\ref{umsu2} implies that there is a free resolution of left $A^U$-modules:
\[
  0\longleftarrow
  U^*(\MSU)\cong A^U/(A^U \partial + A^U \varDelta) {\longleftarrow} A^U \stackrel{f_0}  
  {\longleftarrow} A^U \oplus A^U \stackrel{f_1}{\longleftarrow} A^U \oplus A^U 
  \stackrel{f_2}{\longleftarrow} \ldots
\]
where $A^U\to A^U/(A^U \partial + A^U)$ is the quotient projection, $f_0(a, b) = a \partial + b \varDelta $ and $f_i(a, b)  = (a \partial  + b \varDelta, 0)$ for $i\ge1$. We rewrite it more formally as follows:

\begin{proposition}\label{aures}
There is a free resolution of left  $A^U$-modules:
\[
  0\longleftarrow
  U^*(\MSU){\longleftarrow} R^0 \stackrel{f_0}  
  {\longleftarrow} R^1 \stackrel{f_1}{\longleftarrow} R^2 \stackrel{f_2}  
  {\longleftarrow}\ldots
\]
where $R^0=A^U\langle u_0\rangle$ is a free module on a single generator of degree~$0$, $R^i=A^U\langle u_i,v_i\rangle$ is a free module on two generators, $\deg u_i=2i$, $\deg v_i=2i+2$, $i\ge1$, and $f_{i-1}(u_i)=\partial u_{i-1}$, $f_{i-1}(v_i)=\varDelta u_{i-1}$.
\end{proposition}
\begin{proof}
We have $f_{i-1}f_{i}=0$ because $\partial^2=\varDelta\partial=0$. The exactness at $R^0$ is Theorem~\ref{umsu2}. To prove the exactness at $R^i$ with $i\ge1$, suppose $0=f_{i-1}(au_i+bv_i)=(a\partial+b\varDelta)u_{i-1}$. Then $a\partial+b\varDelta=0$, which implies $b=0$ and $a\partial=0$ by Corollary~\ref{alg1}. Hence, $a\in\mathrm{Ann}_L\partial$, so $a=a'\partial+b'\varDelta$ by Theorem~\ref{umsu2}~(b). Thus, $au_i+bv_i=au_i=f_{i}(a'u_{i+1}+b'v_{i+1})$, as needed.
\end{proof}

Applying $\mathop\mathrm{Hom}_{A^U}^q(-,U^*(pt))$ to the resolution of Proposition~\ref{aures} and using the isomorphism $\varOmega_U^{-q}=\varOmega^U_q$, we obtain a complex whose homology is the terms $E^{*,q}_2$ of the spectral sequence:
\begin{equation}\label{e2com}
  0\longrightarrow 
  \varOmega_{q}^U\stackrel{d^0}{\longrightarrow} \varOmega^U_{q-2} \oplus 
  \varOmega^U_{q-4}\stackrel{d^1}{\longrightarrow} \varOmega^U_{q-4} \oplus
  \varOmega^U_{q-6}\stackrel{d^2}{\longrightarrow} \ldots
\end{equation}
The differentials are given by $d^0(a)=(\partial a,\varDelta a)$ and $d^i(a,b)=(\partial a,\varDelta a)$, $i\ge1$. Here we denote by $\partial$ and $\varDelta$ the action of the corresponding operations on~$\OU$, and continue using this notation below.

Conner and Floyd~\cite{co-fl66m} defined the groups
\[
\mathcal W_q=\Ker(\varDelta\colon\varOmega_{q}^U\to\varOmega_{q-4}^U).
\]
The identities $\partial^2=\varDelta\partial=0$ imply that the restriction of the differential $\partial \colon \mathcal W_k \to \mathcal W_{k-2}$ is defined.

\begin{proposition}\label{wcom}
The complex~\eqref{e2com} is quasi-isomorphic to its subcomplex
\[
  0\longrightarrow\mathcal W_q \stackrel{\partial}{\longrightarrow} 
  \mathcal W_{q-2} \stackrel{\partial}{\longrightarrow} 
  \mathcal W_{q-4} \stackrel{\partial}{\longrightarrow}\cdots.
\]
\end{proposition}
\begin{proof}
Let $i\colon\mathcal W_k\to\varOmega^U_k\oplus\varOmega^U_{k-2}$ be the inclusion $w\mapsto (w,0)$, where $w\in\Ker\varDelta$. It is a map of chain complexes, because $i(\partial w)=(\partial w,0)=(\partial w,\varDelta w)=d(w,0)=di(w)$. The induced map in homology is injective, because $i(w)=d(a,b)$ implies $(w,0)=(\partial a,\varDelta a)$, hence $w=\partial a$ with $a\in\Ker\varDelta=\mathcal W_*$. To prove the surjectivity, take a cycle $(a,b)\in\varOmega^U_k\oplus\varOmega^U_{k-2}$. Then $0=d(a,b)=(\partial a,\varDelta a)$. Since $\varDelta\colon\varOmega_{k+2}^U\to\varOmega_{k-2}^U$ is surjective (it has a right inverse~$\varPsi$), there is $b'\in \varOmega_{k+2}^U$ such that $\varDelta b'=b$. Then $a-\partial b'\in\Ker\varDelta$ is a $\partial$-cycle, and $(a,b)-i(a-\partial b')=(a,b)-(a-\partial b',0)=(\partial b',b)=d(b',0)$, so $i(a-\partial b')$ represents the same homology class as $(a,b)$.
\end{proof}

\begin{proposition}\label{core2}
The $E_2$-term of the spectral sequence satisfies
\begin{itemize}
\item[(a)] $E^{0, q}_2 = \Ker(\partial\colon\mathcal W_q\to\mathcal W_{q-2})= (\Ker \partial)\cap(\Ker\varDelta)\subset \varOmega^U_q$;
\item[(b)] $E^{p, q}_2 = H_{q-2p}(\mathcal W_*, \partial)$ for $p>0$.
\item[(c)] the edge homomorphism $h\colon \OSU_q\to E_2^{0,q}$ coincides with the forgetful ho\-mo\-mor\-phism $\varOmega^{SU}_q\to\mathcal W_q$.
\end{itemize}
Therefore, the spectral sequence is concentrated in the first quadrant (i.\,e., $E_r^{p, q}=0$ for $p<0$ or $q<0$), $E_r^{p,q}=0$ for odd $q$ and for $q<2p$, and the differentials $d_r\colon E^{p, q}_r \to E^{p+r, q+r-1}_r$ are trivial for even~$r$.
\end{proposition}
\begin{proof}
Statements~(a) and (b) follow from Proposition~\ref{wcom}. To prove~(c), recall that the edge homomorphism 
\[
  h \colon \OSU_q\to E^{0,q}_2=\Hom_{A^U}^q(U^*(\MSU),\varOmega_U)
\]
is defined as follows. Given an element $\alpha \in\OSU_q$ represented by a map $f \colon S^q \to\MSU$ and an element $\beta \in U^p(\MSU)$ represented by a map $g \colon\MSU \to \varSigma^p \MU$, the element $h (\alpha)(\beta) \in\varOmega_U^{p-q}$ is represented by the composite $g\circ f\colon S^q\to\varSigma^p\MU$. Through the identification of $E^{0,q}_2$ with $\Ker(\partial\colon\mathcal W_q\to\mathcal W_{q-2})$, an $A^U$-homomorphism $\varphi\colon U^*(\MSU)\to\varOmega_U^{*-q}$ is mapped to $\varphi(\iota)$, where $\iota\in U^0(\MSU)$ is the class represented by the canonical map of spectra $\MSU\to\MU$. The edge homomorphism therefore becomes $\OSU_q\to\OU_q$, $\alpha\mapsto h(\alpha)(\iota)$, which is precisely the forgetful homomorphism, proving~(c). The rest follows from the fact that $\mathcal W_*$ is concentrated in nonnegative even degrees.
\end{proof}

In particular, $d_2=0$ and $E_2=E_3$. We shall denote this term simply by~$E$.

We have $E^{1, 2} = H_0(\mathcal W_*,\partial) = \mathbb Z_2$, because $\mathcal W_0 = \varOmega^U_0 = \mathbb Z$, $\mathcal W_2 = \varOmega^U_2 = \mathbb Z$ generated by $[\mathbb C P^1]$, and $\partial [\mathbb C P^1] = 2$.
Let $\theta \in E^{1, 2}$ be the generator. By dimensional reasons, it is an infinite cycle, because it lies on the `border line' $q=2p$. 

\begin{proposition}\label{multheta}
The multiplication by $\theta$ defines an isomorphism $E^{p, q} \to E^{p+1, q+2}$ for $p>0$ and an epimorphism $E^{0,q}\to E^{1,q+2}$ with kernel $\mathop{\mathrm{Im}}\partial$.
\end{proposition}
\begin{proof}
For $p>0$, the map $E^{p, q}\stackrel{\cdot \theta}\longrightarrow E^{p+1, q+2}$ is the identity isomorphism $H_{q-2p}(\mathcal W_*)\to H_{q-2p}(\mathcal W_*)$. For $p=0$, the homomorphism $E^{0,q}\to E^{1,q+2}$ maps $\Ker(\partial\colon \mathcal W_q\to\mathcal W_{q-2})$ to $H_q(\mathcal W_*)$, so its kernel is $\mathop{\mathrm{Im}}\partial$.
\end{proof}

This implies that $E^{p, q} = \theta E^{p-1, q-2}$ for $p\ge1$. In particular, $E^{k, 2k}=\Z_2$ generated by~$\theta^k$, so the only nontrivial elements on the border line $q=2p$ are $1,\theta,\theta^2,\theta^3,\ldots$

Now consider $E^{0,4}=\Ker(\partial\colon\mathcal W_4\to\mathcal W_2)$. Note that $\partial|_{\varOmega^U_4}=0$, because $c_1$ is the only Chern number in~$\varOmega_2^U$. Hence, $E^{0,4}=\mathcal W_4$. Furthermore, $\mathcal W_4\cong\Z$ is generated by 
\[
  K=9 [\C P^1]^2-8 [\C P^2]
\]
(this bordism class has characteristic numbers $c_1^2=0$ and $c_2=12$). Therefore, $K$ represents a generator of $E^{0,4}=\Z$.

We have a potentially nontrivial differential $d_3\colon E^{0,4}\to E^{3,6}$, see Figure~\ref{Eterm}.

\begin{figure}
\begin{picture}(40,64)
  \multiput(0,0)(8,0){2}{\line(0,1){64}}
  \multiput(24,0)(8,0){2}{\line(0,1){64}}
  \put(16,0){\line(0,1){41}}
  \put(16,46){\line(0,1){18}}
  \multiput(0,0)(0,8){8}{\line(1,0){40}}
  \put(3,3){$1$}
  \put(11,19){$\theta$}
  \put(2.5,35){$K$}
  \put(19,35){$\theta^2$}
  \put(2.5,51){$S^6$}
  \put(10,51){$\theta K$}
  \put(26,51){$\theta^3$}
  \put(6.5,37){\vector(4,3){19}}
  \put(15.5,42){\footnotesize $d_3$}
  \put(3,-3){\footnotesize $0$}
  \put(11,-3){\footnotesize $1$}
  \put(19,-3){\footnotesize $2$}
  \put(27,-3){\footnotesize $3$}
  \put(3,-3){\footnotesize $0$}
  \put(-3,3){\footnotesize $0$}
  \put(-3,11){\footnotesize $1$}
  \put(-3,19){\footnotesize $2$}
  \put(-3,27){\footnotesize $3$}
  \put(-3,35){\footnotesize $4$}
  \put(-3,43){\footnotesize $5$}
  \put(-3,51){\footnotesize $6$}
\end{picture}%
\medskip

\caption{The term $E_2=E_3$ of the Adams--Novikov spectral sequence for $SU$-bordism.}
\label{Eterm}
\end{figure}

\begin{proposition}
We have $d_3(K)=\theta^3$.
\end{proposition}
\begin{proof}
Suppose that $d_3(K) = 0$. We also have $d_i(K)=0$ for $i>3$, because $d_i(K)\in E_i^{i,i+3}$ is below the border line $p=2q$. This implies that $K$ is an infinite cycle, so it represents an element in $E^{0,4}_\infty$. We obtain that $E^{0,4}_2=E^{0,4}_{\vphantom{2}\infty}$, which implies that the edge homomorphism $\OSU_4\to E^{0,4}_2$ is surjective. It coincides with the forgetful homomorphism $\OSU_4\to\mathcal W_4$ by Proposition~\ref{core2}~(c). On the other hand, the forgetful homomorphism is not surjective, as $\mathop\mathrm{td}(K)=1$, while the Todd genus of a $4$-dimensional $SU$-manifold is even (this follows from Rokhlin's signature theorem~\cite{rohl52}). A contradiction.
\end{proof}

\begin{proposition}\label{3columns}
We have $E^{p, q}_4 = 0$ for $p \geqslant 3$ and $E_4 = E_{\infty}$.
\end{proposition}
\begin{proof}
Take a $d_3$-cycle $x \in E^{p, q}$ with $p\ge3$. We have $x =\theta^3y$ for some $y\in E^{p-3,q-6}$ and $0 = d_3x  = \theta^3d_3y$. Now, $d_3y \in E^{p,q-4}$, and the multiplication by $\theta^3$  is an isomorphism in this dimension by Proposition~\ref{multheta}, hence, $d_3y=0$. This implies that $x=\theta^3 y = d_3(K y)$. Hence, $x$ is a boundary, and
$E^{p, q}_4 = 0$ for $p\ge3$. For dimensional reasons, this implies $d_i=0$ for $i\ge 4$ and $E_{\infty} = E_4$.
\end{proof}

It follows that the infinite term of the spectral sequence consists of three columns only, and $E^{1, *}_{\infty} = \theta E^{0, *}_{\infty}$, $E^{2, *}_{\infty} = \theta E^{1, *}_{\infty}$. Furthermore, in the first three columns we have $E_{\infty}=\Ker d_3$, for dimensional reasons, and the multiplication by $\theta$ is injective on~$E^{1,*}_\infty$. In particular, $E^{k, 2k}_{\infty} = E^{k, 2k}$ is $\Z_2$ with generator $\theta^k$ for $0\le k\le2$, and $E^{k, 2k}_{\infty}=0$ for $k\ge3$.

Proposition~\ref{3columns} implies that the Adams--Novikov filtration in $\OSU$ satisfies $F^{p,q}=0$ for $p\ge3$, that is, the filtration consists of three terms only:
\[
  \varOmega^{SU}_n = F^{0, n} \supset F^{1, n+1} \supset F^{2, n+2} = 
  E^{2, n+2}_{\infty}.
\]

If $n=2k+1$ is odd, then $F^{0,2k+1}/F^{1,{2k+2}}=E_\infty^{0,2k+1}=0$ and $F^{2,2k+3}=E_\infty^{2,2k+3}=0$ by Proposition~\ref{core2}. Therefore
\begin{equation}\label{osuodd}
  \OSU_{2k+1} = E^{1, 2k+2}_{\infty}.
\end{equation}

If $n=2k$ is even, then $F^{1,2k+1}/F^{2,2k+2}=E_\infty^{1,2k+1}=0$, so we obtain a short exact sequence
\begin{equation}\label{osueven}
  0 \to  E^{2, 2k+2}_{\infty} \to \OSU_{2k} \to E^{0, 2k}_{\infty} \to 0.
\end{equation}

\begin{example}\label{lssuex}
In low dimensions we have:
\begin{itemize}
\item $\OSU_0=E^{0,0}_{\infty}=E^{0, 0} \cong \mathbb{Z}$, because $E^{2, 2}_{\infty} = 0$.

\smallskip

\item $\OSU_1 = E^{1, 2}_{\infty}=E^{1, 2}\cong\Z_2$ with generator $\theta$.

\smallskip

\item $\OSU_{2} = E^{2, 4}_{\infty}\cong\Z_2$ with generator $\theta^2$, because 
$0 = E^{0,2} = \Ker \partial \subset\mathcal W_2$ (recall that $\mathcal W_2$ is generated by $[\C P^1]$ and $\partial[\C P^2]=2$).

\smallskip

\item $\OSU_3 = E^{1, 4}_{\infty} = \theta E^{0, 2}_{\infty} = 0$.

\smallskip

\item $\OSU_{4} = E^{0, 4}_{\infty}\cong\Z$ with generator $2K$. The identity $\OSU_{4} = E^{0, 4}_{\infty}$ follows from~\eqref{osueven}, because $E^{2,6}_\infty=\theta^2E^{0,2}_\infty=0$. A generator of $E^{0,4}_\infty = \Ker d_3$ is $2K$, because ${d_3(K)=\theta^3}$.

\smallskip

\item $\OSU_5 = E^{1, 6}_{\infty} = \theta E^{0, 4}_{\infty}= 0$ because $\theta\cdot 2K=0$.
\end{itemize}
\end{example}

{\samepage
\begin{theorem}\label{osutors}\ 
\begin{itemize}
\item[(a)] The kernel of the forgetful homomorphism $\OSU\to\OU$ consists of torsion elements.

\item[(b)] Every torsion element in $\OSU$ has order~$2$. More precisely,
\[
  \OSU_{2k+1}=\theta\OSU_{2k},\quad\mathop\mathrm{Tors}\OSU_{2k}=\theta^2\OSU_{2k-2}.
\]  
\end{itemize}
\end{theorem}
}
\begin{proof}
We have $\OSU_{2k+1}=E^{1, 2k+2}_{\infty}=\theta E^{0, 2k}_{\infty}=\theta\OSU_{2k}$, because $\OSU_{2k}\to E^{0,2k}_\infty$ is surjective. This also implies that $\OSU_{2k+1}$ consists of $2$-torsion, proving (a) and (b) in odd dimensions.

In even dimensions, we use the exact sequence~\eqref{osueven}. Since $E^{0,2k}_\infty\subset E^{0,2k}\subset\mathcal W_*\subset\OU$ is torsion-free and $E^{2,2k+2}_\infty=\theta^2 E_\infty^{0,2k-2}$ is a $2$-torsion, we obtain $\mathop\mathrm{Tors}\OSU_{2k}=E^{2,2k+2}_\infty=\theta^2 E_\infty^{0,2k-2}=\theta^2\OSU_{2k-2}$, proving (b). To finish the proof of (a), it remains to note that the kernel of $\OSU\to\OU$ coincides with the kernel of $\OSU_{2k} \to E^{0, 2k}_{\infty}$ by Proposition~\ref{core2}~(c), and the latter kernel is the torsion of $\OSU_{2k}$ by the above.
\end{proof}

The next lemma gives a short exact sequence, originally due to Conner and Floyd~\cite{co-fl66m}, which is the key ingredient in the calculation of the torsion in~$\OSU$.

\begin{lemma}\label{cf3}
There is a short exact sequence of $\Z_2$-modules
\[
  0 \to \varOmega^{SU}_{2k-1} \to H_{2k-2}(\mathcal W_*,\partial) \to \varOmega^{SU}_{2k-5} \to 0.
\]
\end{lemma}

\begin{proof}Consider the commutative diagram
\[
\diagram
  0\rto &\OSU_{2k-1}=E^{1,2k}_\infty \rto& E^{1,2k}\rto^{d_3^{1,2k}} &
  E^{4,2k+2} \rto^{\ \ d_3^{4,2k+2}} & E^{7,2k+4}\\
  &\ \ \ \ \ \ \ \ \ \ \ \ \ \ \ \ 0\rto & \OSU_{2k-5} \rto & 
  E^{1,2k-4} \uto_{\cdot \theta^3}^\cong \rto^{\ \ d_3^{1,2k-4}} &
  E^{4,2k-2} \uto_{\cdot \theta^3}^\cong
\enddiagram
\]
The rows are exact by Proposition~\ref{3columns} and~\eqref{osuodd}. By the commutativity of the diagram, $\mathop\mathrm{Im} d_3^{1,2k}=
\Ker d_3^{4,2k+2}\cong\Ker d_3^{1,2k-4}=\OSU_{2k-5}$. We obtain a short exact sequence
\[
  0 \to \OSU_{2k-1} \to E^{1, 2k} \to \OSU_{2k-5} \to 0.
\]
It remains to note that $E^{1, 2k} = H_{2k-2}(\mathcal W_*,\partial)$. 
\end{proof}

\begin{remark}
The exact sequence of Lemma~\ref{cf3} is the derived exact sequence of the $5$-term exact sequence~\eqref{CF5seq} from the Introduction.
\end{remark}

Homology of $(\mathcal W_*,\partial)$ was described by Conner and Floyd. For the relation of this calculation to the Adams--Novikov spectral sequence, see~\cite[\S5]{botv90}.

\begin{theorem}[{\cite[Theorem~11.8]{co-fl66m}}]\label{cfcoh}
$H(\mathcal W_*,\partial)$ is the following polynomial algebra over~$\Z_2$:
\[
  H(\mathcal W_*,\partial)\cong\Z_2[\omega_2,\omega_{4k}\colon k\ge 2], \quad 
  \deg \omega_2=4,\;\deg \omega_{4k}=8k.
\]
\end{theorem}

\begin{remark}
The multiplication in $H(\mathcal W_*,\partial)$ is induced by the multiplication in~$\OU$, see Section~\ref{Wsect}. It coincides with the multiplication in the $E_2$ term of the Adams--Novikov spectral sequence.
\end{remark}

We finally obtain the following information about the free and torsion parts of~$\OSU$:

\begin{theorem}\label{freetors}\
\begin{itemize}
\item[(a)] $\mathop{\mathrm{Tors}}\OSU_n=0$ unless $n=8k+1$
or $8k+2$, in which case $\mathop{\mathrm{Tors}}\OSU_n$ is a
$\Z_2$-vector space of rank equal to the number of partitions
of~$k$.

\smallskip

\item[(b)] $\OSU_{2i}/\mathop{\mathrm{Tors}}$ is isomorphic to the image of the forgetful homomorphism $\alpha\colon\OSU_{2i}\to\OU_{2i}$, which is
$\Ker(\partial\colon\mathcal W_{2i}\to\mathcal W_{2i-2})$ if $2i\not\equiv
4\mod 8$ and
$\mathop{\mathrm{Im}}(\partial\colon\mathcal W_{2i}\to\mathcal W_{2i-2})$ if
$2i\equiv 4\mod 8$.

\smallskip

\item[(c)] There exist $SU$-bordism classes $w_{4k}\in\OSU_{8k}$, $k\ge1$, such that every torsion element of $\OSU$ is uniquely expressible in the
form $P\cdot\theta$ or $P\cdot\theta^2$ where $P$ is a polynomial
in $w_{4k}$ with coefficients $0$ or~$1$. An element $w_{4k}\in\OSU_{8k}$ is determined by the condition that it represents a polynomial generator $\omega_{4k}$ in $H_{8k}(\mathcal W_*,\partial)$ for $k\ge 2$, and $w_4\in\OSU_8$ represents~$\omega_2^2$.
\end{itemize}
\end{theorem}

\begin{remark}
The only indeterminacy in the definition of $w_{4k}$ is the choice of a $\partial$-cycle in $\mathcal W_{8k}$ representing a polynomial generator~$\omega_{4k}$ or $\omega_2^2$ from Theorem~\ref{cfcoh}. Once we fixed $w_{4k}\in\mathcal W_{8k}$, it lifts uniquely to $w_{4k}\in\OSU_{8k}$, since the forgetful homomorphism $\alpha\colon\OSU_{8k}\to\mathcal W_{8k}$ is injective onto $\Ker\partial$ in dimension~$8k$, by statements~(a) and~(b).
\end{remark}

\begin{proof}[Proof of Theorem~\ref{freetors}]
We prove (a). Theorem~\ref{cfcoh} gives that $H_{q-2p}(\mathcal W_*)=0$ unless $q-2p=8k$ or $q-2p=8k+4$. First consider the case of odd~$n$. Lemma~\ref{cf3} gives an exact sequence 
\[
  0\to\OSU_{8k-1}\to H_{8k-2}(\mathcal W_*)\to\OSU_{8k-5}\to0,
\]
which implies $\OSU_{8k-1}=\OSU_{8k-5}=0$. We also have an exact sequence 
\[
  0\to\OSU_{8k+1}\to H_{8k}(\mathcal W_*)\to\OSU_{8k-3}\to0,
\]
which splits because $H(\mathcal W_*)$ is a $\Z_2$-module. Hence, $\OSU_{8k+1}\oplus\OSU_{8k-3}\cong H_{8k}(\mathcal W_*)\cong H_{8k+4}(\mathcal W_*)\cong \OSU_{8k+5}\oplus\OSU_{8k+1}$. Hence, $\OSU_{8k-3}=\OSU_{8k+5}$. As this is valid for all $k$, we obtain $\OSU_{8k+5}=0$.
Therefore, the only nontrivial $\OSU_n$ with odd $n$ is $\OSU_{8k+1}$, and Lemma~\ref{cf3} gives an isomorphism $\OSU_{8k+1}\cong H_{8k}(\mathcal W_*)$. Now it follows from Theorem~\ref{cfcoh} that $\OSU_{8k+1}$ is a $\Z_2$-vector space of rank equal to the number of partitions of~$k$. 

For even $n=2m$,  Theorem~\ref{osutors} gives $\mathop\mathrm{Tors}\OSU_{2m}=\theta\OSU_{2m-1}$, which is nonzero only for $2m=8k+2$ by the previous paragraph. The multiplication by $\theta$ defines a homomorphism
\[
  \OSU_{8k+1}=E_\infty^{1,8k+2}\stackrel{\cdot\theta}\longrightarrow
   E_\infty^{2,8k+4}=\mathop{\mathrm{Tors}}\OSU_{8k+2},
\]
which is an isomorphism by Proposition~\ref{multheta}. This finishes the proof of~(a).

\medskip 

To prove (b), recall that $\mathop\mathrm{Tors}\OSU_q$ is the kernel of forgetful homomorphism $\OSU_q\to\mathcal W_q$ by Theorem~\ref{osutors}~(a), and the forgetful homomorphism coincides with the edge homomorphism $h\colon \OSU_q\to E^{0,q}_2$ by Proposition~\ref{core2}~(c). Hence, $\OSU/\mathrm{Tors}=\mathop\mathrm{Im} h$. Furthermore, $\mathop\mathrm{Im} h=\Ker(d_3\colon E^{0,*}_3\to E^{3,*+2})$ by Proposition~\ref{3columns}.

Now, if $2i\ne8k,8k+4$, then we have 
\[
  d_3(E^{0,2i})=\theta^{-1}d_3(\theta E^{0,2i})=\theta^{-1}d_3(E^{1,2i+2})=0
\]
because $E^{1,2i+2}=H_{2i}(\mathcal W_*)=0$ by Theorem~\ref{cfcoh}. Therefore, 
$\OSU_{2i}/\mathrm{Tors}=\Ker d_3=E^{0,2i}=\Ker\partial$ in this case.

For $2i=8k$, we observe that 
\[
  0=\OSU_{8k-3}=E_\infty^{1,8k-2}=\Ker d_3^{1,8k-2}\subset E^{1,8k-2}.
\]
This implies that 
\begin{equation}\label{injd3}
  0=\Ker( d_3^{1,8k-2}\theta^{-2})=\Ker(\theta^{-2}d_3^{3,8k+2})=\Ker d_3^{3,8k+2}.
\end{equation}
Hence, $\mathop\mathrm{Im} d^{0,8k}_3\subset \Ker d_3^{3,8k+2}=0$ and $\OSU_{8k}/\mathrm{Tors}=\Ker d^{0,8k}_3=E^{0,8k}=\Ker\partial$.

It remains to consider the case $2i=8k+4$. The exact sequence~\eqref{osueven} gives $\OSU_{8k+2}=E_\infty^{0,8k+4}$ because $E_\infty^{2,8k+6}\subset E^{2,8k+6}=H_{8k+2}(\mathcal W_*)=0$. Consider the commutative diagram with exact rows:
\[
\diagram
  0\rto &\OSU_{8k+4}=E^{0,8k+4}_\infty \rto&  
  E^{0,8k+4} \rto^{\ \ d_3^{0,8k+4}}\dto^{\cdot \theta^3} & 
  E^{3,8k+6}\dto^{\cdot \theta^3}_\cong\\
  &\ \ \ \ \ \ \ \ \ \ \ \ \ \ \ \ 0\rto & 
  E^{3,8k+10} \rto^{\ \ d_3^{3,8k+10}} & E^{6,8k+12}
\enddiagram
\]
The lower row is exact by~\eqref{injd3}. The diagram implies that
\[
  \OSU_{8k+4}\cong\Ker d_3^{0,8k+4}=
  \Ker(E^{0,8k+4}\stackrel{\cdot\theta^3}\longrightarrow E^{3,8k+10})=
  \Ker(E^{0,8k+4}\stackrel{\cdot\theta}\longrightarrow E^{1,8k+6})=
  \mathop\mathrm{Im}\partial,
\]
where the last two identities follow from Proposition~\ref{multheta}. This finishes the proof of~(b).

\smallskip

It remains to prove (c). Using statement (b) and Theorem~\ref{osutors}~(b) we identify the homomorphism $\OSU_{8k}\stackrel{\cdot\theta}\longrightarrow\OSU_{8k+1}$ with the projection $\Ker\partial\to\Ker\partial/\mathop\mathrm{Im}\partial=H_{8k}(\mathcal W_*)$. Take an element $\alpha\in\OSU_{8k+1}$ and write it as a polynomial $P(\omega_{4k})$  in~$\omega_{4k}$ with $\Z_2$-coefficients using Theorem~\ref{cfcoh}. (To simplify the notation, we use $\omega_2^2$ for the missing generator $\omega_4$ in this argument.) Choose lifts $w_{4k}\in\OSU_{8k}=\Ker\partial\subset\mathcal W_{4k}$ of $\omega_{4k}$; then $a=P(w_{4k})$ maps to~$\alpha$. In other words, $\alpha=P(w_{4k})\cdot\theta$, where $P$
is now considered as a polynomial with coefficients~$0$ and~$1$. If $\alpha=Q(w_{4k})\cdot\theta$ for another such~$Q$, then $P(\omega_{4k})=Q(\omega_{4k})$, which implies $P=Q$ because $\omega_{4k}$ are polynomial generators and both $P$ and $Q$ have coefficients $0$ and~$1$. Therefore, any element of $\OSU_{8k+1}$ is uniquely represented as $P\cdot\theta$, as needed. For the elements of $\mathop\mathrm{Tors}\OSU_{8k+2}$, recall that $\OSU_{8k+1}\stackrel{\cdot\theta}\longrightarrow\mathop\mathrm{Tors}\OSU_{8k+2}$ is an isomorphism. This finishes the proof.
\end{proof}

\section{The ring $\mathcal W$}\label{Wsect}
Theorem~\ref{freetors}~(b) relates the group $\OSU/\mathop{\mathrm{Tors}}$ to the subgroup $\Ker(\partial\colon\mathcal W\to\mathcal W)=(\Ker\partial)\cap(\Ker\varDelta)$ in $\OU$.
Although $\mathcal W=\Ker\varDelta$ is \emph{not} a subring of $\OU$, there is a product structure in~$\mathcal W$ such that $\OSU/\mathop{\mathrm{Tors}}\subset \mathcal W$ is a subring. This leads to a description of the ring structure in $\OSU/\mathop{\mathrm{Tors}}$. We review this approach here, following~\cite{co-fl66m}, \cite{wall66} and~\cite{ston68}.

We recall the geometric operations $\partial\colon\OU_{2n}\to\OU_{2n-2}$ and $\varDelta\colon\OU_{2n}\to\OU_{2n-4}$, see~\eqref{partialDelta}.

\begin{construction}[$\partial$ and $\varDelta$ revisited]\label{c1dual}
Consider a stably complex manifold $M=M^{2n}$ with the fundamental class $[M^{2n}]\in H_{2n}(M;\Z)$. Let $N=N^{2n-2}$ be a stably complex submanifold dual to the cohomology class $c_1(M)=c_1(\det\mathcal T M)$. That is, we have an inclusion
\[
  i\colon N^{2n-2}\hookrightarrow M^{2n}\quad\text{such that}\quad
  i_{*}([N])=c_{1}(M)\frown [M] \quad\text{in }H_*(M;\Z).
\]
The restriction of $\det\mathcal TM$ to $N$ is the normal bundle $\nu(N\subset M)$.
The stably complex structure on~$N$ is
defined via the isomorphism $\mathcal T M|_N\cong\mathcal T
N\oplus\nu(N\subset M)$. Then $c_1(N)=0$, so $N$ is an $SU$-manifold.

The homomorphism $\partial=\varDelta_{(1,0)}\colon\OU_{2n}\to\OU_{2n-2}$ sends a bordism class $[M]$ to the bordism class $[N]$ dual to $c_1(M)$ as described above. This operation is well defined on bordism classes, as $[N]=\varepsilon D_U(\cf_1(\det\mathcal T M))$, where $D_U\colon U^2(M)\to U_{2n-2}(M)$ is the Poincar\'e--Atiyah duality homomorphism, and $\varepsilon\colon U_{2n-2}(M)\to\OU_{2n-2}$ is the augmentation.
We have $\partial^2=0$ because $N$ is an $SU$-manifold. 

Similarly, the homomorphism $\varDelta=\varDelta_{(1,1)}\colon\OU_{2n}\to\OU_{2n-4}$ takes a bordism class $[M]$ to the bordism class of the submanifold $L=L^{2n-4}$ dual to 
$\det\mathcal TM\oplus\overline {\det\mathcal TM}$. That is, we have 
\[
  j\colon L^{2n-4}\hookrightarrow M^{2n}\quad\text{such that}\quad
  j_{*}([L])=-c^2_{1}(M)\frown [M] \quad\text{in }H_*(M;\Z).
\]

We also introduce the homomorphism $\partial_k=\varDelta_{(k,0)}\colon\OU_{2n}\to\OU_{2n-2k}$ taking a bordism class $[M]$ to the bordism class of the submanifold $[P]$ dual to $(\det\mathcal TM)^{\oplus k}$. We have $[P]=\varepsilon D_U(u^k)$, where $u=\cf_1(\det\mathcal T M)$.
\end{construction}

\begin{lemma}\label{dklemma}
Let $[M]\in\OU$ be a bordism class such that every Chern number of
$M$ of which $c_1^k$ is a factor vanishes. Then $\partial_k[M]=0$.
\end{lemma}
\begin{proof}
We have $\partial_k[M]=[P]$, where $j\colon P\hookrightarrow M$ is a submanifold such that 
\[
  \mathcal T P\oplus j^*(\det\mathcal TM)^{\oplus k}=
  j^*(\mathcal T M).
\]
Assume that $c_1^kc_\omega[M]=0$ for any $\omega$. We need to prove that $c_\omega[P]=0$. Calculating the Chern classes for the bundles above we get
\[
  c(P)(1+j^*c_1(M))^k=j^*c(M)
\]
or
\[
  c(P)=j^*\biggl(\frac{c(M)}{(1+c_1(M))^k}\biggr)=j^*\,\widetilde c(M),
\]
where $\widetilde c(M)$ is a polynomial in Chern classes of~$M$.
Then for any $\omega=(i_1,\ldots,i_p)$ we have
\[
  \bigl\langle c_\omega(P),[P]\bigr\rangle=
  \bigl\langle j^*\,\widetilde c_\omega(M),[P]\bigr\rangle=
  \bigl\langle \widetilde c_\omega(M),c_1^k(M)\frown[M]\bigr\rangle=
  \bigl\langle c_1^k\,\widetilde c_\omega(M),[M]\bigr\rangle=0.\qedhere
\]
\end{proof}

The group $\mathcal W_{2n}$ was defined as
\[
  \mathcal W_{2n}=\Ker(\varDelta\colon\varOmega_{2n}^U\to\varOmega_{2n-4}^U).
\]
The same group can also be defined in terms of characteristic numbers and geometrically, as described next. A cohomology class $x\in H^2(M)$ is \emph{spherical} if $x=f^*(u)$ for a map $f\colon M\to\C P^1$, where $u=c_1(\bar\eta)$ and $\eta$ is the tautological line bundle over~$\C P^1$.

\begin{theorem}\label{W3def}
The following three groups are identical:
\begin{itemize}
\item[(a)]
the group $\mathcal W=\Ker\varDelta$;

\item[(b)] 
the subgroup of $\OU$ consisting of bordism classes $[M]$ such that every Chern number of
$M$ of which $c_1^2$ is a factor vanishes;

\item[(c)]
the subgroup of $\OU$ consisting of bordism classes $[M]$ for which $c_1(M)$ is a spherical class.
\end{itemize}
\end{theorem}
\begin{proof}
The equivalence of (a) and (b) was proved in~\cite[(6.4)]{co-fl66m}. We give a more direct argument below. By definition, $\varDelta[M]=[L]$, where $j\colon L\hookrightarrow M$ is a submanifold such that 
\[
  \mathcal T L\oplus j^*\bigl(\det\mathcal TM\oplus\overline{\det\mathcal TM}\bigr)=
  j^*(\mathcal T M).
\]
Calculating the Chern classes, we get
\begin{gather*}
  c(L)(1+j^*c_1(M))(1-j^*c_1(M))=j^*c(M),\label{cML}\\
  c_i(L)-c_{i-2}(L)\cdot j^*c_1^2(M)=j^*c_i(M)\notag.
\end{gather*}
In particular, for $i=1$ we obtain $c_1(L)=j^*c_1(M)$, so we can rewrite the formula above~as
\[
  (c_i-c_1^2c_{i-2})(L)=j^*c_i(M).
\]
Given a partition $\omega=(i_1,\ldots,i_p)$ and the corresponding Chern class $c_\omega=c_{i_1}\cdots c_{i_p}$, we obtain the following relation on the characteristic numbers:
\[
  \bigl\langle (c_{i_1}-c_1^2c_{i_1-2})\cdots(c_{i_p}-c_1^2c_{i_p-2})(L),
  [L]\bigr\rangle=\bigl\langle j^*c_\omega(M),[L]\bigr\rangle=
  \bigl\langle -c_1^2c_\omega(M),[M]\bigr\rangle
\]
Now if $\varDelta[M]=[L]=0$, then the left hand side above vanishes, and we obtain from the right hand side that every Chern number of $M$ of which $c_1^2$ is a factor vanishes. 

For the opposite direction, assume that $-c_1^2c_\omega[M]=0$ for any $\omega$. We need to prove that $c_\omega[L]=0$. This is done in the same way as in the proof of Lemma~\ref{dklemma}. 

\smallskip

The equivalence of (a) and (c) is proved in~\cite[Chapter~VIII]{ston68}. 
\end{proof}

\begin{corollary}\label{deltadk}
If $[M]\in\mathcal W$, then $\partial_k[M]=0$ for any $k\ge2$.
\end{corollary}
\begin{proof}
By Theorem~\ref{W3def}, $[M]\in\mathcal W$ implies that every Chern number of
$M$ of which $c_1^2$ is a factor vanishes. Then every Chern number of
$M$ of which $c_1^k$ is a factor vanishes (as $k\ge2$). Thus, $\partial_k[M]=0$ by Lemma~\ref{dklemma}.
\end{proof}

\begin{remark}
For the operation $\partial=\partial_1$, there is no analogue of equivalence between (a) and (b) in Theorem~\ref{W3def}. More precisely, by Lemma~\ref{dklemma}, the group $\Ker\partial$ contains the subgroup of  $\OU$ consisting of bordism classes $[M]$ such that every Chern number of
$M$ of which $c_1$ is a factor vanishes. However, there is no opposite inclusion. For example, any element of $\OU_4$ is contained in $\Ker\partial$, but $c_1^2[\C P^2]\ne0$. In fact, the subgroup of  $\OU$ consisting of bordism classes $[M]$ such that every Chern number of
$M$ of which $c_1$ is a factor vanishes coincides with the intersection $\Ker\partial\cap\Ker\varDelta$.
\end{remark}

It follows from either of the descriptions of the group $\mathcal W_{2n}$ that we have forgetful
homomorphisms $\OSU_{2n}\to\mathcal
W_{2n}\to\OU_{2n}$, and the restriction of the boundary
homomorphism $\partial\colon\mathcal W_{2n}\to\mathcal W_{2n-2}$
is defined.

\begin{lemma}\label{Deltaab}
For any elements $a,b\in\mathcal W$, we have
\begin{gather*}
  \partial(a\cdot b)=a\cdot\partial b+\partial a\cdot b-[\C P^1]\cdot\partial a\cdot\partial b,\\
  \varDelta(a\cdot b)=-2\partial a\cdot\partial b,
\end{gather*}
where $a\cdot b$ denotes the product in~$\OU$.
\end{lemma}
\begin{proof}Let $a=[M^{2m}]$ and $b=[N^{2n}]$ for some stably complex manifolds $M$ and~$N$. Then $\partial(a\cdot b)\in\OU_{2m+2n-2}$ is represented by a submanifold $X\subset M\times N$ dual to $c_1(M\times N)=x+y$, where $x=p_1^*c_1(M)$, $y=p_2^*c_1(M)$ and $p_1\colon M\times N\to M$, $p_2\colon M\times N\to N$ are the  projection maps. Let $u, v\in U^2(M\times N)$ be the geometric cobordisms corresponding to $x,y$, respectively (see Construction~\ref{confgl}). Then we have
\[
  \partial(a\cdot b)=[X]=\varepsilon D_U(u+_{\!{}_H}\!v).
\]
On the other hand,
\[
  u+_{\!{}_H}\!v=F_U(u,v)=u+v+\sum_{k\ge1,\,l\ge1}\alpha_{kl}\,u^kv^l.
\]
To identify $\partial(a\cdot b)=[X]$, we apply $\varepsilon D_U$ to both sides of this identity. We have $\varepsilon D_U(u)=\partial a\cdot b$ (the submanifold dual to $p_1^*c_1(M)$ in $M\times N$ is the product of the submanifold dual to $c_1(M)$ in $M$ with $N$). Similarly, $\varepsilon D_U(v)=a\cdot\partial b$ and $\varepsilon D_U(uv)=\partial a\cdot\partial b$. We claim that $\varepsilon D_U(u^kv^l)=0$ if $k\ge2$ or $l\ge2$. Indeed, $\varepsilon D_U(u^kv^l)$ is the bordism class of the submanifold in $M\times N$ dual to $p_1^*(\det\mathcal T M)^{\oplus k}\oplus
p_2^*(\det\mathcal T N)^{\oplus l}$.
This bordism class is $\partial_k a\cdot\partial_lb$. Since $a,b\in\mathcal W$, Corollary~\ref{deltadk} implies that $\partial_k a=0$ or $\partial_lb=0$. 
%

The first identity of the lemma follows by noting that $\alpha_{11}=-[\C P^1]$ (see~\cite[Theorem~E.2.3]{bu-pa15}, for example).

For the second identity, $\varDelta(a\cdot b)\in\OU_{2m+2n-4}$ is represented by a submanifold $L\subset M\times N$ dual to $-c^2_1(M\times N)=(x+y)(-x-y)$. Similarly to the previous argument,
\[
  \varDelta(a\cdot b)=[L]=\varepsilon D_U\bigl(F_U(u,v)\overline{F_U(u,v)}\bigr)=
  \varepsilon D_U(-2uv)=-2\partial a\cdot\partial b.\qedhere
\]
\end{proof}

The direct sum $\mathcal W=\bigoplus_{i\ge0}\mathcal W_{2i}$ is
\emph{not} a subring of~$\OU$: one has $[\C P^1]\in\mathcal W_2$,
but $c_1^2[\C P^1\times\C P^1]=8\ne0$, so $[\C P^1]\times[\C
P^1]\notin\mathcal W_4$. 

The ring structure in $\mathcal W$ will be defined using a projection operator $\rho\colon\OU\to\OU$ which is described next. Recall the operation $\varPsi \colon\OU_{2n}\to\OU_{2n+4}$ defined in Construction~\ref{Psi}.

\begin{proposition}\label{rhoproj}
The homomorphism $\rho=\id-\varPsi\varDelta\colon\OU\to\OU$ is a projection operator such that
$\Im\rho=\mathcal W$, $\Ker\rho=\varPsi(\OU)$ and $\partial\rho=\rho\partial=\partial$.
\end{proposition}
\begin{proof}
The relation $\varDelta\varPsi=\id$ from Lemma~\ref{algrel} implies $(\id-\varPsi\varDelta)^2=\id-\varPsi\varDelta$, so $\rho$ is a projection. The same relation implies that $\varDelta\rho=0$, so $\Im\rho\subset\Ker\varDelta=\mathcal W$. The inclusion $\Im\rho\supset\Ker\varDelta$ is obvious. The identity $\Ker\rho=\Im\varPsi$ is proved similarly. Finally,
$\partial(\id-\varPsi\varDelta)=\partial-\partial\varPsi\varDelta=\partial$ because $\partial\varPsi=0$, and
$(\id-\varPsi\varDelta)\partial=\partial-\varPsi\varDelta\partial=\partial$ because $\varDelta\partial=0$.
\end{proof}

\begin{corollary}\label{OUW}
$\rank\mathcal W_{2n}=\rank\OU_{2n}-\rank\OU_{2n-4}$.
\end{corollary}
\begin{proof}
The previous proposition implies $\OU=\Ker\rho\oplus\Im\rho$. We have $(\Im\rho)_{2n}=\mathcal W_{2n}$ and $(\Ker\rho)_{2n}=\varPsi(\OU_{2n-4})\cong\OU_{2n-4}$ because $\varPsi$ is injective.
\end{proof}

Using the projection $\rho=\id-\varPsi\varDelta$, define the \emph{twisted
product} of elements $a,b\in\mathcal W$ as
\[
  a\mathbin{*}b=\rho(a\cdot b),
\]
where $\cdot$ denotes the product in $\OU$. A geometric description is given next.

\begin{proposition}\label{twprod}
We have
\[
  a\mathbin{*}b=a\cdot b+2[V^4]\cdot\partial a\cdot\partial b,
\]
where $V^4$ is the manifold $\C P^2$ with the stably complex structure
defined by the isomorphism $\mathcal T\C
P^2\oplus\underline{\R}^2\cong\bar\eta\oplus\bar\eta\oplus\eta$.
\end{proposition}
\begin{proof}
We need to verify that $\varPsi\varDelta(a\cdot b)=-2[V^4]\cdot\partial a\cdot\partial b$. By Lemma~\ref{Deltaab}, ${\varDelta(a\cdot b)}=-2\partial a\cdot\partial b$. Recall from Construction~\ref{Psi} that $\varPsi[M]$ is represented by the manifold $\C P(\overline{\det\mathcal TM}\oplus\underline{\C}^2)$ with the stably complex structure $p^*\mathcal TM\oplus(\bar\eta\otimes p^*\overline{\det\mathcal TM})\oplus\bar\eta\oplus\eta$. In  our case, $[M]=-2\partial a\cdot\partial b$, so $\det\mathcal TM$ is a trivial bundle. We obtain that the bordism class $\varPsi\varDelta(a\cdot b)=\varPsi[M]$ is represented by the total space of a trivial bundle over $M$ whose fibre is $\C P^2$ with the stably complex structure $\bar\eta\oplus\bar\eta\oplus\eta$. The latter bordism class is $[V^4]\cdot [M]=-2[V^4]\cdot \partial a\cdot\partial b$, as claimed.  
\end{proof}

\begin{remark}
We may also take $V^4=\C P^1\times\C P^1-\C P^2$ with the standard complex structure, as this manifold is bordant to the one described in Proposition~\ref{twprod}. 
\end{remark}

\begin{theorem}
The direct sum $\mathcal W=\bigoplus_{i\ge0}\mathcal W_{2i}$ is a 
commutative associative  unital ring with respect to the product~$*$.
\end{theorem}
\begin{proof}
We need to verify that the product~$*$ is associative. This is a direct calculation using the formula from Proposition~\ref{twprod}.
\end{proof}

The projection $\rho=\id-\varPsi\varDelta$ was defined by Conner and Floyd in~\cite[(8.4)]{co-fl66m} and used by Novikov~\cite[Remark~5.3]{novi67}. Stong~\cite[Chapter~VIII]{ston68} introduced another projection $\pi\colon\OU\to\OU$ with image $\mathcal W$, defined geometrically as follows. Take $[M]\in\OU$. Then $\pi[M]$ is the bordism class $[N]$ of the submanifold $N\subset\C P^1\times M$ dual to $\bar\eta\otimes\det\mathcal T M$. It follows easily from this geometric definition that $c_1(\pi[M])$ is a spherical class; in this way the equivalence of (a) and (c) in Theorem~\ref{W3def} is proved.

Buchstaber~\cite{buch72} used Stong's projection $\pi\colon\OU\to\mathcal W$ (under the name ``projection of Conner--Floyd type'') to define a complex-oriented cohomology theory with the coefficient ring~$\mathcal W$ and studied the corresponding formal group law. A general algebraic theory of projections of Conner--Floyd type was developed in~\cite{b-b-n-y00}; it was then used to classify stable associative multiplications in complex cobordism.

Both projection operators $\rho$ and $\pi$ have the same image $\mathcal W$ and coincide on the elements of the form $a\cdot b$ where $a,b\in\mathcal W$. Therefore, they define the same product in~$\mathcal W$. However the projections $\rho$ and $\pi$ are \emph{different}, as they have different kernels. Indeed, take $[M^6]=\varPsi[\C P^1]$. Then $\rho[M^6]=0$ because $[M^6]\in\Im\varPsi$. On the other hand, $\pi[M^6]\ne0$, because one can check that $c_1^3[M^6]=-2$, $c_3[M^6]=2$ and $c_3(\pi[M^6])=(-c_1^3+c_3)[M^6]=4$, which is nonzero. Also, $c_3(\rho[\C P^3])=68$, while $c_3(\pi[\C P^3])=-60$.

Recall from Theorem~\ref{Ustructure} that a bordism class $[M^{2i}]\in\OU_{2i}$ represents a polynomial generator of $\OU$ whenever $s_i[M^{2i}]=\pm m_i$, where the numbers $m_i$ are defined in~\eqref{mi}. A similar description for the ring $\mathcal W$ is given next.

\begin{theorem}\label{Wring}\samepage
$\mathcal W$ is a polynomial ring on generators in every even
degree except~$4$:
\[
  \mathcal W\cong
  \Z[x_1,x_i\colon i\ge3],\quad x_1=[\C P^1],\quad\deg x_i=2i.
\]
Polynomial generators $x_i$ are specified by the condition $s_i(x_i)=\pm m_im_{i-1}$ for $i\ge3$. 
The boundary operator
$\partial\colon\mathcal W\to\mathcal W$, $\partial^2=0$,
satisfies the identity
\begin{equation}\label{da*b}
  \partial(a\mathbin{*} b)=a\mathbin{*}\partial b+
  \partial a\mathbin{*} b-
  x_1\mathbin{*}\partial a\mathbin{*}\partial b.
\end{equation}
and the polynomial generators of $\mathcal W$ can be chosen so as to satisfy
\[
  \partial x_1=2, \quad \partial x_{2i}=x_{2i-1}.
\]
\end{theorem}
\begin{proof}
We start by checking the identity~\eqref{da*b}:
\[
  \partial (a \mathbin{*} b) = \partial \rho (ab) = \partial (ab) = 
  a \partial b + b \partial a - [\C P^1] \partial a \partial b
  =a \mathbin{*} \partial b + b \mathbin{*} \partial a - 
  [\C P^1] \mathbin{*} \partial a \mathbin{*} \partial b.
\]
Here the second identity is by Proposition~\ref{rhoproj}, the third idenity is Lemma~\ref{Deltaab}, and the last identity also follows from Lemma~\ref{Deltaab}, as the identity $\varDelta(ab)=-2\partial a\partial b$ for $a,b\in\mathcal W$ implies that $a\mathbin{*}b=ab$ whenever $a\in\Im\partial$ or $b\in\Im\partial$.

In the rest of this proof we denote the product of elements in $\mathcal W$ by $a\mathbin{*}b$ only when it differs from the product in $\OU$; otherwise we denote it by $a\cdot b$ or simply $ab$.

We start by defining bordism classes $b_i\in\mathcal W_{2i}$ for each $i\ge1$ except $i=2$. Set
\[
  b_i= 
  \begin{cases}
    [\C P^1] &\text{if }i=1,\\
    \pi [\C P^{2^p}\times\C P^{2^{p+1}q}] &\text{if }i=2^p(2q+1),\;p\ge1,\;q\ge1,\\
    \pi[\C P^{2^p}\times\C P^{2^p}] &\text{if }i=2^{p+1},\;p\ge1,\\
    \partial \,b_{i+1} &\text{if $i$ is odd and $i\ge3$},
  \end{cases}
\]
where $\pi\colon\OU\to\mathcal W$ is Stong's projection defined above.
One can check that
\begin{equation}\label{sibi} 
\begin{aligned}
  &s_i(b_i)=1\mod 2\quad\text{if $i\ne2^k-1$, $i\ne 2^k$},\\
  &s_i(b_i)=2\mod 4\quad\text{if $i=2^k-1$,}\\
  &s_i(b_i)=2\mod 4\quad\text{if $i=2^{p+1}$,}\\
  &s_{(2^p,2^p)}(b_{2^{p+1}})=1\mod 2.
\end{aligned}
\end{equation}
Consider the inclusion $\iota\colon\mathcal W\otimes\Z_2\to \OU\otimes\Z_2$. The formula for the product in $\mathcal W$ from Proposition~\ref{twprod} implies that $\iota$ is a ring homomorphism. Relations~\eqref{sibi} imply that there are polynomial generators $a_i$ of the ring $\OU\otimes\Z_2\cong\Z_2[a_i\colon i\ge1]$ such that $\iota(b_i)=a_i$ for $i\ne 2^{p+1}$ and $\iota(b_{2^{p+1}})=(a_{2^p})^2+\cdots$, where $\cdots$ denotes decomposable elements corresponding to partitions strictly less than $(2^p,2^p)$ in the lexicographic order. It follows that  the elements $\iota(b_i)$ are algebraically independent in the polynomial ring $\OU\otimes\Z_2\cong\Z_2[a_i\colon i\ge1]$. Therefore, $\mathcal W\otimes\Z_2$ contains the polynomial subring $\Z_2[b_1, b_i\colon i\ge3]$. By comparing the ranks using Corollary~\ref{OUW} we conclude that
\[
  \mathcal W\otimes\Z_2\cong \Z_2[b_1,b_i\colon i\ge3].
\]
Next we observe that $s_i(b_i)$ is an odd multiple of $m_im_{i-1}$ for $i\ge 3$, that is,
\begin{equation}\label{bioddm}
  s_i(b_i)=(2q_i+1)m_im_{i-1},\quad i\ge3.
\end{equation}
For even $i$ this follows from~\eqref{sibi} and the fact that $s_i(b_i)$ is a multiple of $m_i$, see Theorem~\ref{Ustructure}~(b). For odd $i$ we have $b_i=\partial b_{i+1}$, so $b_i$ is represented by an $SU$-manifold, and~\eqref{bioddm} follows from~\eqref{sibi} and Proposition~\ref{pkSU}.

By Theorem~\ref{noviosu}, there exist elements $y_i\in\OSU_{2i}$, $i\ge2$, such that
\begin{equation}\label{yiki}
  s_i(y_i)=2^{k_i}m_im_{i-1},\quad k_i\ge0.
\end{equation}
For the integers $q_i$ from~\eqref{bioddm} and $k_i$ from~\eqref{yiki} we find integers $\beta_i$ and $\gamma_i$ such that 
\[
  \beta_i2^{k_i+1}+\gamma_i(2q_i+1)=1.
\]
Then $\gamma_i$ is odd, so we have $\gamma_i=2\alpha_i+1$ for an integer~$\alpha_i$. Now we set
$x_1=[\C P^1]$ and
\[
  x'_i=(2\alpha_i+1)b_i+2\beta_iy_i,\quad i\ge3.
\]
Then the identities above imply that $s_i(x'_i)=m_im_{i-1}$.
The required elements $x_i$ are obtained by modifying the $x'_i$ as follows:
\[
  x_{2i-1}=x'_{2i-1}, \quad x_{2i}=
  x'_{2i}-x_1\bigl((\alpha_{2i}-\alpha_{2i-1})b_{2i-1}-\beta_{2i-1}y_{2i-1}\bigr).
\]
Then we have 
\[
  s_i(x_i)=m_im_{i-1}
\]
because $x_i-x'_i$ is decomposable. 
The new element $x_{2i}$ still belongs to $\mathcal W$; to verify this we use the second identity of Lemma~\ref{Deltaab}:
\[
  \varDelta x_{2i}=\varDelta x'_{2i} +2\partial x_1
  \partial\bigl((\alpha_{2i}-\alpha_{2i-1})b_{2i-1}-\beta_{2i-1}y_{2i-1}\bigr)=0
\]
because $x'_{2i}\in\mathcal W=\Ker\varDelta$, $\partial b_{2i-1}=\partial^2b_{2i}=0$ and $\partial y_{2i-1}=0$ because $y_{2i-1}\in\OSU$.

To verify the identity $\partial x_{2i}=x_{2i-1}$ we use the first identity of Lemma~\ref{Deltaab}:
\begin{multline*}
  \partial x_{2i}=\partial x'_{2i}-
  \partial x_1\cdot\bigl((\alpha_{2i}-\alpha_{2i-1})b_{2i-1}-\beta_{2i-1}y_{2i-1}\bigr)=
  (2\alpha_{2i}+1)\partial b_{2i}\\
  -2\bigl((\alpha_{2i}-\alpha_{2i-1})b_{2i-1}-\beta_{2i-1}y_{2i-1}\bigr)=
  (2\alpha_{2i-1}+1)b_{2i-1}+2\beta_{2i-1}y_{2i-1}=x_{2i-1}.
\end{multline*}

Now we define a homomorphism
\[
  \varphi\colon\mathcal R= \Z[x_1,x_i\colon i\ge3]\to\mathcal W,
\]
which sends the polynomial generator $x_i$ to the corresponding element of~$\mathcal W$, defined above. Obseve that $\varphi\otimes\Z_2$ sends $x_i$ to $b_i$ modulo decomposable elements. As we have seen, $\mathcal W\otimes\Z_2\cong \Z_2[b_1,b_i\colon i\ge3]$, which implies that $\varphi\otimes\Z_2$ is an isomorphism. Since $\mathcal R$ and $\mathcal W$ are torsion free, $\varphi$ is injective and $\varphi(\mathcal R_n) \subset \mathcal W_n$ is a subgroup of odd index in each dimension.

We will show that $\varphi\colon\mathcal R\to\mathcal W$ becomes surjective after tensoring with $\Z[\frac{1}{2}]$. This will imply that $\varphi$ is an isomorphism.

Note that for any $\alpha \in \mathcal W$ we have 
\[
  \partial (x_1\mathbin{*} \alpha) = \partial x_1  \cdot \alpha + x_1 \cdot \partial \alpha - 
  x_1 \cdot \partial x_1 \cdot \partial \alpha = 2 \alpha - x_1 \partial \alpha.
\]
Hence, $\alpha = \frac{1}{2} \partial (x_1\mathbin{*} \alpha) + \frac{1}{2} x_1 \partial \alpha$ in $\mathcal W \otimes \Z[\frac{1}{2}]$. It follows that $\mathcal W \otimes \Z[\frac{1}{2}]$ is generated by $1$ and $x_1$ as a module over $\OSU \otimes \Z[\frac{1}{2}] \subset \mathcal W \otimes \Z[\frac{1}{2}]$ (note that $\OSU \otimes \Z[\frac{1}{2}]$ is a subring of $\mathcal W \otimes \Z[\frac{1}{2}]$, by the formula from Proposition~\ref{twprod}). Furthermore, this module is free because $0 = a + x_1 b$ with $a, b \in \OSU \otimes \Z[\frac{1}{2}]$ implies $0 =\partial (a + x_1 b) = \partial x_1 \cdot b = 2 b$ and therefore $b=0$ and $a=0$. Hence, 
\[
  \mathcal W \otimes \Z[{\textstyle\frac{1}{2}}] =
  \OSU \otimes \Z[\textstyle{\frac{1}{2}}] \langle1, x_1 \rangle.
\]
Now we define new elements in $\varphi(\mathcal R)\subset\mathcal W$:
\begin{equation}\label{yidefi}
\begin{aligned}
  &y_2=2x_1\mathbin{*}x_1=\partial(x_1\mathbin{*} x_1 \mathbin{*} x_1),\\
  &y_{2i} = \partial (x_1\mathbin{*} x_{2i}) = 2 x_{2i} - x_1 x_{2i-1},&\quad i\ge2,\\
  &y_{2i-1}=x_{2i-1}=\partial x_{2i},&\quad i\ge2.
\end{aligned}
\end{equation}
These elements actually lie in $\OSU$, because they belong to $\Im\partial$. Then
\begin{equation}\label{siyicalc}
\begin{aligned}
  &s_2(y_2)=2s_2(x_1 \cdot x_1 + 8 [V^4])=-16s_2(\C P^2)=-48=-8m_2m_1,\\
  &s_{2i}(y_{2i}) = 2 s_{2i} (x_{2i}) = 2 m_{2i} m_{2i-1},&\quad i\ge2,\\
  &s_{2i-1}(y_{2i-1}) = s_{2i-1}(x_{2i-1}) = m_{2i-1} m_{2i-2},&\quad i\ge2,
\end{aligned}
\end{equation}
and therefore the $y_i$ are polynomial generators of $\OSU \otimes \Z[\frac{1}{2}]$ by Theorem~\ref{noviosu}. It follows that $\mathcal W \otimes \Z[\frac{1}{2}]=\OSU \otimes \Z[\frac{1}{2}] \langle 1, x_1 \rangle\subset\varphi\bigl(\mathcal R \otimes \Z[\frac{1}{2}]\bigr)$.
Thus, $\varphi \otimes \Z[\frac{1}{2}]$ is epimorphism, which completes the proof.
\end{proof}

\section{The ring structure of $\OSU$}\label{ringosusec}
The forgetful map $\alpha\colon\OSU\to\mathcal W$ is a ring
homomorphism; this follows from Proposition~\ref{twprod} because
$\partial\alpha(x)=0$ for any $x\in\OSU$. Therefore, the ring $\OSU/\mathrm{Tors}$ can be described as a subring in~$\mathcal W$.

Note that we have
\begin{equation}\label{Wring2}
  \mathcal W\otimes\Z[{\textstyle\frac12}]\cong
  \Z[{\textstyle\frac12}][x_1,x_{2k-1},2x_{2k}-x_1x_{2k-1}\colon k\ge2],
\end{equation}
where $x_1^2=x_1\mathbin{*}x_1$ is a $\partial$-cycle, and each of the
elements $x_{2k-1}$ and $2x_{2k}-x_1x_{2k-1}$ with $k\ge2$ is a $\partial$-cycle.

For any integer $n\ge 3$ define
\begin{equation}\label{gn}
  g(n)=\begin{cases}
  2m_{n-1}m_{n-2}&\text{if $n>3$ is odd;}\\
  m_{n-1}m_{n-2} &\text{if $n>3$ is even;}\\
  -48 &\text{if $n=3$.}
\end{cases}
\end{equation}
These numbers appear in~\eqref{siyicalc}.
For example, $g(4)=6$, $g(5)=20$. For $n>3$, the number $g(n)$ can take the following values: $1$, $2$, $4$, $p$, $2p$, $4p$, where $p$ is an odd prime. 

\begin{theorem}\label{SUstructure}
There exist indecomposable elements $y_i\in\OSU_{2i}$, $i\ge2$, with minimal $s$-numbers given by $s_i(y_i)=g(i+1)$. These elements
are mapped as follows under the forgetful homomorphism
$\alpha\colon\OSU\to\mathcal W$:
\[
  y_2\mapsto 2x_1^2,\quad y_{2k-1}\mapsto x_{2k-1},\quad
  y_{2k}\mapsto 2x_{2k}-x_1x_{2k-1},\quad k\ge2,
\]
where the $x_i$ are polynomial generators of $\mathcal W$. In
particular, $\OSU\otimes\Z[\frac12]\cong\Z[\frac12][y_i\colon i\ge2]$ embeds into~\eqref{Wring2} as
the polynomial subring generated by $x_1^2$, $x_{2k-1}$ and
$2x_{2k}-x_1x_{2k-1}$.
\end{theorem}
\begin{proof}
The elements $y_i\in \OSU_{2i}$ were defined in~\eqref{yidefi}, and their $s$-numbers were given by~\eqref{siyicalc}. We only need to check that the $s$-number of $y_i$ is minimal possible in~$\OSU_{2i}$. 

For $y_{2k-1}$, the number $m_{2k-1}m_{2k-2}$ is minimal possible for all elements in~$\mathcal W_{4k-2}$ by Theorem~\ref{Wring}, and therefore it is also minimal possible in~$\OSU_{4k-2}\subset\mathcal W_{4k-2}$. (Note that indecomposability in $\mathcal W$ with respect to the product $*$ is the same as indecomposability in~$\OU$ in dimensions~$>4$; this follows from Proposition~\ref{twprod}.)

For $y_2=2x_1^2$, we have $\OSU_4=\Im\partial=\Z\langle y_2\rangle$, where $y_2=2K$ in the notation of Example~\ref{lssuex}.

Now consider $y_{2k}$ with $k\ge2$. We have $s_{2k}(y_{2k})=2m_{2k}m_{2k-1}$. Take any element $a\in\OSU_{4k}\subset(\Ker\partial)_{4k}$.  It follows from~\eqref{Wring2} that $\Ker(\partial\colon\mathcal W\to\mathcal W)$ consists of $\Z[\frac12]$-polynomials in $x_1^2$, $x_{2i-1}$, $2x_{2i}-x_1x_{2i-1}$ which have integral coefficients in the $x_i$'s. Write
\[
  a=\lambda(2x_{2k}-x_1x_{2k-1})+b,
\]
where $\lambda\in\Z[\frac12]$ and $b$ is a decomposable element in $\Z[{\textstyle\frac12}][x_1^2,x_{2i-1},2x_{2i}-x_1x_{2i-1}]$. Then $b$ does not contain $x_1x_{2k-1}$, hence $\lambda\in\Z$. Therefore, $s_{2k}(a)=2\lambda s_{2k}(x_{2k})=\lambda\cdot2m_{2k}m_{2k-1}$, so
$2m_{2k}m_{2k-1}$ is the minimal possible $s$-number in $\OSU_{4k}$.
\end{proof}


Recall that the image of the forgetful homomorphism $\alpha\colon\OSU\to\mathcal W$ is $\OSU/\mathrm{Tors}$ by Theorem~\ref{osutors}~(a). Furthermore, by Theorem~\ref{freetors}~(b), $\OSU_{2i}/\mathop{\mathrm{Tors}}$ is isomorphic to
$\Ker(\partial\colon\mathcal W\to\mathcal W)$ if $2i\not\equiv
4\mod 8$ and is isomorphic to $\mathop{\mathrm{Im}}(\partial\colon\mathcal W\to\mathcal W)$ if
$2i\equiv 4\mod 8$. Combining this with Theorem~\ref{Wring2}, we obtain a description of $\OSU/\mathop{\mathrm{Tors}}$ as a subring in~$\mathcal W$. Finally, the multiplicative structure of the torsion elements is described by Theorem~\ref{freetors}~(c). Collecting these pieces of information together we obtain, in principle, a  full description of the ring~$\OSU$. However, as noted by Stong at the end of Chapter~X in~\cite{ston68}, an intrinsic description of this ring is extremely complicated. 
For example, the nontrivial graded components of $\OSU$ of dimension $\le10$ are described in terms of the elements $x_i$ and $y_i$ from Theorem~\ref{Wring2} as follows:
\begin{gather*}
  \OSU_0=\Z,\quad\OSU_1=\Z_2\langle\theta\rangle,\quad
  \OSU_2=\Z_2\langle\theta^2\rangle,\\[2pt]
  \varOmega^{SU}_4=\Z\langle y_2\rangle,\; y_2=2x_1^2,\quad
  \varOmega^{SU}_6=\Z\langle y_3\rangle,\; y_3=x_3,\quad
  \varOmega^{SU}_8=\Z\langle {\textstyle\frac14}y^2_2,y_4\rangle,\; y_4=2x_4-x_1x_3,\\[2pt]
  \OSU_9=\Z_2\langle\theta x_1^4\rangle,\quad
  \varOmega^{SU}_{10}=\Z\langle {\textstyle\frac12}y_2y_3,y_5\rangle
  \oplus\Z_2\langle\theta^2 x_1^4\rangle,\; y_5=x_5.
\end{gather*}
We have
\[
  y_2=2x_1^2=2\bigl(9[\C P^1]\times[\C P^1]-8[\C P^2]\bigr)
\]
as a $U$-bordism class. In dimension $8$ we have 
\[
  {\textstyle\frac14}y^2_2=x_1^4=\bigl(9[\C P^1]\times[\C P^1]-8[\C P^2]\bigr)\times
  \bigl(9[\C P^1]\times[\C P^1]-8[\C P^2]\bigr)
\]
as a $U$-bordism class, because $x_1^2=9[\C P^1]\times[\C P^1]-8[\C P^2]$ is a $\partial$-cycle.
Also, $\frac14y^2_2=x_1^4$ can be chosen as $w_4$ in Theorem~\ref{freetors}~(c).
We see that $8$ is the first dimension where $\OSU/\mathrm{Tors}$ differs from a polynomial ring, as the square of the $4$-dimensional generator $y_2$ is divisible by~$4$. Furthermore, the product of the $4$- and $6$- dimensional generators is divisible by~$2$.

\part{Geometric representatives}

\section{Toric varieties and quasitoric manifolds}\label{toricsect}

Here we collect the necessary information about toric varieties
and quasitoric manifolds. Standard references on toric
geometry include Danilov's survey~\cite{dani78} and books by
Oda~\cite{oda88}, Fulton~\cite{fult93} and Cox, Little and
Schenck~\cite{c-l-s11}. More information about quasitoric manifolds can be found
in~\cite[Chapter~6]{bu-pa15}.

A \emph{toric variety} is a normal complex algebraic variety~$V$
containing an algebraic torus $(\C^\times)^n$ as a Zariski open
subset in such a way that the natural action of $(\C^\times)^n$ on
itself extends to an action on~$V$. A nonsingular
complete (compact in the usual topology) toric variety is called a \emph{toric manifold}.

There is the fundamental correspondence of toric geometry between the isomorphism
classes of complex $n$-dimensional toric varieties and rational fans in~$\R^n$. 
Under this correspondence,
\begin{align*}
\text{cones }&\longleftrightarrow\text{ affine toric varieties}\\
\text{complete fans }&\longleftrightarrow\text{ complete (compact) toric varieties}\\
\text{normal fans of polytopes}&\longleftrightarrow\text{ projective toric varieties}\\
\text{nonsingular fans }&\longleftrightarrow\text{ nonsingular toric varieties}\\
\text{simplicial fans }&\longleftrightarrow\text{ toric orbifolds}
\end{align*}

A \emph{fan} is a finite collection $\Sigma=\{\sigma_1,\ldots,\sigma_s\}$ of strongly convex cones
$\sigma_j$ in $\R^n$ such that every face of a cone in $\Sigma$
belongs to $\Sigma$ and the intersection of any two cones in
$\Sigma$ is a face of each. A fan is \emph{rational} (with respect to the standard integer lattice $\Z^n\subset\R^n$) if each of its cones is generated by rational (or lattice) vectors. In particular, each one-dimensional cone of a rational fan $\Sigma$ is generated by a primitive vector $\mb a_i\in \Z^n$. A fan $\Sigma$ is \emph{simplicial} if
each of its cones $\sigma_j$ is generated by part of a basis of $\R^n$ (such a cone is also called \emph{simplicial}). A fan $\Sigma$ is \emph{nonsingular} if
each of its cones $\sigma_j$ is generated by part of a basis of
the lattice~$\Z^n$. A fan
$\Sigma$ is \emph{complete} if the union of its cones is the
whole~$\R^n$.

Projective toric varieties are particularly important. A
projective toric variety $V$ is defined by a  \emph{lattice polytope}, that is, a convex $n$-dimensional polytope $P$ with vertices in~$\Z^n$. The \emph{normal
fan} $\Sigma_P$ is the fan whose $n$-dimensional cones $\sigma_v$
correspond to the vertices $v$ of $P$, and $\sigma_v$ is generated
by the primitive inside-pointing normals to the facets of $P$
meeting at~$v$. The fan $\Sigma_P$ defines a
projective toric variety~$V_P$. Different lattice polytopes with the same normal fan produce different projective embeddings of the same toric variety.

A polytope $P$ is called \emph{nonsingular} or \emph{Delzant} when its normal fan $\Sigma_P$ is nonsingular. 
Projective toric manifolds correspond to nonsingular lattice polytopes. Note that a nonsingular $n$-dimensional polytope $P$ is necessarily \emph{simple}, that is, there are precisely $n$ facets meeting at every vertex of~$P$. 

Irreducible torus-invariant divisors on~$V$ are the toric
subvarieties of complex codimension~1 corresponding to the
one-dimensional cones of~$\Sigma$. When $V$ is projective, they
also correspond to the facets of~$P$. We assume that there are $m$
one-dimensional cones (or facets), denote the corresponding
primitive vectors by $\mb a_1,\ldots,\mb a_m$, and denote the
corresponding codimension-1 subvarieties (irreducible divisors) by~$D_1,\ldots,D_m$.

\enlargethispage{2\baselineskip}

\begin{theorem}[Danilov--Jurkiewicz]\label{cohomtoric}
Let $V$ be a toric manifold of complex dimension~$n$, with the
corresponding complete nonsingular fan~$\Sigma$. The cohomology ring
$H^*(V;\Z)$ is generated by the degree-two classes $v_i$ dual to
the invariant submanifolds $D_i$, and is given~by
\[
  H^*(V;\Z)\cong \Z[v_1,\ldots,v_m]/\mathcal I,\qquad\deg v_i=2,
\]
where $\mathcal I$ is the ideal generated by elements of the
following two types:
\begin{itemize}
\item[(a)] $v_{i_1}\cdots v_{i_k}$ such that $\mb a_{i_1},\ldots,\mb
a_{i_k}$ do not span a cone of~$\Sigma$;
\item[(b)] $\displaystyle\sum_{i=1}^m\langle\mb a_i,\mb x\rangle v_i$, for
any vector $\mb x\in\Z^n$.
\end{itemize}
\end{theorem}

There is the same description of the cohomology ring for complete toric orbifolds with coefficients in~$\Q$.

It is convenient to consider the integer $n\times m$-matrix
\begin{equation}\label{Lambdatoric}
  \varLambda=\begin{pmatrix}
  a_{11}&\cdots& a_{1m}\\
  \vdots&\ddots&\vdots\\
  a_{n1}&\cdots& a_{nm}
  \end{pmatrix}
\end{equation}
whose columns are the vectors $\mb a_i$ written in the standard
basis of~$\Z^n$. Then part~(b) of the ideal $\mathcal I$ in Theorem~\ref{cohomtoric} is
generated by the $n$ linear forms $a_{j1}v_1+\cdots+a_{jm}v_m$
corresponding to the rows of~$\varLambda$.

\begin{theorem}\label{tangenttoric}
For a toric manifold $V$, there is the following isomorphism of complex vector bundles:
\[
  \mathcal T V\oplus\underline{\C}^{m-n}\cong
  \rho_1\oplus\cdots\oplus\rho_m,
\]
where $\mathcal T V$ is the tangent bundle, $\underline{\C}^{m-n}$
is the trivial $(m-n)$-plane bundle, and $\rho_i$ is the line
bundle corresponding to~$D_i$, with $c_1(\rho_i)=v_i$. In
particular, the total Chern class of~$V$ is given by
\[
  c(V)=(1+v_1)\cdots(1+v_m).
\]
\end{theorem}

\begin{example}
A basic example of a toric manifold is the complex projective
space $\C P^n$. The cones of the corresponding fan are generated
by proper subsets of the set of $m=n+1$ vectors $\mb
e_1,\ldots,\mb e_n,-\mb e_1-\cdots-\mb e_n$, where $\mb
e_i\in\Z^n$ is the $i$th standard basis vector. It is the normal
fan of the lattice simplex $\varDelta^n$ with the vertices at $\bf
0$ and $\mb e_1,\ldots,\mb e_n$. The matrix~\eqref{Lambdatoric} is
given by
\[
  \begin{pmatrix}
  1&0&0&-1\\
  0&\ddots&0&\vdots\\
  0&0&1&-1
  \end{pmatrix}
\]

Theorem~\ref{cohomtoric} gives the cohomology of $\C P^n$ as
\[
  H^*(\C P^n)\cong\Z[v_1,\ldots,v_{n+1}]/(v_1\cdots v_{n+1},
  v_1-v_{n+1},\ldots,v_n-v_{n+1})\cong\Z[v]/(v^{n+1}),
\]
where $v$ is any of the $v_i$. Theorem~\ref{tangenttoric} gives
the standard decomposition
\[
  \mathcal T\C P^n\oplus\underline{\C}\cong
  \bar\eta\oplus\cdots\oplus\bar\eta\qquad\text{($n+1$ summands)},
\]
where $\eta=\mathcal O(-1)$ is the \emph{tautological} (Hopf) line
bundle over $\C P^n$, and $\bar\eta=\mathcal O(1)$ is its
conjugate, or the line bundle corresponding to a hyperplane $\C
P^{n-1}\subset\C P^n$.
\end{example}

\begin{example}\label{projex1}
The complex projectivisation of a sum of line bundles over a projective
space is a toric manifold. This example will feature in several subsequent constructions.

Given two positive integers $n_1$, $n_2$ and a sequence of
integers $(i_1,\ldots,i_{n_2})$, consider the projectivisation
$V=\C P(\eta^{\otimes i_1}\oplus\cdots\oplus\eta^{\otimes
i_{n_2}}\oplus\underline{\C})$, where $\eta^{\otimes i}$ denotes
the $i$th tensor power of $\eta$ over $\C P^{n_1}$ when $i\ge0$
and the $i$th tensor power of $\bar\eta$ otherwise. The manifold
$V$ is the total space of a bundle over $\C P^{n_1}$ with fibre
$\C P^{n_2}$. It is also a projective toric manifold with the
corresponding matrix~\eqref{Lambdatoric} given by

\[
  \begin{pmatrix}
  \noalign{\vspace{-1\normalbaselineskip}}
  \multicolumn{3}{c}{\scriptstyle n_1}\\[-5pt]
  \multicolumn{3}{c}{$\downbracefill$}\,\\
  1&0&0&-1                                           & & & &  \\
  0&\ddots&0&\vdots    &&&\textrm{\huge 0} &                  \\
  0&0&1&-1                                           & & & &  \\
   & & &i_1       &                                   1&0&0&-1\\
   &\textrm{\huge 0} & &\vdots &             0&\ddots&0&\vdots\\
   & & &i_{n_2}   &                                  0&0&1&-1 \\[-5pt]
   &&&&\multicolumn{3}{c}{$\upbracefill$\ }\\[-5pt]
   &&&&\multicolumn{3}{c}{\scriptstyle n_2}\\
  \noalign{\vspace{-1\normalbaselineskip}}
  \end{pmatrix}
  \vspace{1\normalbaselineskip}
\]
The polytope $P$ here is combinatorially equivalent to a product
$\varDelta^{n_1}\times\varDelta^{n_2}$ of two simplices.
Theorem~\ref{cohomtoric} describes the cohomology of $V$ as
\[
  H^*(V)\cong\Z[v_1,\ldots,v_{n_1+1},v_{n_1+2},\ldots,v_{n_1+n_2+2}]/\mathcal
  I,
\]
where $\mathcal I$ is generated by the elements
\begin{gather*}
  v_1\cdots v_{n_1+1},\;v_{n_1+2}\cdots
  v_{n_1+n_2+2},\;v_1-v_{n_1+1},\ldots,v_{n_1}-v_{n_1+1},\\
  i_1v_{n_1+1}+v_{n_1+2}-v_{n_1+n_2+2},\ldots,
  i_{n_2}v_{n_1+1}+v_{n_1+n_2+1}-v_{n_1+n_2+2}.
\end{gather*}
In other words,
\begin{equation}\label{cohomproj1}
  H^*(V)\cong\Z[u,v]\big/\bigl(u^{n_1+1},v(v-i_1u)\cdots(v-i_{n_2}u)\bigr),
\end{equation}
where $u=v_{1}=\cdots=v_{n_1+1}$ and $v=v_{n_1+n_2+2}$.
Theorem~\ref{tangenttoric} gives
\begin{equation}\label{cproj1}
  c(V)=(1+u)^{n_1+1}(1+v-i_1u)\cdots(1+v-i_{n_2}u)(1+v).
\end{equation}
If $i_1=\cdots=i_{n_2}=0$, we obtain $V=\C P^{n_1}\times\C
P^{n_2}$.

The same information can be retrieved from the following
well-known description of the tangent bundle and the cohomology
ring of a complex projectivisation.

\begin{theorem}[Borel and Hirzebruch {\cite[\S15]{bo-hi58}}]\label{projd}
Let $p\colon \C P(\xi)\to X$ be the projectivisation of a complex
$n$-plane bundle $\xi$ over a complex manifold~$X$, and let
$\gamma$ be the tautological line bundle over~$\C P(\xi)$. Then
there is an isomorphism of vector bundles
\[
  \mathcal T\C P(\xi)\oplus\underline{\C}\cong
   p^*{\mathcal T}\!X\oplus(\bar\gamma\otimes p^*\xi).
\]
Furthermore, the integral cohomology ring of $\C P(\xi)$
is the quotient of the polynomial ring $H^*(X)[v]$ on one
generator $v=c_1(\bar\gamma)$ with coefficients in $H^*(X)$ by the
single relation
\begin{equation}\label{cntatf}
  v^n+c_1(\xi)v^{n-1}+\cdots+c_n(\xi)=0.
\end{equation}
\end{theorem}

The relation above is just $c_n(\bar\gamma\otimes p^*\xi)=0$.

In Example~\ref{projex1} we have $\xi=\eta^{\otimes
i_1}\oplus\cdots\oplus\eta^{\otimes i_{n_2}}\oplus\underline{\C}$
over $X=\C P^{n_1}$. We have $H^*(X)=\Z[u]/(u^{n_1+1})$ where
$u=c_1(\bar\eta)$, so that~\eqref{cntatf} becomes
$v(v-i_1u)\cdots(v-i_{n_2}u)=0$ and the ring $H^*(\C P(\xi))$
given by Theorem~\ref{projd} is precisely~\eqref{cohomproj1}.
Moreover, the total Chern class of $p^*{\mathcal
T}\!X\oplus(\bar\gamma\otimes p^*\xi)$ is given by~\eqref{cproj1}.
\end{example}

The quotient of the projective toric manifold $V_P$ by the action
of the compact torus $T^n\subset(\C^\times)^n$ is the simple
polytope~$P$. Davis and Januszkiewicz~\cite{da-ja91} introduced the following topological generalisation of projective toric manifolds.

A \emph{quasitoric manifold} over a simple
$n$-dimensional polytope $P$ is a smooth manifold $M$ of dimension $2n$
with a locally standard action of the torus $T^n$ and a continuous projection $\pi\colon M\to P$ whose fibres are $T^n$-orbits. (An
action of $T^n$ on $M^{2n}$ is \emph{locally standard} if every
point $x\in M^{2n}$ is contained in a $T^n$-invariant
neighbourhood equivariantly homeomorphic to an open subset in
$\C^n$ with the standard coordinatewise action of~$T^n$ twisted by
an automorphism of the torus.) The orbit space of a locally
standard action is a manifold with corners. The quotient of a quasitoric manifold
$M/T^n$ is homeomorphic, as a manifold with corners, to~$P$.

Not every simple polytope can be the quotient of a quasitoric
manifold. Nevertheless, quasitoric manifolds constitute a much
larger family than projective toric manifolds, and enjoy more
flexibility for topological applications.

If $F_1,\ldots,F_m$ are the facets of $P$, then each
$M_i=\pi^{-1}(F_i)$ is a quasitoric submanifold of $M$ of
codimension~2, called a \emph{characteristic submanifold}. The
characteristic submanifolds $M_i\subset M$ are analogues of the
invariant divisors $D_i$ on a toric manifold~$V$. Each $M_i$ is
fixed pointwise by a closed $1$-dimensional subgroup (a subcircle)
$T_i\subset T^n$ and therefore corresponds to a primitive vector
$\lambda_i\in\Z^n$ defined up to a sign. Choosing a direction of
$\lambda_i$ is equivalent to choosing an orientation for the
normal bundle $\nu(M_i\subset M)$ or, equivalently, choosing an
orientation for $M_i$, provided that $M$ itself is oriented. An
\emph{omniorientation} of a quasitoric manifold $M$ consists of a
choice of orientation for $M$ and each characteristic submanifold
$M_{i}$, $1\le i\le m$.

The vectors $\lambda_i$ play the role of the generators $\mb a_i$
of the one-dimensional cones of the fan corresponding to a toric
manifold~$V$ (or the normal vectors to the facets of
$P$ when $V$ is projective). However, the $\lambda_i$ need not be
the normal vectors to the facets of $P$ in general.

There is an analogue of Theorem~\ref{cohomtoric} for quasitoric
manifolds:

\begin{theorem}\label{cohomqtoric}
Let $M$ be an omnioriented quasitoric manifold of dimension~$2n$
over a polytope~$P$. The cohomology ring $H^*(M;\Z)$ is generated
by the degree-two classes $v_i$ dual to the oriented
characteristic submanifolds $M_i$, and is given by
\[
  H^*(M;\Z)\cong \Z[v_1,\ldots,v_m]/\mathcal I,\qquad\deg v_i=2,
\]
where $\mathcal I$ is the ideal generated by elements of the
following two types:
\begin{itemize}
\item[(a)] $v_{i_1}\cdots v_{i_k}$ such that $F_{i_1}\cap\cdots\cap F_{i_k}=\varnothing$ in~$P$;
\item[(b)] $\displaystyle\sum_{i=1}^m\langle\lambda_i,\mb x\rangle v_i$, for
any vector $\mb x\in\Z^n$.
\end{itemize}
\end{theorem}

By analogy with \eqref{Lambdatoric}, we consider the integer
$n\times m$-matrix
\begin{equation}\label{Lambdaqtoric}
  \varLambda=\begin{pmatrix}
  \lambda_{11}&\cdots& \lambda_{1m}\\
  \vdots&\ddots&\vdots\\
  \lambda_{n1}&\cdots& \lambda_{nm}
  \end{pmatrix}
\end{equation}
whose columns are the vectors $\lambda_i$ written in the standard
basis of~$\Z^n$. Changing a basis in the lattice results in
multiplying $\varLambda$ from the left by a matrix from
$\mbox{\textit{GL}}\,(n,\Z)$. The ideal~(b) of
Theorem~\ref{cohomqtoric} is generated by the $n$ linear forms
$\lambda_{j1}v_1+\cdots+\lambda_{jm}v_m$ corresponding to the rows
of~$\varLambda$. Also, $\varLambda$ has the property that
$\det(\lambda_{i_1},\ldots,\lambda_{i_n})=\pm1$ whenever the
facets $F_{i_1},\ldots,F_{i_n}$ intersect at a vertex of~$P$.

There is also an analogue of Theorem~\ref{tangenttoric}:

\begin{theorem}\label{tangentqtoric}
For a quasitoric manifold $M$ of dimension $2n$, there is an
isomorphism of real vector bundles:
\begin{equation}\label{TMqtiso}
  \mathcal T M\oplus\underline{\R}^{2(m-n)}\cong
  \rho_1\oplus\cdots\oplus\rho_m,
\end{equation}
where $\rho_i$ is the real $2$-plane bundle corresponding to the
orientable characteristic submanifold $M_i\subset M$, so that
$\rho_i|_{M_i}=\nu(M_i\subset M)$.
\end{theorem}

Buchstaber and Ray~\cite{bu-ra98} introduced a family of projective toric manifolds $\{B(n_1,n_2)\}$ that multiplicatively generates the unitary bordism ring $\OU$. The details of this construction can be found in~\cite[\S9.1]{bu-pa15}. 
We proceed to describe another family of toric generators
for~$\varOmega^U$.

\begin{construction}\label{Pn1n2}
Given two positive integers $n_1$, $n_2$, we define the manifold
$L(n_1,n_2)$ as the projectivisation $\C P
(\eta\oplus\underline{\C}^{n_2})$, where $\eta$ is the
tautological line bundle over~$\C P^{n_1}$. This $L(n_1,n_2)$ is a
particular case of manifolds from Example~\ref{projex1}, so it is
a projective toric manifold with the corresponding
matrix~\eqref{Lambdatoric} given by\\
\begin{equation}\label{LambdaPn1n2}
  \begin{pmatrix}
  \noalign{\vspace{-1\normalbaselineskip}}
  \multicolumn{3}{c}{\scriptstyle n_1}\\[-5pt]
  \multicolumn{3}{c}{$\downbracefill$}\,\\
  1&0&0&-1                                           & & & &  \\
  0&\ddots&0&\vdots    &&&\textrm{\huge 0} &                  \\
  0&0&1&-1                                           & & & &  \\
   & & &1       &                                   1&0&0&-1\\
   &\textrm{\huge 0} & &0 &             0&\ddots&0&\vdots\\
   & & &0   &                                  0&0&1&-1 \\[-5pt]
   &&&&\multicolumn{3}{c}{$\upbracefill$\ }\\[-5pt]
   &&&&\multicolumn{3}{c}{\scriptstyle n_2}\\
  \noalign{\vspace{-1\normalbaselineskip}}
  \end{pmatrix}
  \vspace{1\normalbaselineskip}
\end{equation}
The cohomology ring is given by
\begin{equation}\label{cohomproj2}
  H^*\bigl(L(n_1,n_2)\bigr)\cong\Z[u,v]\big/\bigl(u^{n_1+1},v^{n_2+1}-uv^{n_2}\bigr)
\end{equation}
with $u^{n_1}v^{n_2}\langle L(n_1,n_2)\rangle=1$. There is an
isomorphism of complex bundles
\begin{equation}\label{TPn1n2}
  \mathcal T
  L(n_1,n_2)\oplus\underline{\C}^2\cong
  \underbrace{p^*\bar\eta\oplus\cdots\oplus p^*\bar\eta}_{n_1+1}\oplus
  (\bar\gamma\otimes p^*\eta)\oplus
  \underbrace{\bar\gamma\oplus\cdots\oplus\bar\gamma}_{n_2},
\end{equation}
where $\gamma$ is the tautological line bundle over $L(n_1,n_2)=\C
P(\eta\oplus\underline{\C}^{n_2})$. The total Chern class is
\begin{equation}\label{cproj2}
  c\bigl(L(n_1,n_2)\bigr)=(1+u)^{n_1+1}(1+v-u)(1+v)^{n_2}
\end{equation}
with $u=c_1(p^*\bar\eta)$ and $v=c_1(\bar\gamma)$. We also set
$L(n_1,0)=\C P^{n_1}$ and $L(0,n_2)=\C P^{n_2}$, then the
identities \eqref{cohomproj2}--\eqref{cproj2} still hold.
\end{construction}

\begin{theorem}[{\cite[Theorem~3.8]{lu-pa16}}]\label{projcobgen}
The bordism classes $[L(n_1,n_2)]\in\varOmega^U_{2(n_1+n_2)}$
generate multiplicatively the unitary bordism ring~$\varOmega^U$.
\end{theorem}

Theorem~\ref{projcobgen} implies that every unitary bordism class
can be represented by a disjoint union of products of projective
toric manifolds. Products of toric manifolds are toric, but
disjoint unions are not, as toric manifolds are connected. In
bordism theory, a disjoint union may be replaced by a connected
sum, representing the same bordism class. However, connected sum
is not an algebraic operation, and a connected sum of two
algebraic varieties is rarely algebraic. This can be remedied by
appealing to quasitoric manifolds, as explained next. Recall that
an omnioriented quasitoric manifold has an intrinsic stably
complex structure, arising from the isomorphism of
Theorem~\ref{tangentqtoric}. One can form the equivariant connected
sum of quasitoric manifolds, as explained in Davis and Januszkiewicz~\cite{da-ja91}, but
the resulting invariant stably complex structure does not
represent the bordism sum of the two original manifolds. A more
intricate connected sum construction is needed, as outlined below.
The details can be found in~\cite{b-p-r07} or~\cite[\S9.1]{bu-pa15}.

\begin{construction}\label{diamondsum}
The construction applies to two omnioriented $2n$-dimensional
quasitoric manifolds $M$ and $M'$ over $n$-polytopes $P$ and $P'$
respectively. The connected sum will be taken at the fixed points
of $M$ and $M'$ corresponding to vertices $v\in P$ and $v'\in P'$.
We need to assume that $v$ is the intersection of the first $n$
facets of $P$, i.e. $v=F_1\cap\cdots\cap F_n$, and the
corresponding characteristic matrix~\eqref{Lambdaqtoric} of~$M$ is
in the \emph{refined form}, i.e.
\[
\varLambda=\left(I\;|\;\varLambda_\star\right)\;=\;\begin{pmatrix}
  1&0&0&\lambda_{1,n+1}&\ldots&\lambda_{1,m}\\
  0&\ddots&0&\vdots&\ddots&\vdots\\
  0&0&1&\lambda_{n,n+1}&\ldots&\lambda_{n,m}
\end{pmatrix}
\]
where $I$ is the unit matrix and $\varLambda_\star$ is an
$n\times(m-n)$-matrix. The same assumptions are made for $M'$,
$P'$, $v'$ and $\varLambda'$.

The next step depends on the \emph{signs} of the fixed points,
$\omega(v)$ and $\omega(v')$. The sign of $v$ is determined by the
omniorientation data; it is $+1$ when the orientation of $\mathcal
T_v M$ induced from the global orientation of $M$ coincides with
the orientation arising from
$\rho_1\oplus\cdots\oplus\rho_n|_{v}$, and is $-1$ otherwise.

If $\omega(v)=-\omega(v')$, then we take the connected sum $M\cs
M'$ at $v$ and $v'$. It is a quasitoric manifold over $P\cs P'$
with the characteristic matrix
$\left(\varLambda_\star\;|\;I\;|\;\varLambda'_\star\right)$.

If $\omega(v)=\omega(v')$, then we need an additional connected
summand. Consider the quasitoric manifold $S=S^2\times\cdots\times
S^2$ over the $n$-cube~$I^n$, where each $S^2$ is the quasitoric
manifold over the segment $I$ with the characteristic matrix
$(1\;1)$. It represents zero in $\varOmega^U$, and may be thought
of as $\C P^1$ with the stably complex structure given by the
isomorphism $\mathcal T\C
P^1\oplus\underline{\R}^2\cong\bar\eta\oplus\eta$. The
characteristic matrix of $S$ is therefore $(I\;|\;I)$. Now
consider the connected sum $M\cs S\cs M'$. It is a quasitoric
manifold over $P\cs I^n\cs P'$ with the characteristic matrix
$\left(\varLambda_\star\;|\;I\;|\;I\;|\;\varLambda'_\star\right)$.

In either case, the resulting omnioriented quasitoric manifold
$M\cs M'$ or $M\cs S\cs M'$ with the canonical stably complex
structure represents the sum of bordism classes
$[M]+[M']\in\varOmega^U_{2n}$.
\end{construction}

The conclusion, which can be derived from the above construction
and any of the toric generating sets $\{B(n_1,n_2)\}$ or
$\{L(n_1,n_2)\}$ for~$\varOmega^U$, is as follows:

\begin{theorem}[\cite{b-p-r07}]\label{6.11}
In dimensions $>2$, every unitary bordism class contains a
quasitoric manifold, necessarily connected, whose stably complex
structure is induced by an omniorientation, and is therefore
compatible with the torus action.
\end{theorem}

\section{Quasitoric $SU$-manifolds}\label{qtSUsect}
Omnioriented quasitoric manifolds whose
stably complex structures are $SU$ can be detected using the
following simple criterion:

\begin{proposition}[\cite{b-p-r10}]
An omnioriented quasitoric manifold $M$ has $c_1(M)=0$ if and only
if there exists a linear function $\varphi\colon\Z^n\to\Z$ such
that $\varphi(\lambda_i)=1$ for $i=1,\ldots,m$. Here the
$\lambda_i$ are the columns of matrix~\eqref{Lambdaqtoric}.

In particular, if some $n$ vectors of $\lambda_1,\ldots,\lambda_m$
form the standard basis $\mb e_1,\ldots,\mb e_n$, then $M$ is $SU$
if and only if the column sums of $\varLambda$ are all equal
to~$1$.
\end{proposition}
\begin{proof}
By Theorem~\ref{tangentqtoric}, $c_1(M)=v_1+\cdots+v_m$. By
Theorem~\ref{cohomqtoric}, $v_1+\cdots+v_m$ is zero in $H^2(M)$ if
and only if $v_1+\cdots+v_m=\sum_i\varphi(\lambda_i)v_i$ for some
linear function $\varphi\colon\Z^n\to\Z$, whence the result
follows.
\end{proof}

\begin{proposition}\label{noSU}
A toric manifold $V$ cannot be~$SU$.
\end{proposition}
\begin{proof}
If $\varphi(\lambda_i)=1$ for all~$i$, then the vectors
$\lambda_i$ lie in the positive halfspace of~$\varphi$, so they
cannot span a complete fan.
\end{proof}

A more subtle result also rules out low-dimensional quasitoric
manifolds:

\begin{theorem}[{\cite[Theorem~6.13]{b-p-r10}}]\label{lowdimqt}
A quasitoric $SU$-manifold $M^{2n}$ represents $0$ in $\OU_{2n}$
whenever~$n<5$.
\end{theorem}

The reason for this is that the Krichever genus $\varphi_{\mathrm
K}\colon\OU\to R_{\mathrm K}$ (see~\cite[\S E.5]{bu-pa15}) vanishes on quasitoric
$SU$-manifolds, but $\varphi_{\mathrm K}$ is an isomorphism in
dimensions~$<10$.

First examples of quasitoric $SU$-manifolds representing nonzero bordism classes in $\OU_{2n}$ for all $n\ge5$, except~$n=6$, were constructed in~\cite{lu-wa16}. Subsequently, in~\cite{lu-pa16} there were constructed two general series of quasitoric $SU$-manifolds representing nonzero bordism classes in $\OU_{2n}$ (and therefore in~$\OSU_{2n}$) for all $n\ge5$, including~$n=6$. These series are presented next. They will be used below to provide geometric representatives for multiplicative generators in the $SU$-bordism ring.

\begin{construction}
Assume now that $n_1=2k_1$ is positive even and $n_2=2k_2+1$ is positive odd. We change the stably complex
structure~\eqref{TPn1n2} to the following:
\begin{multline*}
  \mathcal T
  L(n_1,n_2)\oplus\underline{\R}^4\\
  \cong
  \underbrace{p^*\bar\eta\oplus p^*\eta\oplus\cdots\oplus
  p^*\bar\eta\oplus p^*\eta}_{2k_1}\oplus p^*\bar\eta\oplus
  (\bar\gamma\otimes p^*\eta)\oplus
  \underbrace{\bar\gamma\oplus\gamma\oplus\cdots\oplus\bar\gamma\oplus\gamma}_{2k_2}
  \oplus\gamma
\end{multline*}
and denote the resulting stably complex manifold by $\widetilde
L(n_1,n_2)$. Its cohomology ring is given by the same
formula~\eqref{cohomproj2}, but
\begin{equation}\label{chernL}
  c\bigl(\widetilde
  L(n_1,n_2)\bigr)=(1-u^2)^{k_1}(1+u)(1+v-u)(1-v^2)^{k_2}(1-v),
\end{equation}
so $\widetilde L(n_1,n_2)$ is an $SU$-manifold of dimension
$2(n_1+n_2)=4(k_1+k_2)+2$.

Viewing $L(n_1,n_2)$ as a quasitoric manifold with the
omniorientation coming from the complex structure, we see that
changing a line bundle $\rho_i$ in~\eqref{TMqtiso} to its
conjugate results in changing $\lambda_i$ to $-\lambda_i$
in~\eqref{Lambdaqtoric}. By applying this operation to the
corresponding columns of~\eqref{LambdaPn1n2} and then multiplying
from the left by an appropriate matrix from
$\mbox{\textit{GL}}\,(n,\Z)$, we obtain that $\widetilde
L(n_1,n_2)$ is the omnioriented quasitoric manifold over
$\varDelta^{n_1}\times\varDelta^{n_2}$ corresponding to the
matrix\\
\[
  \left(\begin{array}{ccccccccccc}
  \noalign{\vspace{-1\normalbaselineskip}}
  \multicolumn{5}{c}{\scriptstyle n_1=2k_1}\\[-5pt]
  \multicolumn{5}{c}{$\downbracefill$}\,\\
  1&0&0&\cdots&0&1                                           & & & & & \\
  0&1&0&\cdots&0&-1                                          & & & & & \\
  \vdots&\ddots&\ddots&\ddots&\vdots&\vdots    &&&\textrm{\huge 0} & & \\
  0&0&0&1&0&1                                                & & & & & \\
  0&0&0&0&1&-1                                               & & & & & \\
   & & & & &1       &                                    1&0&\cdots&0&1\\
   & & & & &        &                                    0&1&\cdots&0&-1\\
   & &\textrm{\huge 0} & & & &            \vdots&\ddots&\ddots&0&\vdots\\
   & & &                 & & &                                0&0&0&1&1\\[-5pt]
   & & & & &        & \multicolumn{4}{c}{$\upbracefill$}\\[-3pt]
   & & & & &        & \multicolumn{4}{c}{\scriptstyle n_2=2k_2+1}\\
  \noalign{\vspace{-1\normalbaselineskip}}
  \end{array}\right)
  \vspace{1\normalbaselineskip}
\]
The column sums of this matrix are~$1$ by inspection.
\end{construction}

\begin{construction}\label{constrN} The previous construction can be iterated by
considering projectivisations of sums of line bundles
over~$L(n_1,n_2)$. We shall need just one particular family of
this sort.

Given positive even $n_1=2k_1$ and odd $n_2=2k_2+1$, consider the
omnioriented quasitoric manifold $\widetilde N(n_1,n_2)$ over
$\varDelta^1\times\varDelta^{n_1}\times\varDelta^{n_2}$ with the
characteristic matrix\\
\[
  \left(\begin{array}{cccccccccccccc}
  1&1\\
   & & 1&0&0&\cdots&0&1                                           & & & & & &\\
   & & 0&1&0&\cdots&0&-1                                          & & & & & &\\
  \textrm{\huge 0}\!\!\!\! & & \vdots&\ddots&\ddots&\ddots&\vdots&\vdots    &&&&\textrm{\huge 0}  & &\\
   & & 0&0&0&1&0&1                                                & & & & & &\\
   & & 0&0&0&0&1&-1                                               & & & & & &\\[5pt]
   \noalign{\vspace{-1\normalbaselineskip}}
   & & \multicolumn{5}{c}{$\upbracefill$}\,\\[-3pt]
   &-1& \multicolumn{5}{c}{\scriptstyle n_1=2k_1} &0   &      1&0&0&\cdots&0&1\\
   &1& & & & & &0       &                                    0&1&0&\cdots&0&-1\\
   &0& & & & & &1       &                                    0&0&1&\cdots&0&1\\
   & & & &\textrm{\huge 0} & & & &  \vdots&\vdots&\vdots&\ddots&\vdots&\vdots\\
   & & & & &                 & & &                                0&0&0&0&1&1\\[-5pt]
   & & & & & & &        & \multicolumn{5}{c}{$\upbracefill$}\\[-3pt]
   & & & & & & &        & \multicolumn{5}{c}{\scriptstyle n_2=2k_2+1}\\
  \noalign{\vspace{-1\normalbaselineskip}}
  \end{array}\right)
  \vspace{1\normalbaselineskip}
\]
The column sums are $1$ by inspection, so $\widetilde N(n_1,n_2)$
is a quasitoric $SU$-manifold of dimension
$2(1+n_1+n_2)=4(k_1+k_2)+4$.

It can be seen that $\widetilde N(n_1,n_2)$ is a projectivisation
of a sum of $n_2+1$ line bundles over $\C P^1\times\C P^{n_1}$
with an amended stably complex structure.

The cohomology ring given by Theorem~\ref{cohomqtoric} is
\begin{equation}\label{cohomN}
  H^*(\widetilde N(n_1,n_2))\cong\Z[u,v,w]\big/
  \bigl(u^2,v^{n_1+1},(w-u)^2(v+w)w^{n_2-2}\bigr)
\end{equation}
with $uv^{n_1}w^{n_2}\langle\widetilde N(n_1,n_2)\rangle=1$. The
total Chern class is
\begin{equation}\label{chernN}
  c(\widetilde N(n_1,n_2))=(1-v^2)^{k_1}(1+v)
  (1-(w-u)^2)(1-v-w)(1-w^2)^{k_2-1}(1+w).
\end{equation}
\end{construction}

\section{Quasitoric generators for the $SU$-bordism ring}\label{qtSUgensect} 
As shown in~\cite{lu-pa16}, the
elements $y_i\in\OSU_{2i}$ described in Theorem~\ref{SUstructure} can
be represented by quasitoric $SU$-manifolds when $i\ge5$. We outline the proof here, emphasising some interesting divisibility conditions for binomial coefficients. These divisibility properties arise from analysing the characteristic numbers of the quasitoric $SU$-manifolds $\widetilde L(n_1,n_2)$ and $\widetilde N(n_1,n_2)$ introduced in the previous section.

\begin{lemma}\label{snL}
For $n_1=2k_1>0$ and $n_2=2k_2+1>0$, we have
\[
  s_{n_1+n_2}\bigl[\widetilde L(n_1,n_2)\bigr]=
  -\bin{n_1+n_2}1+\bin{n_1+n_2}2-\cdots-\bin{n_1+n_2}{n_1-1}+\bin{n_1+n_2}{n_1}.
\]
\end{lemma}
\begin{proof}
Using~\eqref{chernL} and~\eqref{cohomproj2} we calculate
\begin{align*}
  s_{n_1+n_2}\bigl(\widetilde L(n_1,n_2)\bigr)
  &=(v-u)^{n_1+n_2}+(k_2+1)(-1)^{n_1+n_2}v^{n_1+n_2}+k_2v^{n_1+n_2}\\
  &=(v-u)^{n_1+n_2}-v^{n_1+n_2}\\
  &=\Bigl(-\bin{n_1+n_2}1+\bin{n_1+n_2}2-\cdots
  -\bin{n_1+n_2}{n_1-1}+\bin{n_1+n_2}{n_1}\Bigr)u^{n_1}v^{n_2},
\end{align*}
and the result follows by evaluating at the fundamental class of $\widetilde
L(n_1,n_2)$.
\end{proof}

Note that $s_3(\widetilde L(2,1))=0$ in accordance with
Theorem~\ref{lowdimqt}. On the other hand, $s_{2+n_2}(\widetilde
L(2,n_2))\ne0$ for $n_2>1$, providing an example of a
non-bounding quasitoric $SU$-manifold in each dimension $4k+2$
with $k>1$.

\begin{lemma}\label{yL}
For $k>1$, there is a linear combination $y_{2k+1}$ of
$SU$-bordism classes $[\widetilde L(n_1,n_2)]$ with $n_1+n_2=2k+1$
such that $s_{2k+1}(y_{2k+1})=m_{2k+1}m_{2k}$.
\end{lemma}
\begin{proof}
By the previous lemma,
\[
  s_{n_1+n_2}\bigl[\widetilde L(n_1,n_2)-\widetilde L(n_1-2,n_2+2)\bigr]
  =\bin{n_1+n_2}{n_1}-\bin{n_1+n_2}{n_1-1}.
\]
The result follows from the next lemma.
\end{proof}

\begin{lemma}[{\cite[Lemma~4.14]{lu-pa16}}]\label{gcddif}
For any integer $k>1$, we have
\[
  \gcd\bigl\{\bin{2k+1}{2i}-\bin{2k+1}{2i-1},\;0<i\le k\bigr\}
  =m_{2k+1}m_{2k}.
\]
\end{lemma}

Lemma~\ref{gcddif} also follows from the results of Buchstaber and Ustinov on the coefficient rings of universal formal group laws~\cite[\S9]{bu-us15}.

Now we turn our attention to the manifolds $\widetilde N(n_1,n_2)$
from Construction~\ref{constrN}.

\begin{lemma}\label{snN}
For $n_1=2k_1>0$ and $n_2=2k_2+1>0$, set $n=n_1+n_2+1$, so that
$\dim\widetilde N(n_1,n_2)=2n=4(k_1+k_2+1)$. Then
\[
  s_n\bigl[\widetilde N(n_1,n_2)\bigr]=2\bigl(-\bin{n}1+\bin{n}2-\cdots
  -\bin{n}{n_1-1}+\bin{n}{n_1}-n_1\bigr).
\]
\end{lemma}
\begin{proof}
Using~\eqref{chernN} and \eqref{cohomN} we calculate
\begin{multline}\label{snN1}
  s_{n}\bigl(\widetilde N(n_1,n_2)\bigr)
  =2(w-u)^n+(v+w)^n+(2k_2-1)w^n\\
  =2w^n-2nuw^{n-1}+w^n+\bin n1vw^{n-1}+\cdots+\bin n{2k_1}v^{2k_1}w^{2k_2+2}+(2k_2-1)w^n\\
  =-2nuw^{n-1}+(n-n_1)w^n+\bin n1vw^{n-1}+\cdots+\bin
  n{n_1}v^{n_1}w^{n-n_1}.
\end{multline}
Now we have to express each monomial above via $uv^{n_1}w^{n_2}$
using the identities in~\eqref{cohomN}, namely
\begin{equation}\label{3ident}
  u^2=0,\quad v^{n_1+1}=0,\quad
  w^{n_2+1}=2uw^{n_2}-vw^{n_2}+2uvw^{n_2-1}.
\end{equation}
We have
\begin{multline}\label{uwn-1}
  uw^{n-1}=uw^{n_1-1}w^{n_2+1}=uw^{n_1-1}(2uw^{n_2}-vw^{n_2}+2uvw^{n_2-1})
  \\=-uvw^{n-2}=\cdots=(-1)^juv^jw^{n-j-1}=\cdots=uv^{n_1}w^{n_2}.
\end{multline}
Also, we show that
\begin{equation}\label{vjwn-j}
  v^jw^{n-j}=(-1)^j2uv^{n_1}w^{n_2},\quad 0\le j\le n_1,
\end{equation}
by verifying the identity successively for $j=n_1,n_1-1,\ldots,0$.
Indeed, $v^{n_1}w^{n-n_1}=v^{n_1}w^{n_2+1}=2uv^{n_1}w^{n_2}$
by~\eqref{3ident}. Now, we have
\begin{multline*}
  v^{j-1}w^{n-j+1}=v^{j-1}w^{n_1+1-j}w^{n_2+1}
  =v^{j-1}w^{n_1+1-j}(2uw^{n_2}-vw^{n_2}+2uvw^{n_2-1})\\
  =2uv^{j-1}w^{n-j}-v^jw^{n-j}+2uv^jw^{n-1-j}=-v^jw^{n-j},
\end{multline*}
where the last identity holds because of~\eqref{uwn-1}. The
identity~\eqref{vjwn-j} is therefore verified completely.
Plugging~\eqref{uwn-1} and~\eqref{vjwn-j} into~\eqref{snN1} we
obtain
\[
  s_{n}\bigl(\widetilde N(n_1,n_2)\bigr)=\bigl(-2n+2(n-n_1)
  -2\bin{n}1+2\bin{n}2-\cdots
  -2\bin{n}{n_1-1}+2\bin{n}{n_1}\bigr)uv^{n_1}w^{n_2}.
\]
The result follows by evaluating at $\langle\widetilde
N(n_1,n_2)\rangle$.
\end{proof}

Note that $s_4(\widetilde N(2,1))=0$ in accordance with
Theorem~\ref{lowdimqt}. On the other hand, $s_n(\widetilde
N(2,n_2))=n^2-3n-4>0$ for $n>4$, providing an example of a
non-bounding quasitoric $SU$-manifold in each dimension $4k$
with $k>2$. This includes a 12-dimensional quasitoric
$SU$-manifold $\widetilde N(2,3)$, which was missing
in~\cite{lu-wa16}.

\begin{lemma}\label{yN}
For $k>2$, there is a linear combination $y_{2k}$ of $SU$-bordism
classes $[\widetilde N(n_1,n_2)]$ with $n_1+n_2+1=2k$ such that
$s_{2k}(y_{2k})=2m_{2k}m_{2k-1}$.
\end{lemma}
\begin{proof}
The result follows from Lemma~\ref{snN} and
Lemmata~\ref{Nmod2},~\ref{Nmodp} below.
\end{proof}

\begin{lemma}[{\cite[Lemma~4.17]{lu-pa16}}]\label{Nmod2}
For $k>2$, the largest power of $2$ which divides each number
\[
  a_i=-\bin{2k}1+\bin{2k}2-\cdots
  -\bin{2k}{2i-1}+\bin{2k}{2i}-2i,\quad 0<i< k,
\]
is $2$ if $2k=2^s$ and is $1$ otherwise.
\end{lemma}

\begin{lemma}[{\cite[Lemma~4.18]{lu-pa16}}]\label{Nmodp}
For $k>2$, the largest power of odd prime $p$ which divides each
\[
  a_i=-\bin{2k}1+\bin{2k}2+\cdots
  -\bin{2k}{2i-1}+\bin{2k}{2i}-2i,\quad 0<i< k,
\]
is $p$ if $2k+1=p^s$ and is $1$ otherwise.
\end{lemma}

We now obtain the following result about quasitoric representatives in $SU$-bordism:

\begin{theorem}\label{mainth}
There exist quasitoric $SU$-manifolds $M^{2i}$, $i\ge5$, with
$s_i(M^{2i})=m_im_{i-1}$ if $i$ is odd and
$s_i(M^{2i})=2m_{i}m_{i-1}$ if $i$ is even. These quasitoric
$SU$-manifolds have minimal possible $s_i$ numbers and represent polynomial generators
of~$\OSU\otimes\Z[\frac12]$.
\end{theorem}
\begin{proof}
It follows from Lemmata~\ref{yL} and~\ref{yN} that there exist
linear combinations of $SU$-bordism classes represented by
quasitoric $SU$-manifolds with the required properties. We observe
that application of Construction~\ref{diamondsum} to two
quasitoric $SU$-manifolds $M$ and $M'$ produces a quasitoric
$SU$-manifold representing their bordism sum. Also, the
$SU$-bordism class $-[M]$ can be represented by the omnioriented
qua\-si\-toric $SU$-manifold obtained by reversing the global
orientation of~$M$. Therefore, we can replace the linear
combinations obtained using Lemmata~\ref{yL} and~\ref{yN} by
appropriate connected sums, which are quasitoric $SU$-manifolds.
\end{proof}

By analogy with Theorem~\ref{6.11}, we may ask the following:

\begin{question}\label{question}
Which $SU$-bordism classes of dimension $>8$ can be represented by
quasitoric $SU$-manifolds?
\end{question}

\section{$SU$-manifolds arising in toric geometry}\label{SUtoricsect}

We refer to a compact K\"ahler manifold $M$ with $c_1(M)=0$ as a \emph{Calabi--Yau manifold}. (Apparently, this is the most standard definition; however, other definitions of a Calabi--Yau manifold, sometimes inequivalent to this one, also appear in the literature.) According to the theorem of Yau, conjectured by Calabi, a Calabi--Yau manifold admits a K\"ahler metric with zero Ricci curvature (for this, only vanishing of the  first \emph{real} Chern class is required). By definition, a Calabi--Yau manifold is an $SU$-manifold.

The standard complex structure on a toric manifold is never $SU$ (Propostion~\ref{noSU}), so  there are no toric Calabi--Yau manifolds. However, the following construction gives Calabi--Yau hypersurfaces in special toric manifolds.

\begin{construction}[Batyrev~\cite{baty94}]\label{bacon}
A toric manifold $V$ is \emph{Fano} if its anticanonical class $D_1+\cdots+D_m$ (representing $c_1(V)$) is very ample. In geometric terms, the projective embedding $V\hookrightarrow \C P^s$ corresponding to $D_1+\cdots+D_m$ comes from a lattice polytope $P$ in which the lattice distance from $0$ to each hyperplane containing a facet is~$1$. Such a lattice polytope $P$ is called \emph{reflexive}; its polar polytope $P^*$ is also a lattice polytope.  

The submanifold $N$ dual to $c_1(V)$ (see Construction~\ref{c1dual}) is given by the hyperplane section of the embedding $V\hookrightarrow \C P^s$ defined by $D_1+\cdots+D_m$. Therefore, $N\subset V$ is a smooth algebraic hypersurface in~$V$, so $N$ is a Calabi--Yau manifold of complex dimension $n-1$.

In this way, any toric Fano manifold $V$ of dimension $n$ (or equivalently, any non-singular  reflexive  $n$-dimensional polytope~$P$) gives rise to a canonical $(n-1)$-dimensional Calabi--Yau manifold~$N_P$. 

Batyrev~\cite{baty94} also extended this construction to some singular toric Fano varieties. A complex normal irreducible $n$-dimensional projective algebraic variety
$W$ with only Gorenstein canonical singularities is called a \emph{Calabi--Yau variety} if $W$ has trivial canonical bundle and $H^{i}(W,\mathcal O_W)=0$ for $0<i<n$.

Suppose $f$ is a Laurent polynomial in $n$ variables, and let $P=P(f)$ be its Newton polytope (the convex hull of the lattice points corresponding to the nonzero coefficients of~$f$). Then $f$ defines an affine hypersurface $Z_f$ in the algebraic torus $(\mathbb{C}^\times)^{n}$, and its Zariski closure $\overline{Z}_{f,P}$, a hypersurface in the projective toric variety~$V_P$. A hypersurface $\overline{Z}_{f,P}$ is said to be \emph{$P$-regular} if it intersects each facial subvariety of $V_P$ at a subvariety of codimension one (in particular, it does not intersect the points fixed under the torus actions). By~\cite[Theorem 4.1.9]{baty94}, the following conditions are equivalent for a $P$-regular hypersurface $\overline{Z}_{f,P}$:
\begin{itemize}
\item[(a)] 
$\overline{Z}_{f,P}$ is a Calabi--Yau variety with canonical singularities;
\item[(b)]
$V_P$ is a toric Fano variety with Gorenstein singularities;
\item[(c)]
$P$ is a reflexive polytope (up to shifting the origin).
\end{itemize}
Furthermore, by~\cite[Theorem~4.2.2]{baty94}, there exists a special resolution of singularities $\widehat Z_{f,P}\to\overline Z_{f,P}$ (a toroidal MPCP-desingularization) such that $\widehat Z_{f,P}$ is a Calabi--Yau variety with singularities in codimension~$\ge4$. In particular, if $\dim P\le 4$, then we obtain a smooth Calabi--Yau manifold. 
This led to defining a family of \emph{mirror-dual} pairs of Calabi--Yau $3$-folds arising from reflexive $4$-polytopes and their polars. 
\end{construction}

The $s$-number of the Calabi--Yau manifold $N=N_P$ is given as follows.

\begin{lemma}\label{sN}
We have
\[
  s_{n-1}(N)=\bigl\langle(v_1+\cdots+v_m)(v_1^{n-1}+\cdots+v_m^{n-1})-
  (v_1+\cdots+v_m)^n,[V]\bigr\rangle.
\]
\end{lemma}
\begin{proof}
We have an isomorphism of complex bundles $\mathcal TN\oplus\nu\cong i^*\mathcal T V$, where $\nu$ is the normal bundle of the embedding $i\colon N\hookrightarrow V$. Hence, $s_{n-1}(\mathcal TN)+s_{n-1}(\nu)=i^*s_{n-1}(\mathcal T V)$. Now we calculate
\begin{multline*}
  \bigl\langle s_{n-1}(\mathcal T N),[N]\bigr\rangle=
  \bigl\langle -s_{n-1}(\nu)+i^* s_{n-1}(\mathcal T V),[N]\bigr\rangle\\=
  \bigl\langle c_1(\mathcal TV) \bigl(-c_1^{n-1}(\mathcal TV)+s_{n-1}(\mathcal T V)\bigr),     
  [V]\bigr\rangle \\=
  \bigl\langle c_1(\mathcal TV)s_{n-1}(\mathcal T V)-c_1^n(\mathcal TV),     
  [V]\bigr\rangle.\qedhere
\end{multline*}
\end{proof}

\section{Calabi--Yau generators for the $SU$-bordism ring}\label{CYSUsect}
A family of Calabi--Yau manifolds whose $SU$-bordism classes generate the special unitary bordism ring $\varOmega^{SU}[\frac{1}{2}]\cong\mathbb{Z}[\frac{1}{2}][y_{i}\colon{i\ge 2}]$ was constructed in~\cite{l-l-p18}. This construction is reviewed below.

Let $\omega=(i_1,\ldots,i_k)$ be an unordered partition of $n$ into a sum of $k$ positive integers, that is, $i_1+\cdots+i_k=n$. Let $\varDelta^{i}$ be the standard reflexive simplex of dimension~$i$. Then $\varDelta^\omega=\varDelta^{i_1}\times\cdots\times\varDelta^{i_k}$ is a reflexive polytope with the corresponding toric Fano manifold $\C P^\omega=\C P^{i_1}\times\cdots\times\C P^{i_k}$. We denote by $N_\omega$ the Calabi--Yau hypersurface in~$\C P^\omega$ given by Construction~\ref{bacon}.

Let $\widehat P(n)$ be the set of all partitions $\omega$ with parts of size at most $n-2$. That is,
\[
  \widehat P(n)=\{\omega=(i_1,\ldots,i_k)\colon
  i_1+\cdots+i_k=n,\quad\omega\ne(n),(1,n-1).\} 
\]
The multinomial coefficient $\binom n\omega=\frac{n!}{i_1!\cdots i_k!}$ is defined for each~$\omega=(i_1,\ldots,i_k)$. We set
\begin{equation}\label{alphasigma}
  \alpha(\omega)=\binom{n}{\omega} 
  (i_{1}+1)^{i_{1}}\cdots(i_k+1)^{i_k}.
\end{equation}

\begin{lemma}\label{snumber}
For any $\omega\in \widehat P(n)$ we have
\[
  s_{n-1}(N_\omega)=-\alpha(\omega).
\]
\end{lemma}
\begin{proof}
The cohomology ring of $\C P^\omega=\C P^{i_1}\times\cdots\times\C P^{i_k}$ is given by
\[
  H^*(\C P^\omega;\Z)\cong\Z[u_1,\ldots,u_k]/(u_1^{i_1+1},\ldots,u_k^{i_k+1}),
\]
where $u_1:=v_1=\cdots=v_{i_1+1}$, $u_2:=v_{i_1+2}=\cdots=v_{i_1+i_2+2}$, $\ldots,$ $u_k:=v_{i_1+\cdots+i_{k-1}+k}=\cdots=v_{i_1+\cdots+i_k+k}=v_m$.
As $\omega\in \widehat P(n)$, we have $v_i^{n-1}=0$ in $H^*(\C P^\omega;\Z)$ for any~$i$. The formula from Lemma~\ref{sN} gives
\[
  s_{n-1}(N_\omega)=-\langle(v_1+\cdots+v_m)^n,[\C P^\omega]\rangle=
  -\langle((i_1+1)u_1+\cdots+(i_k+1)u_k)^n,[\C P^\omega]\rangle.
\]
Evaluating at $[\C P^\omega]$ gives the coefficient of $u_1^{i_1}\cdots u_k^{i_k}$ 
in the polynomial above, whence the result follows.
\end{proof}

\begin{lemma}[{\cite[Lemma~2.3]{l-l-p18}}]\label{gcd}
For $n\ge3$, we have
\[
  \gcd_{\omega\in\widehat{P}(n)}\alpha(\omega)=g(n),
\]
where the numbers $g(n)$ and $\alpha(\omega)$ are given by~\eqref{gn} and~\eqref{alphasigma} respectively.
\end{lemma}

The proof of this Lemma given in~\cite{l-l-p18} uses the results of Mosley~\cite{mosl} on the divisibility of multinomial coefficients.

\begin{theorem}\label{SUCY}
The $SU$-bordism classes of the Calabi--Yau hypersurfaces $N_\omega$ in $\C P^{i_1}\times\cdots\times\C P^{i_k}$ 
with $\omega\in\widehat{P}(n)$, $n\ge 3$, multiplicatively generate the $SU$-bordism 
ring~$\varOmega^{SU}[\frac{1}{2}]$.
\end{theorem} 
\begin{proof}
For any $n\ge3$ we use Lemma~\ref{gcd} and Lemma~\ref{snumber} to find a linear combination of the bordism classes $[N_\omega]\in\varOmega^{SU}_{2n-2}$ whose $s$-number is precisely~$g(n)$. This linear combination is the polynomial generator $y_{n-1}$ of $\varOmega^{SU}[\frac{1}{2}]$, as described in Theorem~\ref{SUstructure}.
\end{proof}

We actually prove an \emph{integral} result: the elements $y_i\in\varOmega^{SU}$ described in Theorem~\ref{SUstructure} can be represented by integral linear combinations of the bordism classes of Calabi--Yau manifolds~$N_\omega$. The element $y_i$ is part of a basis of the abelian group $\varOmega^{SU}_{2i}$. There is the following related question:

\begin{question}
Which bordism classes in $\varOmega^{SU}$ can be represented by Calabi--Yau manifolds?
\end{question}

This question is an $SU$-analogue of the following well-known problem of Hirzebruch: which bordism classes in $\varOmega^U$ contain connected (i.\,e., irreducible) non-singular algebraic varieties? If one drops the connectedness assumption, then any $U$-bordism class of positive dimension can be represented by an algebraic variety. Since a product and a positive integral linear combination of algebraic classes is an algebraic class (possibly, disconnected), one only needs to find in each dimension $i$ algebraic varieties $M$ and $N$ with $s_i(M)=m_i$ and $s_i(N)=-m_i$, see Theorem~\ref{Ustructure}. The corresponding argument, originally due to Milnor, is given in~\cite[p.~130]{ston68}. Note that it uses hypersurfaces in $\C P^n$ and a calculation similar to Lemma~\ref{sN}. For $SU$-bordism, the situation is different: if a class $a\in\varOmega^{SU}$ can be represented by a Calabi--Yau manifold, then $-a$ does not necessarily have this property. Therefore, the next step towards the answering the question above is whether $y_i$ and $-y_i$ can be simultaneously represented by Calabi--Yau manifolds. We elaborate on this in the next section.

\section{Low dimensional generators in the $SU$-bordism ring}\label{lowdimsect}

Here we describe geometric Calabi--Yau representatives for the generators $y_i$ of the $SU$-bordism ring (see Theorem~\ref{SUstructure}) in complex dimension $i\le 4$. Note that for $i\ge 5$, each generator $y_i\in\varOmega^{SU}_{2i}$ can be represented by a quasitoric manifold, by Theorem~\ref{mainth}. On the other hand, every quasitoric $SU$-manifold of real dimension $\le8$ is null-bordant by Theorem~\ref{lowdimqt}.

Recall from Section~\ref{ringosusec} that we have
\[
  \varOmega^{SU}_4=\Z\langle y_2\rangle,\quad
  \varOmega^{SU}_6=\Z\langle y_3\rangle,\quad
  \varOmega^{SU}_8=\Z\langle {\textstyle\frac14}y^2_2,y_4\rangle,
\]
with the values of the $s$-number given by
\[
  s_2(y_2)=-48, \quad s_3(y_3)=m_3m_2=6,\quad s_4(y_4)=2m_4m_3=20. 
\]

\begin{example}\label{CYsurface}
Consider the Calabi--Yau hypersurface $N_{(3)}\subset\C P^3$ corresponding to the partition $\omega=(3)$.
We have $c_1(\C P^3)=4u$, where $u\in H^2(\C P^3;\Z)$ is the canonical generator dual to a hyperplane section. Therefore, $N_{(3)}$ can be given by a generic quartic equation in homogeneous coordinates on $\C P^3$. The standard example is the quartic given by $z_0^4+z_1^4+z_2^4+z_3^4=0$, which is a $K3$-surface. Lemma~\ref{sN} gives
\[
  s_3(N_{(3)})=\langle4u^2\cdot4u-(4u)^3,[\C P^3]\rangle=-48,
\]
so $N_{(3)}$ represents the generator $y_2\in\varOmega^{SU}_4$.

Note that Theorem~\ref{SUCY} gives another representative for the same generator~$y_2$. Namely, the only partition of $n=3$ which belongs to $\widehat P(n)$ is $(1,1,1)$. The corresponding Calabi--Yau surface is $N_{(1,1,1)}\subset\C P^1\times\C P^1\times\C P^1$. We have 
\[
  c_1(\C P^1\times\C P^1\times\C P^1)=2u_1+2u_2+2u_3,
\]
so $N_{(1,1,1)}$ is a surface of multidegree $(2,2,2)$ in $\C P^1\times\C P^1\times\C P^1$. Lemma~\ref{snumber} gives $s_3(N_{(1,1,1)})=-\alpha(1,1,1)=-48$, so $N_{(1,1,1)}$ also represents~$y_2$.

On the other hand, the additive generator $-y_2\in\varOmega_4^{SU}$ cannot be represented by a compact complex surface. This is proved in~\cite[Theorem~3.2.5]{mosl16} by analysing the classification results on complex surfaces. It is easy to see that a complex surface $S$ with $H^1(S;\Z)=0$ (which holds for Calabi--Yau surfaces arising from toric Fano varieties) cannot represent~$-y_2$. Indeed, such $S$ has the Euler characteristic $c_2(S)=\chi(S)\ge2$, while $s_2(-y_2)=48=-2c_2(-y_2)$, so $c_2(-y_2)=-24$ is negative.
\end{example}

\begin{example}
The $6$-dimensional sphere $S^6$ has a $T^2$-invariant almost complex structure arising from its identification with the homogeneous space $G_{2}/SU(3)$ of the exceptional Lie group $G_2$, see~\cite[\S13]{bo-hi58}. Therefore, $S^6$ is an $SU$-manifold with $s_{3}[S^6]=3c_{3}[S^6]=6$. Hence, the $SU$-bordism class $[S^6]$ can be taken as~$y_{3}$.
\end{example}

\begin{example}
Here we show that the generator $-y_4\in\varOmega_8^{SU}$ can be represented by the Grassmannian $\mathit{Gr}_2(\C^4)$ of $2$-planes in $\C^4$ with an amended stably complex structure. 

Let $\gamma$ be the tautological $2$-plane bundle on~$\mathit{Gr}_2(\C^4)$, and $\gamma^\perp$ the orthogonal $2$-plane bundle. Then we have
$\mathcal T \mathit{Gr}_2(\C^4)\cong\Hom(\gamma,\gamma^\perp)$ and
\[
  \mathcal T \mathit{Gr}_2(\C^4)\oplus\Hom(\gamma,\gamma)\cong\Hom(\gamma,  
  \gamma^\perp\oplus\gamma)\cong\Hom(\gamma,\underline{\C}^4)\cong 
  \overline\gamma\oplus\overline\gamma\oplus\overline\gamma\oplus\overline\gamma. 
\]
The standard complex structure on $\mathit{Gr}_2(\C^4)$ is therefore given by the stable bundle isomorphism
\[
  \mathcal T \mathit{Gr}_2(\C^4)\cong 4\overline\gamma - \overline\gamma\gamma,
\]
where we denote $4\overline\gamma=
\overline\gamma\oplus\overline\gamma\oplus\overline\gamma\oplus\overline\gamma$ and
$\overline\gamma\gamma=\overline\gamma\otimes\gamma=\Hom(\gamma,\gamma)$.
We change the stable complex structure to the following:
\[
  \mathcal T \mathit{Gr}_2(\C^4)\cong 2\overline\gamma + 2\gamma -  
  \overline\gamma\gamma,
\]
and denote the resulting stably complex manifold by $\widetilde{\mathit{Gr}}_2(\C^4)$. Note that $c_1(\widetilde{\mathit{Gr}}_2(\C^4))=0$, so $\widetilde{\mathit{Gr}}_2(\C^4)$ is an $SU$-manifold. It has the same cohomology ring as the Grassmannian,
\[
  H^*(\mathit{Gr}_2(\C^4))\cong\mathbb Z[c_1, c_2]/(c_1^3 = 2 c_1 c_2, \, c_2^2 = c_1^2 c_2),
\]
where $c_i = c_i (\gamma)$. The top-degree cohomology $H^8(\mathit{Gr}_2(\C^4))\cong\Z$ is generated by $c_1^2 c_2$.

Now we calculate $s_4(\widetilde{\mathit{Gr}}_2(\C^4))=2s_4(\overline\gamma)+2s_4(\gamma)-s_4(\overline\gamma\gamma)$. We have
\[  
  s_4 = c_1^4 - 4c_1^2 c_2 + 4 c_1 c_3 + 2c^2_2 - 4c_4,
\]
so that
\[
  s_4(\overline\gamma) = s_4(\gamma) = c_1^4-4c_1^2 c_2 +2c_2^2= 2 c_1^2 c_2 - 4c_1^2 c_2 
  + 2 c_1^2 c_2 = 0.
\]
It remains to calculate $s_4(\overline\gamma\otimes\gamma)$. Using the splitting principle we write $\gamma=\eta_1+\eta_2$ for line bundles $\eta_1,\eta_2$ and calculate
\begin{multline*}
  c(\overline\gamma\gamma) = c((\overline {\eta}_1 + \overline{\eta}_2)( \eta_1 + \eta_2)) 
  = c(\overline{\eta}_1 \eta_2 + \overline{\eta}_2 \eta_1) = 
  c(\overline{\eta}_1 \eta_2)c(\overline{\eta_2} \eta_1) \\= 
  (1 - c_1(\eta_1) + c_1(\eta_2))(1- c_1(\eta_2) + c_1(\eta_1))  
  = 1 - c_1(\eta_1)^2 - c_1(\eta_2)^2 + 2 c_1(\eta_1)c_1(\eta_2) \\= 
  1 -(c_1(\eta_1)+c_1(\eta_2))^2+4c_1(\eta_1)c_1(\eta_2) = 1 - c_1(\gamma)^2 + 4 c_2(\gamma).
\end{multline*}
Hence, $c_1(\overline\gamma\gamma) =c_3(\overline\gamma\gamma) =c_4(\overline\gamma\gamma) =0$, and
\[
  s_4(\overline\gamma\gamma) = 2 c_2 (\overline\gamma\gamma)^2 = 2 (4 c_2 - c_1^2)^2 = 
  2(16 c_2^2 - 8 c_1^2 c_2 + c_1^4) = 20 c_1^2 c_2.
\]
It follows that $s_4[\widetilde{\mathit{Gr}}_2(\C^4)]=-20$, and $[\widetilde{\mathit{Gr}}_2(\C^4)]=-y_4\in\varOmega^{SU}_8$.
\end{example}

\begin{example}
Theorem~\ref{SUCY} gives the following representatives for the generators $y_3\in\varOmega_6^{SU}$ and $y_4\in\varOmega_8^{SU}$:
\[
  y_3=15 N_{(2,2)}-19 N_{(1,1,1,1)},\quad y_4=56 N_{(1,1,3)}-59 N_{(1,2,2)}.
\]
\end{example}

Unlike the situation in dimension~$2$, both $y_3$ and $-y_3$ can be represented by Calabi--Yau manifolds. The same holds in complex dimension~$4$, as shown by the next theorem.

\begin{theorem}\label{ldCYSUthm}
The following statements hold.
\begin{itemize}
\item[(a)] In complex dimension 2, the class $-y_{2}\in\OSU_{4}$ can be represented by a Calabi--Yau surface~$M$. 
One can take as $M$ any $K3$-surface different from a torus; it has Euler characteristic $\chi(M)=24$ and
$$
  h^{1,1}(M)=20.
$$
The class $y_2\in\OSU_4$ cannot be represented by a Calabi--Yau surface.

\smallskip

\item[(b)] In complex dimension 3, both $SU$-bordism classes $y_3$ and $-y_3$ can be represented by Calabi--Yau 3-folds.
These 3-folds $M$ can be obtained using Batyrev's construction from Fano toric varieties over reflexive $4$-polytopes. Such $M$ represents  $y_{3}\in\OSU_{6}$ if $\chi(M)=2$ or, equivalently,
$$
  h^{1,1}(M)-h^{2,1}(M)=1.
$$
Similarly, $M$ represents $-y_{3}\in\OSU_{6}$ if $\chi(M)=-2$ or, equivalently,
$$
  h^{1,1}(M)-h^{2,1}(M)=-1.
$$

\smallskip

\item[(c)] In complex dimension 4, both $SU$-bordism classes $y_4$ and $-y_4$ can be represented by Calabi--Yau 4-folds.
These 4-folds $M$ can be obtained using Batyrev's construction from Fano toric varieties over reflexive $5$-polytopes. Such $M$ represents  $y_{4}\in\OSU_{8}$ if $\chi(M)=282$ or, equivalently,
$$
  h^{1,1}(M)-h^{2,1}(M)+h^{3,1}(M)=39.
$$
Similarly, $M$ represents $-y_{4}\in\OSU_{8}$ if $\chi(M)=294$ or, equivalently,
$$
h^{1,1}(M)-h^{2,1}(M)+h^{3,1}(M)=41.
$$
\end{itemize}
\end{theorem}

\begin{proof}
We denote both the Chern characteristic classes and characteristic numbers of $M$ by~$c_i$ throughout this proof, denote the Hodge numbers by~$h^{i,j}$ and denote the (real) Betti numbers by~$b^i$, for $i=0,\ldots,\dim_\C M$. For a K\"ahler $n$-manifold $M$ we have $h^{p,q}=h^{q,p}$ (Hodge duality), $b^i=\sum_{p+q=i}h^{p,q}$ and $\chi(M)=\sum_{i=0}^{2n}(-1)^ib^i=\sum_{p,q=0}^n(-1)^{p+q}h^{p,q}$. Furthermore, a Calabi--Yau manifold $M$ obtained from Batyrev's construction is projective algebraic, so it satisfies $h^{p,q}=h^{n-p,n-q}$ (Serre duality). Finally, such a Calabi--Yau manifold $M$ has full $SU(n)$ holonomy and therefore $h^{n,0}=1$ and $h^{i,0}=0$ for $0<i<n$ (see~\cite[Theorem~4.1.9]{baty94}).

Statement (a) is a summary of Example~\ref{CYsurface}.

\smallskip

We prove (b).
For the generator $y_3\in\OSU_6$ we have $6=s_3(y_3)=3c_3(y_3)$, so the Euler characteristic of a complex $SU$-manifold $M$ representing $y_3$ satisfies
$\chi(N)=c_3(N)=2$. For a Calabi--Yau 3-fold $M$ obtained from Batyrev's construction we have 
\[
  b^1=2h^{1,0}=0,\quad b^2=2h^{2,0}+h^{1,1}=h^{1,1},
  \quad b^3=2h^{3,0}+2h^{2,1}=2+2h^{2,1},
\]
and
\[
  \chi(M)=2b^0-2b^1+2b^2-b^3=2(h^{1,1}-h^{2,1}).
\]
It follows that $M$ represents $y_3$ if and only if $h^{1,1}-h^{2,1}=1$. Similarly, $M$ represents $-y_3$ if and only if $h^{1,1}-h^{2,1}=-1$. 

The fact that such $M$ exist follows by analysing the database~\cite{kr-sk} (see also~\cite{a-g-h-j-n15}) of reflexive polytopes and the Calabi--Yau hypersurfaces in their corresponding toric Fano varieties. This database contains the full list of 473,800,776 reflexive polytopes in dimension $4$, and the list of Hodge numbers of the corresponding Calabi--Yau $3$-folds. From there one can see that for each $h^{1,1}$ satisfying $16\le h^{1,1}\le 90$ there exists a reflexive $4$-polytope with the corresponding Calabi--Yau $3$-fold satisfying $h^{1,1}-h^{2,1}=1$, and if $h^{1,1}$ is not within this range, then there is no Calabi--Yau $3$-fold with $h^{1,1}-h^{2,1}=1$ coming from a toric Fano variety. In the case of the identity $h^{1,1}-h^{2,1}=-1$, the possible range is $15\le h^{1,1}\le 89$. 

We note also that the Calabi--Yau 3-folds $M$ and $M^*$ representing $y_3$ and $-y_3$ can be chosen to be mirror dual in the sense of~\cite{baty94}, that is, to satisfy the condition $h^{1,1}(M)=h^{2,1}(M^*)$ and $h^{2,1}(M)=h^{1,1}(M^*)$.

\smallskip

We prove (c). It is convenient to use the partial Euler characteristics $\chi_{k}=\sum_{i=0}^{4}(-1)^{i}h^{i,k}$, for $0\le k\le 4$. In particular, $\chi_0$ is the Todd genus of a complex manifold.
For a Calabi--Yau 4-fold $M$ obtained from Batyrev's construction we have
\begin{align*}
  \chi_0&=h^{0,0}-h^{1,0}+h^{2,0}-h^{3,0}+h^{4,0}=2;\\
  \chi_1&=h^{0,1}-h^{1,1}+h^{2,1}-h^{3,1}+h^{4,1}=-h^{1,1}+h^{2,1}-h^{3,1};\\
  \chi_2&=h^{0,2}-h^{1,2}+h^{2,2}-h^{3,2}+h^{4,2}=-2h^{2,1}+h^{2,2}.
\end{align*}
Therefore,
\begin{equation}\label{chiMy4}
  \chi(M)=\chi_0-\chi_1+\chi_2-\chi_3+\chi_4=2\chi_0-2\chi_1+\chi_2=
  2(2+h^{1,1}-2h^{2,1}+h^{3,1})+h^{2,2}.
\end{equation}
On the other hand, the Hirzebruch--Riemann--Roch theorem~\cite[Theorem~21.1.1]{hirz66} implies the following identities in terms of the Chern numbers:
$$
  720\chi_{0}=-c_4+3c_2^2,\quad 180\chi_{1}=-31c_4+3c_2^2,  
  \quad 120\chi_{2}=79c_{4}+3c_2^2.
$$

For the generator $y_4\in\OSU_8$, we have $s_4=2c_2^2-4c_4=20$. Since $\chi_0=2$, the identity $2c_2^2-4c_4=20$ is equivalent to any of the following:
\[
  \chi(M)=c_4=282\quad\text{or}\quad -\chi_{1}=h^{1,1}(M)-h^{2,1}(M)+h^{3,1}(M)=39,
\]
as claimed.

Similarly, for $-y_4$, the condition $s_4=2c_2^2-4c_4=-20$ is equivalent to
\[
  \chi(M)=c_4=294\quad\text{or}\quad -\chi_{1}=h^{1,1}(M)-h^{2,1}(M)+h^{3,1}(M)=41.
\]

The existence of $M$ follows by analysing the database~\cite{kr-sk} as in~(b). In particular, there exist a Calabi--Yau fourfold with $h^{1,1}=16$, $h^{2,1}=30$,
$h^{3,1}=53$, representing $y_4$, and a Calabi--Yau fourfold with
$h^{1,1}=17$, $h^{2,1}=45$,
$h^{3,1}=69$, representing $-y_4$.
\end{proof}

The class $-y_4\in\OSU_8$ can also be represented by a Calabi--Yau manifold $Z_S$ of Borcea--Voisin type, constructed in~\cite{c-g-m18} as a crepant resolution of the
quotient of a hyperk\"ahler manifold by a non-symplectic involution. This follows by comparing the formula in Theorem~\ref{ldCYSUthm}~(c) with the calculation of the Hodge numbers in~\cite[\S5.2]{c-g-m18}. 

The generator $\frac14y_2^2=x_1^4=w_4$ of  the group $\varOmega^{SU}_8=\Z\langle {\textstyle\frac14}y^2_2,y_4\rangle$ cannot be represented by a Calabi--Yau fourfold with full $SU(4)$ holonomy. Indeed, as noted at the end of Section~\ref{ringosusec},
\[
  {\textstyle\frac14}y^2_2=x_1^4=\bigl(9[\C P^1]\times[\C P^1]-8[\C P^2]\bigr)\times
  \bigl(9[\C P^1]\times[\C P^1]-8[\C P^2]\bigr),
\]
so the Todd genus of~${\textstyle\frac14}y^2_2$ is~$1$. On the other hand, a Calabi--Yau fourfold with full $SU(4)$ holonomy has $h^{0,1}=h^{0,2}=h^{0,3}=0$, so its Todd genus is equal to $h^{0,0}+h^{0,4}=2$.

\end{document}